\documentclass[a4paper,11pt,reqno,makeidx]{customart}

\usepackage[utf8]{inputenc}
\usepackage[charter]{mathdesign}
\usepackage[T1]{fontenc}
\usepackage{graphicx}
\usepackage{microtype}
\usepackage{amsmath}
\usepackage{xstring,nicefrac}
\usepackage{refcount}
\usepackage{lastpage}
\usepackage[all]{xy}
\usepackage[margin=1in]{geometry}
\usepackage{multirow}
\usepackage{changepage}
\usepackage{setspace}
\usepackage{enumitem}
\usepackage{todonotes}
\usepackage{hyperref}
\usepackage{tikz}
\usepackage{caption}
\usepackage{url}

\newtheorem{theorem}{Theorem}[section]
\newtheorem{proposition}[theorem]{Proposition} 
\newtheorem{lemma}[theorem]{Lemma}
\newtheorem{corollary}[theorem]{Corollary}

\newtheorem*{claim*}{Claim}

\theoremstyle{definition}
\newtheorem{definition}[theorem]{Definition}

\newtheorem{question}[theorem]{Question}

\newtheorem{notation}[theorem]{Notation}
\newtheorem{widget}[theorem]{Widget}

\theoremstyle{remark}
\newtheorem{remark}[theorem]{Remark}

\newtheoremstyle{noparens}%
  {}{}%
{}{}%
{\bfseries}{.}%
{ }%
{\thmname{#1}\thmnumber{ #2}\thmnote{ #3}}
\theoremstyle{noparens}
\newtheorem*{question*}{Question}

\DeclareMathOperator{\dom}{\textup{dom}}

\newcommand{\andd}{\wedge}
\newcommand{\orr}{\vee}
\newcommand{\la}{\langle}
\newcommand{\ra}{\rangle}
\newcommand{\da}{\!\downarrow}
\newcommand{\ua}{\!\uparrow}
\newcommand{\imp}{\rightarrow}

\newcommand{\biimp}{\leftrightarrow}
\newcommand{\Nb}{\mathbb{N}}

\newcommand{\Qb}{\mathbb{Q}}

\newcommand{\Scal}{\mathcal{S}}

\newcommand{\smf}{\smallfrown}

\newcommand{\str}{2^{<\Nb}}

\newcommand{\llb}{\llbracket}
\newcommand{\rrb}{\rrbracket}
\newcommand{\cyl}[1]{\llb {#1} \rrb}
\renewcommand{\setminus}{\smallsetminus}

\newcommand{\bad}{\mathrm{Bad}}

\DeclareMathOperator{\atoms}{\mathrm{atoms}}

\DeclareMathOperator{\dnrf}{\mathrm{DNR}}
\DeclareMathOperator{\rng}{\mathrm{ran}}

\setlist*[enumerate]{label=\arabic*.,leftmargin=1cm}
\setlist*[itemize]{leftmargin=1cm}

\DeclareMathOperator{\true}{\textup{\texttt{T}}}
\DeclareMathOperator{\false}{\textup{\texttt{F}}}

\newcommand{\set}[1]{\left\{ #1 \right\}}
\newcommand{\card}[1]{\left| #1 \right|}

\newcommand{\tuple}[1]{\left\langle #1 \right\rangle}

\newcommand{\s}[1]{\ensuremath{{\sf #1}}}
\newcommand{\dnrs}[1]{#1\mbox{-}\s{DNR}}
\newcommand{\ran}[1]{#1\mbox{-}\s{RAN}}
\newcommand{\rwkls}[1]{#1\mbox{-}\s{RWKL}}
\newcommand{\wwkls}[1]{#1\mbox{-}\s{WWKL}}
\newcommand{\rwwkls}[1]{#1\mbox{-}\s{RWWKL}}

\DeclareMathOperator{\rca}{\s{RCA}_0}
\DeclareMathOperator{\aca}{\s{ACA}_0}
\DeclareMathOperator{\wklz}{\s{WKL}_0}
\DeclareMathOperator{\kl}{\s{KL}}
\DeclareMathOperator{\wkl}{\s{WKL}}
\DeclareMathOperator{\wwkl}{\s{WWKL}}
\DeclareMathOperator{\dnr}{\s{DNR}}

\DeclareMathOperator{\bsig}{\s{B}\Sigma}
\DeclareMathOperator{\bpi}{\s{B}\Pi}
\DeclareMathOperator{\isig}{\s{I}\Sigma}
\DeclareMathOperator{\ipi}{\s{I}\Pi}
\DeclareMathOperator{\bst}{\bsig^0_2}
\DeclareMathOperator{\iso}{\s{I}\Sigma^0_1}

\DeclareMathOperator{\rkl}{\s{RKL}}
\DeclareMathOperator{\rwkl}{\s{RWKL}}
\DeclareMathOperator{\rwwkl}{\s{RWWKL}}
\DeclareMathOperator{\lrwkl}{\s{LRWKL}}
\DeclareMathOperator{\lrwwkl}{\s{LRWWKL}}
\DeclareMathOperator{\rt}{\s{RT}}
\DeclareMathOperator{\srt}{\s{SRT}}
\DeclareMathOperator{\crt}{\s{CRT}}
\DeclareMathOperator{\rrt}{\s{RRT}}

\DeclareMathOperator{\cads}{\s{CADS}}
\DeclareMathOperator{\cac}{\s{CAC}}
\DeclareMathOperator{\scac}{\s{SCAC}}

\DeclareMathOperator{\sat}{\s{SAT}}
\DeclareMathOperator{\rsat}{\s{RSAT}}
\DeclareMathOperator{\lrsat}{\s{LRSAT}}
\DeclareMathOperator{\rcolor}{\s{RCOLOR}}
\DeclareMathOperator{\lrcolor}{\s{LRCOLOR}}
\DeclareMathOperator{\coh}{\s{COH}}
\DeclareMathOperator{\pizog}{\Pi^0_1\s{G}}
\DeclareMathOperator{\fip}{\s{FIP}}

\DeclareMathOperator{\emo}{\s{EM}}
\DeclareMathOperator{\semo}{\s{SEM}}
\DeclareMathOperator{\ndtip}{\s{\bar{D}_2IP}}

\DeclareMathOperator{\opt}{\s{OPT}}
\DeclareMathOperator{\amt}{\s{AMT}}

\DeclareMathOperator{\atr}{\s{ATR}_0}
\DeclareMathOperator{\pica}{\mathrm{\Pi}^1_1-\s{CA}_0}

\def§#1§{\text{#1}}%

\hypersetup{
	pdfborder={0 0 0}
}

\usepackage{amsmath, amsthm}

\usepackage{xcolor}	
\usepackage{soul}

\definecolor{lightblue}{rgb}{.60,.60,1}

\definecolor{lightred}{rgb}{1,.60,.60}

\title{On the logical strengths of partial solutions to mathematical problems}

\author{Laurent Bienvenu}
\address{
LIRMM, CNRS \& Universit\'e de Montpellier, 161 rue Ada, 34095 Montpellier Cedex 5
}
\email{laurent.bienvenu@computability.fr}
\urladdr{\url{http://www.lirmm.fr/~bienvenu}}

\author{Ludovic Patey}
\address{
Department of Mathematics, University of California, Berkeley, CA 94720, USA}
\email{ludovic.patey@computability.fr}
\urladdr{\url{http://www.ludovicpatey.com}}

\author{Paul Shafer}
\address{Department of Mathematics,
Ghent University,
Krijgslaan 281, S22,
9000 Ghent,
Belgium}
\email{paul.shafer@ugent.be}
\urladdr{\url{http://cage.ugent.be/~pshafer/}}

\thanks{Laurent Bienvenu and Ludovic Patey are funded by the John Templeton Foundation (`Structure and Randomness in the Theory of Computation' project). The opinions expressed in this publication are those of the authors and do not necessarily reflect the views of the John Templeton Foundation.  Paul Shafer is an FWO Pegasus Long Postdoctoral Fellow.  He was also supported by the Fondation Sciences Math\'ematiques de Paris.}

\date{\today}

\begin{document}

\begin{abstract}
We use the framework of reverse mathematics to address the question of, given a mathematical problem, whether or not it is easier to find an infinite partial solution than it is to find a complete solution.  Following Flood~\cite{Flood2012}, we say that a \emph{Ramsey-type} variant of a problem is the problem with the same instances but whose solutions are the infinite partial solutions to the original problem.  We study Ramsey-type variants of problems related to K\"onig's lemma, such as restrictions of K\"onig's lemma, Boolean satisfiability problems, and graph coloring problems.  We find that sometimes the Ramsey-type variant of a problem is strictly easier than the original problem (as Flood showed with weak K\"onig's lemma) and that sometimes the Ramsey-type variant of a problem is equivalent to the original problem.  We show that the Ramsey-type variant of weak K\"onig's lemma is \emph{robust} in the sense of Montalb\'an~\cite{Montalban2011Open}:  it is equivalent to several perturbations.  We also clarify the relationship between Ramsey-type weak K\"onig's lemma and algorithmic randomness by showing that Ramsey-type weak weak K\"onig's lemma is equivalent to the problem of finding diagonally non-recursive functions and that these problems are strictly easier than Ramsey-type weak K\"onig's lemma.  This answers a question of Flood.
\end{abstract}


\maketitle

\section{Introduction}

This work presents a detailed study of the question \emph{given some mathematical problem, is it easier to find an infinite partial solution than it is to find a complete solution?}\ that was implicitly raised by Flood's work in~\cite{Flood2012}.  By `mathematical problem,' we simply mean any theorem from ordinary mathematics that can be easily formulated in the language of instances and solutions in the sense illustrated by the key example of K\"onig's lemma.  K\"onig's lemma states that every infinite, finitely branching tree has an infinite path.  The corresponding problem is thus that of finding an infinite path through a given infinite, finitely branching tree.  The problem's instances are the infinite, finitely branching trees $T$, and the solutions to a given instance $T$ are the infinite paths through $T$.

Formally, we consider $\Pi^1_2$ statements of the form
\begin{align*}
\forall X(\varphi(X) \imp \exists Y \psi(X,Y)).
\end{align*}
Every such statement corresponds to a problem whose instances are the sets $X$ such that $\varphi(X)$ and whose solutions to a given instance $X$ are the sets $Y$ such that $\psi(X,Y)$.  In the example of K\"onig's lemma, $\varphi(X)$ expresses that $X$ is an infinite, finitely branching tree, and $\psi(X,Y)$ expresses that $Y$ is an infinite path through $X$.

The problems we consider come with natural notions of infinite partial solutions.  Again, consider K\"onig's lemma, where we specify that an infinite, finitely branching tree means an infinite, finitely branching subtree of $\Nb^{<\Nb}$.  For such a tree $T$, a path through $T$, which we think of as a \emph{complete solution} to the instance $T$, is a function $f \colon \Nb \imp \Nb$ such that $\forall n(\la f(0), f(1), \dots, f(n-1) \ra \in T)$.  An infinite \emph{partial solution} to the instance $T$ is then a function $g \colon X \imp \Nb$ for an infinite $X \subseteq \Nb$ such that there is a function $f \colon \Nb \imp \Nb$ that extends $g$ and is a path through $T$.  Following Flood~\cite{Flood2012}, we call the variant of a problem in which we ask not for complete solutions but for infinite partial solutions the \emph{Ramsey-type} variant of the problem.  Thus, for example, \emph{Ramsey-type K\"onig's lemma} is the problem of producing an infinite partial path (in the sense described above) through an infinite, finitely-branching tree.  The label `Ramsey-type' comes from an analogy with the infinite versions of Ramsey's theorem.  Any infinite subset of an infinite homogenous set for some coloring is also an infinite homogeneous set for that coloring.  The Ramsey-type variant of a problem has this same flavor:  an infinite piece of a partial solution to some instance of the problem is also a partial solution to that same instance.

Thus given a mathematical problem, we ask whether or not it can be solved using its Ramsey-type variant.  If the answer is positive, then finding partial solutions to the problem is just as hard as finding complete solutions.  If the answer is negative, then it is easier to find partial solutions than it is to find complete solutions.  \emph{Reverse mathematics}, a foundational program whose aim is to classify the theorems of ordinary (i.e., non-set-theoretic) mathematics according to their provability strengths, provides an appropriate framework in which to analyze such questions.  In reverse mathematics, theorems are formalized in the language of second-order arithmetic (which even suffices for theorems concerning the structure of the real line or analysis on complete separable metric spaces), and the implications among them are studied over a base theory called $\rca$.  Roughly speaking, the theorems provable in $\rca$ are those that are computable in the sense illustrated by the example of the intermediate value theorem.  Given a continuous real-valued function which is negative at $0$ and positive at $1$, one can compute an $x \in (0,1)$ such that $f(x)=0$ essentially by using the usual interval-halving procedure.  This argument can be formalized to a proof of the intermediate value theorem in $\rca$ (see~\cite{Simpson2009} Theorem~II.6.6).

Implication over $\rca$ provides a natural classification of logical strength.  We think of a theorem $\varphi$ as being at least as strong as a theorem $\psi$ if $\varphi \imp \psi$ can be proved in $\rca$.  Similarly, we think of $\varphi$ and $\psi$ as having equivalent strength if $\varphi \biimp \psi$ can be proved in $\rca$.  Thus we may, for example, formalize the question of whether or not it is easier to find partial paths through infinite, finitely branching trees than it is to find complete paths by asking whether or not the statement ``for every infinite, finitely branching tree there exists an infinite partial path implies K\"onig's lemma'' can be proved in $\rca$.  Flood~\cite{Flood2012} was the first to consider such questions and he showed (among other results) that the Ramsey-type variant of weak K\"onig's lemma (which is K\"onig's lemma restricted to infinite, binary branching trees) is indeed easier than weak K\"onig's lemma.  In contrast, we show that the forgoing example of K\"onig's lemma for arbitrary infinite, finitely branching trees is equivalent to its Ramsey-type variant (Theorem~\ref{thm-RKLisACA} below).  Thus for some problems it is easier to find infinite partial solutions and for other problems it is not. 

Much of the present work is dedicated to understanding the relationships among Flood's Ramsey-type variant of weak K\"onig's lemma (henceforth `$\rwkl$'), Ramsey-type variants of other problems, and problems that are well-studied in reverse mathematics.  For example, Flood proved that $\rwkl$ is strictly weaker than weak K\"onig's lemma and at least as strong as $\dnr$ (an important statement defined in Section~\ref{sec:DNR}), but he left as an open question whether or not $\rwkl$ is strictly stronger than $\dnr$.  We answer Flood's question by showing that $\rwkl$ is indeed strictly stronger than $\dnr$ (Corollary~\ref{cor:DNRNoImpRWKL} below), and we also show that $\dnr$ is equivalent to the Ramsey-type variant of weak weak K\"onig's lemma (which is K\"onig's lemma restricted to binary branching trees of positive measure; Theorem~\ref{thm:dnr-rwwkl} below).\footnote{These results have been independently proven by Flood and Towsner~\cite{Flood2014Separating}.}  Thus $\rwkl$ is distinct from every theorem previously studied in the context of reverse mathematics.  This raises the question of whether $\rwkl$ is a sort of logical artifact or whether $\rwkl$ characterizes the logical strength of a fundamental mathematical idea.  We propose that $\rwkl$ is indeed fundamental, in no small part because the basic question that inspires $\rwkl$, that is, the question of whether or not it is easier to find an infinite partial solution to a problem than to find a complete solution, is so natural.  In order to provide further support for $\rwkl$, we prove a number of theorems which, together, suggest that $\rwkl$ is \emph{robust} in the informal sense proposed by Montalb\'an~\cite{Montalban2011Open}.  Theorem~\ref{thm:lrwkl-rwkl} shows that $\rwkl$ is equivalent to several small perturbations.  Much more significantly, in Section~\ref{sec-RSAT} and Section~\ref{sec-RCOLOR} we show that $\rwkl$ is equivalent to several quite large perturbations.  In these sections, we consider statements that are equivalent to weak K\"onig's lemma (compactness for propositional logic in Section~\ref{sec-RSAT} and graph coloring in Section~\ref{sec-RCOLOR}) and show that their corresponding Ramsey-type variants are equivalent to $\rwkl$.

$\rwkl$ is also of significant technical interest because it provides a sufficient amount of compactness for many separation results concerning Ramsey-type statements.  For example, Seetapun's theorem~\cite{seetapun1995strength} (separating $\rt^2_2$ from $\aca$), Wang's separation of the free set and thin set theorems from $\aca$~\cite{wang2013some}, and various recent separations of Patey~\cite{Patey2015strength,Patey2015weakness} can be streamlined by using models of $\rwkl$ in place of models of $\wkl$.   Many computability-theoretic properties are preserved by both $\rt^2_2$ and $\wkl$, such as cone avoidance~\cite{seetapun1995strength}, hyperimmunity~\cite{Patey2015Iterativea}, and fairness~\cite{Patey2015strength}.  Explicit use of models of $\rwkl$ is helpful when proving that $\rt^2_2$ preserves a property which is \emph{not} preserved by $\wkl$, such as constant-bound-enumeration avoidance~\cite{coneavoidclosed}.  In particular, Liu's theorems~\cite{Liu,coneavoidclosed}, that $\rt^2_2$ does not imply $\wkl$ or $\wwkl$, can be simplified by making explicit use of models of $\rwkl$. In this sense, using of models of $\rwkl$ rather than of $\wkl$ is more general because it facilitates proving preservations of more computability-theoretic properties.

The paper is organized as follows.  In the next section, we present the necessary reverse mathematics background.  In Section~\ref{sec-rtkl-variants}, we study several Ramsey-type variants of full, bounded, weak, and weak weak K\"onig's lemma.  The remainder of the paper focuses on Ramsey-type variants of theorems equivalent to weak K\"onig's lemma.  In Section~\ref{sec-RSAT}, we study Ramsey-type variants of the compactness theorem for propositional logic.  In Section~\ref{sec-RCOLOR}, we study Ramsey-type variants of graph coloring theorems.  In Section~\ref{sec-nonImp}, we prove several non-implications concerning the Ramsey-type theorems, including that $\dnr$ does not imply $\rwkl$.

\subsection{Basic notation}
We follow the standard notation from computability theory.  $(\Phi_e)_{e \in \Nb}$ is an effective list of all partial recursive functions.  $W_e = \dom(\Phi_e)$ is the $e\textsuperscript{th}$ r.e.\ set. These relativize to any oracle~$X$, and we denote the corresponding lists by $(\Phi^X_e)_{e \in \Nb}$ and $(W^X_e)_{e \in \Nb}$.

Identify each $k \in \Nb$ with the set $\{0,1,\dots,k-1\}$.  For $k \in \Nb \cup \{\Nb\}$ and $s \in \Nb$, $k^s$ is the set of strings of length $s$ over $k$, $k^{<s}$ is the set of strings of length $<s$ over $k$, $k^{<\Nb}$ is the set of finite strings over $k$, and $k^{\Nb}$ is the set of infinite strings over $k$.  The length of a finite string $\sigma$ is denoted $|\sigma|$.  For $i \in \Nb$ and $\sigma$ a finite or infinite string, $\sigma(i)$ is the $(i+1)\textsuperscript{th}$ value of $\sigma$.  For finite or infinite strings $\sigma$ and $\tau$, $\sigma$ is a prefix of $\tau$ (written $\sigma \preceq \tau$) if $\dom(\sigma) \subseteq \dom(\tau)$ and $(\forall i \in \dom(\sigma))(\sigma(i) = \tau(i))$.  For an $n \in \Nb$ and a string (finite or infinite) $\sigma$ of length $\geq n$, $\sigma \restriction n = \la \sigma(0), \sigma(1), \dots, \sigma(n-1) \ra$ is the initial segment of $\sigma$ of length $n$. 

A tree is a set $T \subseteq \Nb^{<\Nb}$ such that $\forall \sigma \forall \tau(\sigma \in T \andd \tau \preceq \sigma \imp \tau \in T)$. If $T$ is a tree and $s \in \Nb$, then $T^s$ is the set of strings in $T$ of length $s$.  An $f \in \Nb^{\Nb}$ is a path through a tree $T$ if $(\forall n \in \Nb)(f \restriction n \in T)$.  The set of paths through $T$ is denoted $[T]$.

For $k \in \Nb \cup \{\Nb\}$, the space $k^{\Nb}$ is topologized by viewing it as $\prod_{i \in \Nb}k$, giving each copy of $k$ the discrete topology, and giving the product the product topology.  Basic open sets, also called cylinders, are sets of the form $\cyl{\sigma} = \{f \in k^{\Nb} : f \succeq \sigma\}$ for $\sigma \in k^{<\Nb}$.  Open sets are of the form $\bigcup_{\sigma \in W}\cyl{\sigma}$ for $W \subseteq k^{<\Nb}$.  If the set $W$ is an r.e.\ subset of $k^{<\Nb}$, then $\bigcup_{\sigma \in W}\cyl{\sigma}$ is said to be \emph{r.e.} (or \emph{effectively}) \emph{open}.  We identify the space $2^{\Nb}$ of infinite binary strings with $\mathcal{P}(\Nb)$ by equating each subset of $\Nb$ with its characteristic string as usual.  $2^{\Nb}$ is compact, and its clopen sets are exactly the finite unions of cylinders.  The \emph{uniform} (or \emph{Lebesgue}) \emph{measure} $\mu$ on $2^{\Nb}$ is the Borel probability measure for which $(\forall \sigma \in 2^{<\Nb})(\mu(\cyl{\sigma}) = 2^{-|\sigma|})$.

It is a convention, when working in second-order arithmetic, to use the symbol `$\omega$' to refer to the standard natural numbers and to use the symbol `$\Nb$' to refer to the first-order part of a possibly non-standard model of some fragment of arithmetic.  We follow this convention.  For example, the definitions above use `$\Nb$' because they are intended to be interpreted in possibly non-standard models.  We use `$\omega$' when we explicitly build a structure whose first-order part is standard.

\section{Reverse mathematics background}\label{sec-rm-background}

Reverse mathematics is a foundational program, introduced by Friedman~\cite{Friedman} and developed by Friedman and by Simpson, whose goal is to classify the theorems of ordinary mathematics according to their provability strengths.  Simpson's book~\cite{Simpson2009} is the standard reference.  A truly remarkable phenomenon is that five equivalence classes, called the Big Five (in order of increasing strength:  $\rca$, $\wklz$, $\aca$, $\atr$, and $\pica$), emerge and classify the majority of usual theorems.  The Big Five classes also have satisfying interpretations as the ability to perform well-known computability-theoretic operations.  For example, $\rca$ corresponds to the ability to perform Turing reductions and Turing joins, whereas $\aca$ corresponds to the ability to perform Turing reductions, Turing joins, and Turing jumps.

There is, however, a notable family of theorems which are not classified by the Big Five.  These are what we call the \emph{Ramsey-type} theorems, perhaps the most famous of which is Ramsey's theorem for pairs and two colors.  Since the seminal paper of Cholak, Jockusch, and Slaman~\cite{cholak2001strength}, an abundant literature has developed surrounding the strength of Ramsey's theorem for pairs and related theorems, such as chain-antichain, ascending or descending sequence, and the Erd\H{o}s-Moser theorem (see, for example, \cite{hirschfeldt2007combinatorial} and \cite{lerman2013separating}).  These Ramsey-type theorems do not typically have nice computability-theoretic characterizations of their equivalence classes.  

We are primarily concerned with the logical relationships among combinatorial statements (specifically Ramsey-type statements) provable in the system $\aca$.  Thus we now summarize several of the subsystems of second-order arithmetic below $\aca$ and the relationships among them.

\subsection{Recursive comprehension, weak K\"onig's lemma, and arithmetical comprehension}
First we summarize the induction, bounding, and comprehension schemes and three of the most basic subsystems of second-order arithmetic.  Everything stated here is explained in full detail in~\cite{Simpson2009}.

Full second-order arithmetic consists of the \emph{basic axioms}:

\[
\begin{array}{ll}
\forall m (m+1 \neq 0) & \forall m \forall n (m \times (n+1) = (m \times n) + m)\\
\forall m \forall n(m+1 = n+1 \imp m=n) & \forall m \forall n (m < n+1 \biimp (m < n \orr m = n))\\
\forall m (m+0 = m) & \forall m \neg(m < 0)\\
\forall m \forall n(m+(n+1) = (m+n)+1) & \forall m (m \times 0 = 0) \\
\end{array}
\]

the {\em induction axiom}:
\begin{align*}
\forall X((0 \in X \andd \forall n(n \in X \imp n+1 \in X)) \imp \forall n(n \in X));
\end{align*}
and the {\em comprehension scheme}, which consists of the universal closures of all formulas of the form
\begin{align*}
\exists X \forall n(n \in X \biimp \varphi(n)),
\end{align*}
where $\varphi$ is any formula in the language of second-order arithmetic in which $X$ is not free.  We obtain subsystems of second-order arithmetic by limiting induction and comprehension to predicates of a prescribed complexity.

For each $n \in \omega$, the $\Sigma^0_n$ ($\Pi^0_n$) \emph{induction scheme}, denoted $\isig^0_n$ ($\ipi^0_n$), consists of the universal closures of all formulas of the form
\begin{align*}
[\varphi(0) \andd \forall n(\varphi(n) \imp \varphi(n+1))] \imp \forall n \varphi(n),
\end{align*}
where $\varphi$ is $\Sigma^0_n$ ($\Pi^0_n$).  The induction schemes are closely related to the bounding (also called collection) schemes.  For each $n \in \omega$, the $\Sigma^0_n$ ($\Pi^0_n$) \emph{bounding scheme}, denoted $\bsig^0_n$ ($\bpi^0_n$), consists of the universal closures of all formulas of the form
\begin{align*}
\forall a[(\forall n < a)(\exists m)\varphi(n,m) \imp \exists b(\forall n < a)(\exists m < b)\varphi(n,m)],
\end{align*}
where $\varphi$ is $\Sigma^0_n$ ($\Pi^0_n$).

The \emph{arithmetical comprehension scheme} consists of the universal closures of all formulas of the form
\begin{align*}
\exists X \forall n(n \in X \biimp \varphi(n)),
\end{align*}
where $\varphi$ is an arithmetical formula in which $X$ is not free.  A further restriction of comprehension is the $\Delta^0_1$ \emph{comprehension scheme}, which consists of the universal closures of all formulas of the form
\begin{align*}
\forall n (\varphi(n) \biimp \psi(n)) \imp \exists X \forall n(n \in X \biimp \varphi(n)),
\end{align*}
where $\varphi$ is $\Sigma^0_1$, $\psi$ is $\Pi^0_1$, and $X$ is not free in $\varphi$.

$\rca$ (for \emph{recursive comprehension axiom}) encapsulates recursive mathematics and is the usual base system used when comparing the logical strengths of statements of second-order arithmetic.  The axioms of $\rca$ are the basic axioms, $\iso$, and the $\Delta^0_1$ comprehension scheme.

$\rca$ proves sufficient number-theoretic facts to implement the codings of finite sets and sequences that are typical in computability theory.  Thus inside $\rca$, we can fix an enumeration $(\Phi_e)_{e \in \Nb}$ of the partial recursive functions.  We can also interpret the existence of the set $\Nb^{<\Nb}$ of all finite strings and give the usual definition of a tree as subset of $\Nb^{<\Nb}$ that is closed under initial segments.

Weak K\"onig's lemma ($\wkl$) is the statement ``every infinite subtree of $\str$ has an infinite path,''  and $\wklz$ is the subsystem $\rca + \wkl$.  $\wklz$ captures compactness arguments, and it is strictly stronger than $\rca$ (i.e., $\rca \nvdash \wkl$).

$\aca$ (for \emph{arithmetical comprehension axiom}) is the subsystem axiomatized by the basic axioms, the induction axiom, and the arithmetical comprehension scheme.  It can also be obtained by adding the arithmetical comprehension scheme to $\rca$.  $\aca$ is strictly stronger than $\wklz$, and all of the statements that we consider are provable in $\aca$.

\subsection{Ramsey's theorem and its consequences}
Let $S \subseteq \Nb$ and $n \in \Nb$.  $[S]^n$ denotes the set of $n$-element subsets of $S$, typically thought of as coded by the set of strictly increasing $n$-tuples over $S$. 

\begin{definition}[Ramsey's theorem]
Fix $n, k \in \Nb$ with $n,k > 0$.  A set $H$ is \emph{homogeneous} for a coloring $f : [\Nb]^n \to k$ (or \emph{$f$-homogeneous}) if there is a color $c < k$ such that $f([H]^n) = \{c\}$.  A coloring $f : [\Nb]^n \to k$ is \emph{stable} if for every $\sigma \in [\Nb]^{n-1}$ there is a color $c$ such that $(\exists m)(\forall s > m)(f(\sigma,s) = c)$.  $\rt^n_k$ is the statement ``for every coloring $f : [\Nb]^n \to k$, there is an infinite $f$-homogeneous set.''  $\srt^n_k$ is the restriction of $\rt^n_k$ to stable colorings.
\end{definition}

\begin{definition}[Cohesiveness]\label{def-COH}
Let $\vec R =  (R_i)_{i \in \Nb}$ be a sequence of subsets of $\Nb$.  A set $C \subseteq \Nb$ is called \emph{$\vec R$-cohesive} if $C$ is infinite and $\forall i(C \subseteq^* R_i \orr C \subseteq^* \overline{R_i})$, where $A \subseteq^* B$ means that $A \setminus B$ is finite.  $\coh$ is the statement ``for every sequence of sets $\vec{R}$, there is an $\vec{R}$-cohesive set.''
\end{definition}

For every fixed $n \in \omega$ with $n \geq 3$, the statement $(\forall k \geq 2)\rt^n_k$ is equivalent to $\aca$ over $\rca$.  Indeed, the statement $\rt^3_2$ is already equivalent to $\aca$ over $\rca$ (see~\cite{Simpson2009} Theorem~III.7.6).  Much work was motivated by the desire to characterize the logical strength of $\rt^2_2$.  Among many results, Cholak, Jockusch, and Slaman~\cite{cholak2001strength} (with a bug-fix in \cite{miletipartition}) showed that $\rt^2_2$ splits into $\coh$ and $\srt^2_2$ over $\rca$:  $\rca \vdash \rt^2_2 \biimp \coh \andd \srt^2_2$.  By work of Chong, Slaman, and Yang~\cite{chongmetamathematics}, $\srt^2_2$ is strictly weaker than $\rt^2_2$ over $\rca$.  By work of Hirst~\cite{Hirst1987} and Liu~\cite{Liu}, $\rt^2_2$ and $\srt^2_2$ are independent of $\wkl$ over $\rca$.

\begin{definition}[Chain-antichain]
A \emph{partial order} $P = (P, \leq_P)$ consists of a set $P \subseteq \Nb$ together with a reflexive, antisymmetric, transitive, binary relation $\leq_P$ on $P$.  A \emph{chain} in $P$ is a set $S \subseteq P$ such that $(\forall x, y \in S)(x \leq_P y \orr y \leq_P x)$.  An \emph{antichain} in $P$ is a set $S \subseteq P$ such that $(\forall x, y \in S)(x \neq y \imp x |_P y)$ (where $x |_P y$ means that $x \nleq_P y \andd y \nleq_P x$).  A partial order $(P, \leq_P)$ is \emph{stable} if either $(\forall i \in P)(\exists s)[(\forall j > s)(j \in P \imp i \leq_P j) \orr (\forall j > s)(j \in P \imp i \mid_P j)]$ or $(\forall i \in P)(\exists s)[(\forall j > s)(j \in P \imp i \geq_P j) \orr (\forall j > s)(j \in P \imp i \mid_P j)]$.  $\cac$ is the statement ``every infinite partial order has an infinite chain or an infinite antichain.''  $\scac$ is the restriction of $\cac$ to stable partial orders.
\end{definition}

Hirschfeldt and Shore give a detailed study of $\cac$ and $\scac$ (and many other principles) in~\cite{hirschfeldt2007combinatorial}.  They show that $\rca \vdash \cac \biimp \coh \andd \scac$ and that, over $\rca$, $\scac$ is strictly weaker than $\cac$ and $\cac$ is strictly weaker than $\rt^2_2$.

\begin{definition}[The Erd\H{o}s-Moser theorem]
A \emph{tournament} $T = (D,T)$\footnote{The notational convention is that a partial order $P = (P, \leq_P)$ is identified with its underlying set, whereas a tournament $T = (D,T)$ is identified with its relation.} consists of a set $D \subseteq \Nb$ and an irreflexive binary relation on $D$ such that for all $x,y \in D$ with $x \neq y$, exactly one of $T(x,y)$ and $T(y,x)$ holds.  A tournament $T$ is \emph{transitive} if the relation $T$ is transitive in the usual sense.  A tournament $T$ is \emph{stable} if $(\forall x \in D)(\exists n)[(\forall y > n)(y \in D \imp T(x,y)) \orr (\forall y > n)(y \in D \imp T(y,x))]$.  A \emph{sub-tournament} of $T$ is a tournament of the form $(E, E^2 \cap T)$ for an $E \subseteq D$.  $\emo$ is the statement ``for every infinite tournament there is an infinite transitive sub-tournament.''  $\semo$ is the restriction of $\emo$ to stable tournaments.
\end{definition}  

It is easy to see that $\rca \vdash \rt^2_2 \imp \emo$ and that $\rca \vdash \srt^2_2 \imp \semo$.  Furthermore, $\semo$ is strictly weaker than $\emo$ over $\rca$.  This can be deduced from the fact that $\rca \vdash \emo \imp \dnrs{2}$(Joseph Miller, personnal communication; see Section~\ref{sec:DNR} below for the definition of $\dnrs{2}$) and the fact that there is a (non-standard) model of $\rca + \srt^2_2$ (and hence of $\rca + \semo$) that contains only low sets~\cite{chongmetamathematics} (see~\cite{PateyRainbow} for a complete explanation).  By work of Bovykin and Weiermann~\cite{bovykin2005strength} and of Lerman, Solomon, and Towsner~\cite{lerman2013separating}, $\emo$ and $\semo$ are strictly weaker than $\rt^2_2$ over $\rca$ and are independent of $\cac$ and $\scac$ over $\rca$.

\subsection{Weak weak K\"onig's lemma and Martin-L\"of randomness}\label{sec:WWKL&MLR}
Let $T \subseteq 2^{<\Nb}$ be a tree and let $q \in \Qb$.  The measure of (the set of paths through) $T$ is $\geq q$ (written $\mu(T) \geq q$) if $\forall s(2^{-s}|T^s| \geq q)$ (recall that $T^s$ is the set of strings in $T$ of length $s$).  A tree $T \subseteq 2^{<\Nb}$ has \emph{positive measure}, written $\mu(T) > 0$, if $(\exists q \in \Qb)(q > 0 \andd \mu(T) \geq q)$.  Weak weak K\"onig's lemma ($\wwkl$), introduced by Yu and Simpson~\cite{YuSimpson}, is the statement ``every subtree of $2^{<\Nb}$ with positive measure has an infinite path.''  $\wwkl$ is strictly weaker than $\wkl$ over $\rca$~\cite{YuSimpson}.  It is well-known that, over $\rca$, $\wwkl$ is equivalent to $\ran{1}$, which is the statement ``for every set $X$, there is a set $Y$ that is Martin-L\"of random relative to $X$'' (see~\cite{AvigadDeanRute}, for example).

Avigad, Dean, and Rute~\cite{AvigadDeanRute} generalize $\wwkl$ to $\wwkls{n}$ for each $n \in \omega$ with $n \geq 1$.  Informally, $\wwkls{n}$ asserts that if $X$ is a set and $T \subseteq 2^{<\Nb}$ is a tree of positive measure that is recursive in $X^{(n-1)}$, then $T$ has an infinite path.  Care must be taken to formalize $\wwkls{n}$ without implying the existence of $X^{(n-1)}$ or of $T$.  For $n \in \omega$ with $n \geq 1$, let $e \in X^{(n)}$ abbreviate the formula
\begin{align*}
(Q x_{n-2}) \dots (\exists x_1)(\forall x_0)(\exists \sigma \preceq X)[\Phi_e^\sigma(\la x_{n-2}, \dots, x_0, 0 \ra)\da].
\end{align*}
The quantifier `$Q$' is `$\forall$' if $n$ is even and is `$\exists$' if $n$ is odd.  In the case $n=1$, the formula is simply $(\exists \sigma \preceq X)[\Phi_e^\sigma(\la 0 \ra)\da]$.  Let $\sigma \preceq X^{(n)}$ abbreviate the formula $\sigma \in 2^{<\Nb} \andd (\forall e < |\sigma|)(\sigma(e) = 1 \biimp e \in X^{(n)})$.  Let $\Phi_e^{X^{(n)}}(x) = y$ abbreviate the formula $(\exists \sigma \preceq X^{(n)})(\Phi_e^\sigma(x) = y)$.  If $\varphi(\sigma)$ is a formula defining a subtree of $2^{<\Nb}$ and $q \in \Qb$, then that the measure of this tree is $\geq q$ can be expressed by a formula that states that for every $n$ there is a sequence $\la \sigma_0, \sigma_1, \dots, \sigma_{k-1} \ra$ of distinct strings in $2^n$ such that $k2^{-n} \geq q$ and $(\forall i < k)\varphi(\sigma_i)$.  Similarly, that the measure of the tree defined by $\varphi$ is positive can be expressed by a formula that says that there is a rational $q > 0$ such that the measure of the tree is $\geq q$.

\begin{definition}
For $n \in \omega$ with $n \geq 2$, $\wwkls{n}$ is the statement ``for every $X$ and $e$, if $\Phi_e^{X^{(n-1)}}$ is the characteristic function of a subtree of $2^{<\Nb}$ with positive measure, then this tree has an infinite path.''  (That is, there is a function $f \colon \Nb \imp 2$ such that $\forall m [\Phi_e^{X^{(n-1)}}(f \restriction m) = 1]$.)
\end{definition}

Avigad, Dean, and Rute~\cite{AvigadDeanRute} also generalize $\ran{1}$ to $\ran{n}$, which is a formalization of the statement ``for every $X$ there is a $Y$ that is $n$-random relative to $X$,'' for all $n \in \omega$ with $n \geq 1$.  They prove that the correspondence between $\wwkls{1}$ and $\ran{1}$ also generalizes to all $n$ once $\bsig^0_n$ is added to $\ran{n}$:  for every $n \in \omega$ with $n \geq 1$, $\wwkls{n}$ and $\ran{n} + \bsig^0_n$ are equivalent over $\rca$.  Notice that this implies that for every $n \in \omega$ with $n \geq 1$, $\rca + \wwkls{n} \vdash \bsig^0_n$.

\subsection{Diagonally non-recursive functions}\label{sec:DNR}
A function $f: \Nb \rightarrow \Nb$ is \emph{diagonally non-recursive} ($\dnrf$) if $\forall e(f(e) \neq \Phi_e(e))$ and is \emph{diagonally non-recursive relative to a set $X$} ($\dnrf(X)$) if $\forall e(f(e) \neq \Phi_e^X(e))$.  An important characterization is that a set computes a $\dnrf$ function if and only if it computes a \emph{fixed-point free} function, i.e., a function $g \colon \Nb \imp \Nb$ such that $\forall e(W_{g(e)} \neq W_e)$.

\begin{definition}\label{def:DNR}
$\dnr$ is the statement ``for every $X$ there is a function $f$ such that $\forall e(f(e) \neq \Phi_e^X(e))$.'' 
\end{definition}

It is clear that no $\dnrf$ function is recursive and therefore that $\rca \nvdash \dnr$.  On the other hand, it is a classical result of Ku\v{c}era~\cite{Kucera1985} that every Martin-L\"of random set computes a $\dnr$ function, and its proof readily relativizes and easily formalizes in $\rca$.  Therefore $\rca \vdash \wwkl \rightarrow \dnr$.  By work of Ambos-Spies, Kjos-Hanssen, Lempp, and Slaman~\cite{AmbosSpies2004}, $\dnr$ is strictly weaker than $\wwkl$ over $\rca$.

As with weak weak K\"onig's lemma and Martin-L\"of randomness, we can define a hierarchy of principles expressing the existence of diagonally non-recursive functions.  For every $n \in \omega$ with $n \geq 1$, we generalize $\dnr$ to $\dnrs{n}$, which is a formalization of the statement ``for every $X$ there exists a function that is diagonally non-recursive relative to $X^{(n-1)}$.''

\begin{definition}\label{def:nDNR}
$\dnrs{n}$ is the statement ``for every $X$ there is a function $f$ such that $\forall e(f(e) \neq \Phi_e^{X^{(n-1)}}(e))$''.  
\end{definition}

Of course, the `$\Phi_e^{X^{(n-1)}}(e)$' in the above definition should be interpreted as it is in Section~\ref{sec:WWKL&MLR}.  Again, $\rca \vdash \wwkls{n} \rightarrow \dnrs{n}$.  We prove this via $\ran{n}$ to avoid the use of $\bsig^0_n$.

\begin{theorem}
$\rca \vdash \ran{n} \imp \dnrs{n}$.
\end{theorem}

\begin{proof}
Let $X$ be given, and, by $\ran{n}$, let $Y$ be $n$-random relative to $X$.  Define $f \colon \Nb \imp \Nb$ by $\forall e(f(e) = \text{the number whose binary expansion is $Y \restriction e$})$.  We show that $f$ is almost $\dnrf$ relative to $X^{(n-1)}$.  Consider the sequence $(\mathcal{U}_i)_{i \in \Nb}$ defined by
\begin{align*}
\mathcal{U}_i = \{Z : (\exists e > i)(\exists \sigma \in 2^e)(\text{the binary expansion of $\Phi_e^{X^{(n-1)}}(e)$ is $\sigma$ and $\sigma \preceq Z$})\}.
\end{align*}
$(\mathcal{U}_i)_{i \in \Nb}$ is a uniform sequence of strict (in the sense of~\cite{AvigadDeanRute}) $\Sigma^{0,X}_n$ sets, and $\forall i(\mu(\mathcal{U}_i) \leq 2^{-i})$ because $\mathcal{U}_i$ contains at most one string of length $e$ for each $e > i$.  Thus $(\mathcal{U}_i)_{i \in \Nb}$ is a $\Sigma^{0,X}_n$-test.  Therefore $Y \notin \mathcal{U}_i$ for some $i \in \Nb$.  Suppose for a moment that $f(e) = \Phi_e^{X^{(n-1)}}(e)$ for an $e > i$.  This means that $\Phi_e^{X^{(n-1)}}(e)$ is the number whose binary expansion is $Y \restriction e$ and thus that $\cyl{Y \restriction e} \subseteq \mathcal{U}_i$, a contradiction.  Therefore $f$ is $\dnrf$ relative to $X^{(n-1)}$ at all $e > i$.  For each $e \leq i$, we can effectively find an index $m_e$ such that $\forall \sigma \forall x (\Phi_{m_e}^\sigma(x) = \Phi_e^\sigma(e))$.  Thus $f(m_e) \neq \Phi_{m_e}^{X^{(n-1)}}(m_e) = \Phi_e^{X^{(n-1)}}(e)$.  So we may obtain a function that is $\dnrf$ relative to $X^{(n-1)}$ by changing $f(e)$ to $f(m_e)$ for all $e \leq i$.
\end{proof}

It follows that $\rca \vdash \wwkls{n} \imp \dnrs{n}$ because $\rca \vdash \wwkls{n} \imp \ran{n}$.  By work of Slaman~\cite{slaman2011first}, $\rca + \ran{2} \nvdash \bsig^0_2$, so we may also conclude that $\rca + \dnrs{2} \nvdash \bsig^0_2$.

\subsection{Ramsey-type weak K\"onig's lemma}\label{subsect:ramsey-konig}
In~\cite{Flood2012}, Flood introduced the principle \emph{Ramsey-type weak K\"onig's lemma}, a simultaneous weakening of $\wkl$ and $\rt^2_2$.  Informally, $\rwkl$ states that if $T \subseteq 2^{<\Nb}$ is an infinite tree, then there is an infinite set $X$ that is either a subset of a path through $T$ or disjoint from a path through $T$ (when thinking of the paths through $T$ as characteristic strings of subsets of $\Nb$).  When formalizing $\rwkl$, care must be taken to avoid implying the existence of a path through $T$ and hence implying $\wkl$.

\begin{definition}\label{def-RWKL}
A set $H \subseteq \Nb$ is \emph{homogeneous} for a $\sigma \in 2^{<\Nb}$ if $(\exists c < 2)(\forall i \in H)(i < |\sigma| \imp \sigma(i)=c)$, and a set $H \subseteq \Nb$ is \emph{homogeneous} for an infinite tree $T \subseteq 2^{<\Nb}$ if the tree $\{\sigma \in T : \text{$H$ is homogeneous for $\sigma$}\}$ is infinite.  $\rwkl$ is the statement ``for every infinite subtree of $2^{<\Nb}$, there is an infinite homogeneous set.''
\end{definition}

\begin{remark}\label{rem:rkl-renamed}
Flood actually named his principle $\s{RKL}$, for \emph{Ramsey-type K\"onig's lemma}.  We found it more convenient to refer to this principle as $\rwkl$.  Indeed, we study Ramsey-type variations of several principles, and the convention we follow is to add an `$\s{R}$' to a principle's name to denote its Ramsey-type variation (see, for example, $\rsat$, $\rcolor_n$, and $\rwwkl$ below).  The typical scheme is to view a combinatorial principle as a problem comprised of instances and solutions to these instances.  For example, with $\wkl$, an instance would be an infinite subtree of $2^{<\Nb}$, and a solution to that instance would be a path through the tree.  The Ramsey-type variation of a principle has the same class of instances, but instead of asking for a full solution in the problem's original sense, we ask only for an infinite set consistent with being a solution.
\end{remark}

Flood~\cite{Flood2012} proved that $\rca \vdash \wkl \imp \rwkl$ and that $\rca \vdash \srt^2_2 \imp \rwkl$.  He also noted that $\rwkl$ is strictly weaker than both $\wkl$ and $\srt^2_2$ over $\rca$ because $\wkl$ and $\srt^2_2$ are independent over $\rca$.  The result $\rca \vdash \srt^2_2 \imp \rwkl$ can be improved to $\rca \vdash \semo \imp \rwkl$, which we show now.

\begin{theorem}
$\rca \vdash \semo \imp \rwkl$.\footnote{Obtained independently by Flood and Towsner~\cite{Flood2014Separating}.}
\end{theorem}
\begin{proof}
Let $T \subseteq 2^{<\Nb}$ be an infinite tree.  For each $s \in \Nb$, let $\sigma_s$ be the leftmost element of~$T^s$.  We define a tournament $R$ from the tree $T$.  For $x < s$, if $\sigma_s(x)=1$, then $R(x,s)$ holds and $R(s,x)$ fails; otherwise, if $\sigma_s(x) = 0$, then $R(x,s)$ fails and $R(s,x)$ holds. This tournament $R$ is essentially the same as the coloring $f(x,s)=\sigma_s(x)$ defined by Flood in his proof that $\rca \vdash \srt^2_2 \imp \rwkl$ (\cite{Flood2012}~Theorem~5), in which he showed that $f$ is stable.  By the same argument, $R$ is stable. 

Apply $\semo$ to $R$ to get an infinite transitive sub-tournament $U$.  Say that a $\tau \in U^{<\Nb}$ satisfies~$(\star)$ if $\rng(\tau)$ is not homogenous for $T$ with color $1$ and $(\forall k < |\tau|)R(\tau(k),\tau(k+1))$.  Consider a hypothetical $\tau \in U^{<\Nb}$ satisfying $(\star)$.  There must be a $k < |\tau|$ such that $R(s,\tau(k))$ for cofinitely many $s$.  This is because otherwise there would be infinitely many $s$ such that $(\forall k < |\tau|)R(\tau(k),s)$ and hence infinitely many $s$ for which $\rng(\tau)$ is homogeneous for $\sigma_s$ with color $1$, contradicting that $\rng(\tau)$ is not homogeneous for $T$ with color $1$.  From the facts that $R(s,\tau(k))$ for cofinitely many $s$, that $(\forall k < |\tau|)R(\tau(k),\tau(k+1))$, and that $U$ is transitive, we conclude that $R(s,\tau(|\tau|-1))$ for cofinitely many $s$.

The proof now breaks into two cases.  First, suppose that the $\tau(|\tau|-1)$ for the $\tau \in U^{<\Nb}$ satisfying~$(\star)$ are unbounded.  Then, because $(\star)$ is a $\Sigma^0_1$ property of $U$, there is an infinite set $X$ consisting of numbers of the form $\tau(|\tau|-1)$ for $\tau \in U^{<\Nb}$ satisfying~$(\star)$.  As argued above, every $x \in X$ satisfies $R(s,x)$ for cofinitely many $s$.  Thus we can thin out $X$ to an infinite set $H$ such that $(\forall x,y \in H)(x < y \imp R(y,x))$. Thus~$H$ is homogeneous for $T$ with color $0$ because $H$ is homogeneous for $\sigma_y$ with color $0$ for every $y \in H$.

Second, suppose that the $\tau(|\tau|-1)$ for the $\tau \in U^{<\Nb}$ satisfying $(\star)$ are bounded, say by $m$.  Then $H = U \setminus \{0,1,\dots,m\}$ is homogeneous for $T$ with color $1$.  To see this, suppose not.  Then there is a finite $V \subseteq H$ that is not homogeneous for $T$ with color $1$.  Let $\tau \in V^{<\Nb}$ be the enumeration of $V$ in the order given by $R$:  $(\forall k < |\tau|)R(\tau(k),\tau(k+1))$.  Then $\tau$ satisfies $(\star)$, but $\tau(|\tau|-1) > m$.  This is a contradiction.
\end{proof}

Flood also proved that $\rca \vdash \rwkl \imp \dnr$, and this result prompted him to ask if $\rca \vdash \dnr \imp \rwkl$.  Corollary~\ref{cor:DNRNoImpRWKL} shows that the answer to this question is negative.

\section{Ramsey-type K\"onig's lemma and its variants}\label{sec-rtkl-variants}

We investigate the strengths of several variations of $\rwkl$.  Our variations are obtained in one of two ways.  First, we consider Ramsey-type K\"onig's lemma principles applied to different classes of trees.  We show that when we restrict to trees of positive measure, the resulting principle is equivalent to $\dnr$ (Theorem~\ref{thm:dnr-rwwkl}); that when we allow subtrees of $k^{<\Nb}$ (for a fixed $k \in \omega$ with $k \geq 2$), the resulting principle is equivalent to $\rwkl$ (Theorem~\ref{thm:lrwkl-rwkl}); that when we allow bounded subtrees of $\Nb^{<\Nb}$, the resulting principle is equivalent to $\wkl$ (Theorem~\ref{thm-RbWKLisWKL}); and that when we allow arbitrary finitely-branching subtrees of $\Nb^{<\Nb}$, the resulting principle is equivalent to $\aca$ (Theorem~\ref{thm-RKLisACA}).  Second, we impose additional requirements on the homogeneous sets that $\rwkl$ asserts exist.  If we require that homogeneous sets be homogeneous for color $0$ (and restrict to trees that have no paths that are eventually $1$), then the resulting principle is equivalent to $\wkl$ (Theorem~\ref{thm:fixed-rwkl-wkl}).  If we impose a bound on the sparsity of the homogeneous sets, then the resulting principle is also equivalent to $\wkl$ (Theorem~\ref{thm:packed-rwkl-wkl}).  If we require that the homogeneous sets be subsets of some prescribed infinite set, then the resulting principle is equivalent to $\rwkl$ (Theorem~\ref{thm:lrwkl-rwkl}).  It is interesting to note that each variation of $\rwkl$ that we consider is either equivalent to $\rwkl$ itself or some other well-known statement.  We also note that sometimes the Ramsey-type variant of a principle is equivalent to the original principle, as with K\"onig's lemma for bounded trees and K\"onig's lemma for arbitrary finitely-branching trees; and that sometimes the Ramsey-type variant of a principle is strictly weaker than the original principle, as with weak K\"onig's lemma and weak weak K\"onig's lemma.

Several results in this section indicate robustness in $\rwkl$.  For example, we may generalize $\rwkl$ to subtrees of $k^{<\Nb}$ (for fixed $k \in \omega$ with $k \geq 2$) without changing the principle's strength.  We explore the robustness of $\rwkl$ more fully in Section~\ref{sec-RSAT} and Section~\ref{sec-RCOLOR}.  This robustness we take as evidence that $\rwkl$ is a natural principle.

\subsection{\texorpdfstring{$\dnrf$}{DNR} functions and subsets of paths through trees of positive measure}\label{sect:rwwkl-as-dnr}
Just as $\wkl$ can be weakened to $\wwkl$ by restricting to trees of positive measure, so can $\rwkl$ be weakened to $\rwwkl$ by restricting to trees of positive measure.

\begin{definition} \label{def:RWWKL}
$\rwwkl$ is the statement ``for every subtree of $2^{<\Nb}$ with positive measure, there is an infinite homogeneous set.''
\end{definition}

Applying $\rwwkl$ to a tree in which every path is Martin-L\"of random yields an infinite subset of a Martin-L\"of random set, and every infinite subset of every Martin-L\"of random set computes a $\dnrf$ function.  In fact, computing an infinite subset of a Martin-L\"of random set is equivalent to computing a $\dnr$ function, as the following theorem states.

\begin{theorem}[Kjos-Hanssen~\cite{KjosHanssen2009}, Greenberg and Miller~\cite{GreenbergM2009}]\label{thm:miller-dnc}
For every $A \in 2^\omega$, $A$ computes a $\dnrf$ function if and only if $A$ computes an infinite subset of a Martin-L\"of random set.
\end{theorem}

Theorem~\ref{thm:miller-dnc} also relativizes:  a set $A$ computes a $\dnrf(X)$ function if and only if it computes an infinite subset of a set that is Martin-L\"of random relative to $X$.  Thus one reasonably expects that $\dnr$ and $\rwwkl$ are equivalent over $\rca$.  This is indeed the case, as we show.  The proof makes use of the following recursion-theoretic lemma, which reflects a classical fact concerning diagonally non-recursive functions.

\begin{lemma}\label{lemma:dnr-rwwkl_helper}
The statement ``for every set $X$ there is a function $g \colon \Nb^3 \imp \Nb$ such that $\forall e,k,n(g(e,k,n) > n \andd (|W_e^X| < k \imp g(e,k,n) \notin W_e^X))$'' is provable in $\rca + \dnr$.
\end{lemma}

\begin{proof}
Fix a sequence of functions $(b_k)_{k \in \Nb}$ such that, for each $k \in \Nb$, $b_k$ maps $\Nb$ onto $\Nb^k$ in such a way that $b_k^{-1}(\vec x)$ is infinite for every $\vec x \in \Nb^k$.  Let $c \colon \Nb \imp \Nb$ be a function such that, for all $e,i,k,x \in \Nb$ with $i < k$, $\Phi_{c(e,i,k)}^X(x) = b_k(y)(i)$ for the $(i+1)$\textsuperscript{th} number $y$ enumerated in $W_e^X$ if $|W_e^X| \geq i+1$; and $\Phi_{c(e,i,k)}^X(x)\ua$ otherwise.  Let $f$ be diagonally non-recursive relative to $X$.  Define $g$ by letting $g(e,k,n)$ be the least $x > n$ such that $b_k(x) = \la f(c(e,0,k)), f(c(e,1,k)), \dots, f(c(e,k-1,k)) \ra$.  Suppose for a contradiction that $|W_e^X| < k$ but that $g(e,k,n) \in W_e^X$.  Then $g(e,k,n)$ is the $(i+1)$\textsuperscript{th} number enumerated into $W_e^X$ for some $i+1 < k$.  Hence $\Phi_{c(e,i,k)}^X(c(e,i,k)) = b_k(g(e,k,n))(i)$.  However, by the definition of $g$, $b_k(g(e,k,n))(i) = f(c(e,i,k))$.  Thus $f(c(e,i,k)) = \Phi_{c(e,i,k)}^X(c(e,i,k))$, contradicting that $f$ is $\dnrf$ relative to $X$.
\end{proof}

Notice that in the statement of the above lemma, $W_e^X$ need not exist as a set.  Thus `$|W_e^X| < k$' should be interpreted as `$\forall s(|W_{e,s}^X| < k)$,' where $(W_{e,s}^X)_{s \in \Nb}$ is the standard enumeration of $W_e^X$.

\begin{theorem}\label{thm:dnr-rwwkl}
$\rca \vdash \dnr \biimp \rwwkl$.\footnote{Obtained independently by Flood and Towsner~\cite{Flood2014Separating}.}
\end{theorem}

\begin{proof}
The direction $\rwwkl \rightarrow \dnr$ is implicit in Flood's proof that $\rca \vdash \rwkl \imp \dnr$ (\cite{Flood2012}~Theorem~8).  Indeed, Flood's proof uses the construction of a tree of positive measure due to Jockusch~\cite{Jockusch1974}.  (For a similar construction proving a generalization of $\rwwkl \rightarrow \dnr$, see the proof of Lemma~\ref{lem-nRWWKLimpliesnDNR} below.)  The proof that $\dnr \imp \rwwkl$ is similar to the original proof of Theorem~\ref{thm:miller-dnc}.  However, some adjustments are needed as the original argument uses techniques from measure theory and algorithmic randomness which can only be formalized within $\wwkl$.  We instead use explicit combinatorial bounds.  

Assume $\dnr$, and consider a tree $T$ of measure $\geq 2^{-c}$ for some $c$, which we can assume to be $\geq 3$ (the reason for this assumption will become clear). For a given set~$H \subseteq \Nb$ and a value $v \in \{0,1\}$, let $\Gamma^v_H = \{\sigma \in \str : (\forall i \in H)(\sigma(i) = v)\}$, and abbreviate $\Gamma^v_{\{n\}}$ by $\Gamma^v_{n}$.  For a tree $T$ and a constant $c$, let $\bad(n,T,c)$ be the $\Sigma^0_1$ predicate `$\mu(T \cap \Gamma^0_n) <  2^{-2c}$.'  In the following claim, $\{n : \bad(n,T,c)\}$ need not \emph{a priori} exist as a set, so `$|\{n : \bad(n,T,c)\}| < 2c$' should be interpreted in the same manner as `$|W_e^X| < k$' in the statement of Lemma~\ref{lemma:dnr-rwwkl_helper}.

\begin{claim*}
If $c \geq 3$ and $\mu(T) \geq 2^{-c}$, then $|\{n : \bad(n,T,c)\}| < 2c$.
\end{claim*}

\begin{proof}
Suppose for a contradiction that $|\{n : \bad(n,T,c)\}| \geq 2c$, and let $B$ be the first $2c$ elements enumerated in $\{n : \bad(n,T,c)\}$.  For each $n \in B$, the tree $T \cap \Gamma^0_n$ has measure $<2^{-2c}$, which implies that $(\forall n \in B)(\exists i)(|T^i \cap \Gamma^0_n| < 2^{i-2c})$ (recall that $T^i$ is the set of strings in $T$ of length $i$).  By $\bsig^0_1$, let $N_0$ be such that $(\forall n \in B)(\exists i < N_0)(|T^i \cap \Gamma^0_n| < 2^{i-2c})$, and observe that $(\forall n \in B)(\forall j > N_0)(|T^j \cap \Gamma^0_n| < 2^{j-2c})$.  Let $N = N_0 + \max(B)$.

On the one hand,
\begin{align*}
\left|T^N \cap \bigcup_{n \in B} \Gamma^0_n\right| = |T^N \setminus \Gamma^1_B| \geq  |T^N| - |\Gamma^1_B \cap \{0,1\}^N| \geq  2^{N-c} - 2^{N-2c}.
\end{align*}

On the other hand,
\begin{align*}
\left|T^N \cap \bigcup_{n \in B} \Gamma^0_n\right| = \left| \bigcup_{n \in B} T^N \cap \Gamma^0_n\right| \leq \sum_{n \in B} |T^N \cap \Gamma^0_n| \leq 2c2^{N-2c}.
\end{align*}

Putting the two together, we get that $ 2^{N-c} - 2^{N-2c} \leq 2c2^{N-2c}$, which is a contradiction for $c \geq 3$. 
\end{proof}

Let $g$ be as in Lemma~\ref{lemma:dnr-rwwkl_helper} for $X=T$.  Given a (canonical index for a) finite set $F$ and a $c$, we can effectively produce an index $e(F,c)$ such that $\forall n(n \in W_{e(F,c)}^T \biimp \bad(n, T \cap \Gamma^0_F, c))$.  Recursively construct an increasing sequence $h_0 < h_1 < h_2 < \dots$ of numbers by letting, for each $s \in \Nb$, $H_s = \{h_i : i < s\}$ and $h_s = g(e(H_s, c2^s), c2^{s+1}, \max(H_s \cup \{0\}))$.  Using $\ipi^0_1$, we prove that $\forall s(\mu(T \cap \Gamma^0_{H_s}) \geq 2^{-c2^s})$.  For $s=0$, this is simply the assumption $\mu(T) \geq 2^{-c}$.  Assuming $\mu(T \cap \Gamma^0_{H_s}) \geq 2^{-c2^s}$, the claim implies that $|W_{e(H_s, c2^s)}^T| < c2^{s+1}$.  Thus $h_s = g(e(H_s, c2^s), c2^{s+1}, \max(H_s \cup \{0\})) \notin W_{e(H_s, c2^s)}$, and therefore $\neg \bad(h_s, T \cap \Gamma^0_{H_s}, c2^s)$.  This means that $\mu(T \cap \Gamma^0_{H_s} \cap \Gamma^0_{h_s}) \geq 2^{-c2^{s+1}}$, which is what we wanted because $\Gamma^0_{H_s} \cap \Gamma^0_{h_s} = \Gamma^0_{H_{s+1}}$.

Let $H = \{h_s : s \in \Nb\}$, which exists by $\Delta^0_1$ comprehension because the sequence $h_0 < h_1 < h_2 < \dots$ is increasing.  We show that $H$ is homogeneous for $T$.  Suppose for a contradiction that $H$ is not homogeneous for $T$.  This means that there are only finitely many $\sigma \in T$ such that $H$ is homogeneous for $\sigma$. Therefore at some level $s$, $\{\sigma \in T^s : (\forall i \in H)(\sigma(i) = 0)\}=\emptyset$.  As $H \cap \{0,1,\dots,s\} \subseteq H_s$, we in fact have that $\{\sigma \in T^s : (\forall i \in H_s)(\sigma(i) = 0)\}=\emptyset$.  In other words, $T \cap \Gamma^0_{H_s} = \emptyset$, which contradicts $\mu(T \cap \Gamma^0_{H_s}) \geq 2^{-c2^s}$.  Thus $H$ is homogeneous for $T$.
\end{proof}

Fix $n \in \omega$ with $n \geq 2$.  Just as with $\wwkls{n}$, it is possible to define $\rwwkls{n}$ to be the generalization of $\rwwkl$ to $X^{(n-1)}$-computable trees.  The equivalence between $\dnrs{n}$ and $\rwwkls{n}$ persists in the presence of sufficient induction.

\begin{definition}
For $n \in \omega$ with $n \geq 2$, $\rwwkls{n}$ is the statement ``for every $X$ and $e$, if $\Phi_e^{X^{(n-1)}}$ is the characteristic function of a subtree of $2^{<\Nb}$ with positive measure, then there is an infinite homogeneous set.''  (That is, there is an infinite $H \subseteq \Nb$ that is homogeneous for infinitely many $\sigma \in 2^{< \Nb}$ such that $\Phi_e^{X^{(n-1)}}(\sigma) = 1$.)
\end{definition}

\begin{lemma}\label{lem-nRWWKLimpliesnDNR}
For every $n \in \omega$ with $n \geq 1$, $\rca + \bsig^0_n \vdash \rwwkls{n} \imp \dnrs{n}$.
\end{lemma}

\begin{proof}
Fix a sequence of functions $(b_k)_{k \in \Nb}$ such that, for each $k \in \Nb$, $b_k$ is a bijection between $\Nb$ and $\Nb^{[k]}$.  Let $X$ be given.  Let $e$ be an index such that $\Phi_e^{X^{(n-1)}}(\sigma)=1$ if
\begin{align*}
(\forall i < |\sigma|)(\Phi_{i,|\sigma|}^{X^{(n-1)}}(i)\da \imp \text{$b_{i+3}(\Phi_{i,|\sigma|}^{X^{(n-1)}}(i))$ is not homogeneous for $\sigma$}),
\end{align*}
and $\Phi_e^{X^{(n-1)}}(\sigma)=0$ otherwise.  It is clear that $\Phi_e^{X^{(n-1)}}$ is the characteristic function of a tree.  We need to show that this tree has positive measure.  Fix $s \in \Nb$.  By bounded $\Delta^0_n$ comprehension, which is a consequence of $\bsig^0_n$ (see, for example, \cite{Hajek1998}~Lemma 2.19), $T^s = \{\sigma \in 2^s : \Phi_e^{X^{(n-1)}}(\sigma)=1\}$ exists as a finite set.  For each $i < s$, the proportion of strings in $2^s$ missing from $T^s$ on account of $\Phi_i^{X^{(n-1)}}$ is at most $2^{-i-2}$.  Therefore $|T^s|2^{-s} \geq \sum_{i < s}2^{-i-2} \geq 1/2$, so the tree indeed has positive measure.

By $\rwwkls{n}$, there is an infinite homogeneous set $H$ for the tree described by $\Phi_e^{X^{(n-1)}}$.  For each $i \in \Nb$, let $H_i$ denote the set consisting of the $i$ least elements of $H$.  Define $f \colon \Nb \imp \Nb$ by $f(i) = b_{i+3}^{-1}(H_{i+3})$.  We finish the proof by showing that $f$ is $\dnrf$ relative to $X^{(n-1)}$.  Suppose for a contradiction that there is an $i \in \Nb$ such that $f(i) = \Phi_i^{X^{(n-1)}}(i)$, and let $s$ be such that $\Phi_{i,s}^{X^{(n-1)}}(i)\da$.  By the definition of $f$, we have that $b_{i+3}^{-1}(H_{i+3}) = f(i) = \Phi_{i,s}^{X^{(n-1)}}(i)$.  By applying the bijection $b_{i+3}$, we have that $b_{i+3}(\Phi_{i,s}^{X^{(n-1)}}(i)) = H_{i+3}$ is homogeneous for the tree described by $\Phi_e^{X^{(n-1)}}$.  This is a contradiction because if $b_{i+3}(\Phi_{i,s}^{X^{(n-1)}}(i))$ is homogeneous for a $\sigma \in 2^{<\Nb}$ with $|\sigma| > s$, then $\Phi_e^{X^{(n-1)}}(\sigma) = 0$.
\end{proof}

\begin{lemma}\label{lem-nDNRimpliesnRWWKL}
For every $n \in \omega$ with $n \geq 1$, $\rca + \isig^0_n \vdash \dnrs{n} \imp \rwwkls{n}$.
\end{lemma}

\begin{proof}[Proof sketch]
Follow the proof that $\rca \vdash \dnr \imp \rwwkl$ from Theorem~\ref{thm:dnr-rwwkl}, but interpret $T$ as an $X^{(n-1)}$-computable tree of positive measure in the sense of Section~\ref{sec:WWKL&MLR}.  The proof of Lemma~\ref{lemma:dnr-rwwkl_helper} goes through in $\rca$ when $X$ is replaced by $X^{(n-1)}$ and $\dnr$ is replaced by $\dnrs{n}$.  The predicate $\bad(n,T,c)$ is now $\Sigma^0_n$, and the proof of the claim goes through in $\rca + \bsig^0_n$.  The function $g$ exists by the generalization of Lemma~\ref{lemma:dnr-rwwkl_helper}, and the function $e$ is the same as it was before.  The set $H$ is constructed from $g$ and $e$ as it was before.  Use $\ipi^0_n$, a consequence of $\rca + \isig^0_n$, to prove the analog of $\forall s(\mu(T \cap \Gamma^0_{H_s}) \geq 2^{-c2^s})$.  The rest of the proof is the same as it was before.
\end{proof}

\begin{theorem}
For every $n \in \omega$ with $n \geq 1$, $\rca + \isig^0_n \vdash \dnrs{n} \biimp \rwwkls{n}$.
\end{theorem}

\begin{proof}
The theorem follows from Lemma~\ref{lem-nRWWKLimpliesnDNR} and Lemma~\ref{lem-nDNRimpliesnRWWKL}.
\end{proof}

We leave open the question of the exact amount of induction required to prove Lemma~\ref{lem-nRWWKLimpliesnDNR} and Lemma~\ref{lem-nDNRimpliesnRWWKL}.  It would be particularly interesting to determine whether or not $\rwwkls{n}$ implies $\bsig^0_2$.

\begin{question}\label{qu-2RWWKLvsBSig2}
Does $\rca + \rwwkls{2} \vdash \bsig^0_2$?
\end{question}

In~\cite{Flood2012}, Flood also studies what he calls $\rkl^{(1)}$, which is $\rwkl$ for $\Sigma^0_1$-definable infinite subtrees of $2^{<\Nb}$.  He notes that $\rkl^{(1)}$ is equivalent to $\rwkl$ for $\Pi^0_2$ trees, and thus it follows that $\rkl^{(1)}$ is equivalent to $\rwkl$ for $\Delta^0_2$ trees, a statement which we would call $\rwkls{2}$ in the foregoing notation.  Flood presents Yokoyama's proof that $\rca \vdash \rwkls{2} \imp \srt^2_2$, and Flood asks (\cite{Flood2012}~Question 22) if the reverse implication holds.  We show that it does not.

\begin{theorem}
$\rca + \srt^2_2 \nvdash \rwkls{2}$.
\end{theorem}

\begin{proof}
Over $\rca + \bst$, $\rwkls{2}$ implies $\rwwkls{2}$ and, by Lemma~\ref{lem-nRWWKLimpliesnDNR}, $\rwwkls{2}$ in turn implies $\dnrs{2}$.  However, $\srt^2_2$ does not imply $\dnrs{2}$ over $\rca + \bst$ because there are models of $\rca + \bst + \srt^2_2$ in which every set is low~\cite{chongmetamathematics}.  In particular, every set in such a structure is computable from $0'$, so such a structure is not a model of $\dnrs{2}$.
\end{proof}

\subsection{Changing homogeneity constraints}\label{sect:changing-homogeneity-constraints}
Notice that the homogeneous set constructed in the proof of Theorem~\ref{thm:dnr-rwwkl} is always homogeneous for color $0$, and we could just as easily constructed a set homogeneous for color $1$.  Thus no additional power is gleaned from $\rwwkl$ by prescribing the color of the homogeneous set ahead of time.

\begin{corollary}[to the proof of Theorem~\ref{thm:dnr-rwwkl}]\label{cor:fixed-rwwkl-rwwkl}
The following statements are equivalent over $\rca$:
\begin{itemize}
  \item[(i)] $\dnr$
  \item[(ii)] $\rwwkl$
  \item[(iii)] For every tree $T \subseteq \str$ of positive measure, there is an infinite set that is homogeneous for $T$ with color $0$.
\end{itemize}
\end{corollary}

One then wonders if any additional strength is gained by modifying $\rwkl$ to require that homogeneous sets be homogeneous for color $0$.  Of course an infinite homogeneous set for color $0$ need not exist in general, so we restrict to trees that do not have paths that are eventually $1$.  For the purposes of the next theorem, ``$T$ has no path that is eventually $1$'' means $\forall \sigma \exists n (\sigma^\smf 1^n \notin T)$.

\begin{theorem}\label{thm:fixed-rwkl-wkl}
The following statements are equivalent over $\rca$:
\begin{itemize}
  \item[(i)] $\wkl$
  \item[(ii)] For every infinite tree $T \subseteq \str$ with no path that is eventually $1$, there is an infinite set homogeneous for $T$ with color $0$.
\end{itemize}
\end{theorem}

\begin{proof}
Clearly $(i) \imp (ii)$.  For $(ii) \imp (i)$, let $S \subseteq 2^{<\Nb}$ be an infinite tree.  We define a tree $T \subseteq \str$ whose paths have $0$'s only at positions corresponding to codes of initial segments of paths through $S$.  Let $(\tau_i)_{i \in \Nb}$ be the enumeration of $\str$ in length-lexicographic order, and note that $\forall i(|\tau_i| \leq i)$.  Let
\begin{align*}
T = \{\sigma \in \str : (\exists \tau \in S^{|\sigma|})(\forall i < |\sigma|)(\sigma(i) = 0 \biimp \tau_i \preceq \tau)\}.
\end{align*}

$T$ is a tree because if $\sigma \in T$ is witnessed by $\tau \in S$ and $n < |\sigma|$, then $\tau \restriction n \in S$ witnesses that $\sigma \restriction n \in T$.  Every string of length $n$ in $S$ witnesses the existence of a string of length $n$ in $T$, so $T$ is infinite because $S$ is infinite.

We show that $T$ has no path that is eventually $1$.  Consider a $\sigma \in \str$.  Choose $m$ and $n$ such that $\forall i(|\tau_i| = m \imp |\sigma| < i < |\sigma|+n)$.  Suppose for a contradiction that $\tau \in S$ witnesses that $\sigma^\smf 1^n \in T$.  If $i$ is such that $\tau_i = \tau \restriction m$, then $|\sigma| < i < |\sigma|+n$.  So, because $\tau$ witnesses that $\sigma^\smf 1^n \in T$, we have the contradiction $(\sigma^\smf 1^n)(i) = 0$.  Thus $T$ has no path that is eventually $1$.

By $(ii)$, let $H$ be infinite and homogeneous for $T$ with color $0$.  If $i$ and $j$ are in $H$ with $i \leq j$, then $\tau_i$ and $\tau_j$ are in $S$ with $\tau_i \preceq \tau_j$.  This can be seen by considering a $\sigma \in T$ of length $j+1$ for which $H$ is homogeneous with color $0$ and a $\tau \in S$ witnessing that $\sigma \in T$.  Thus we can define an $f \in 2^\Nb$ by $f = \bigcup_{i \in H}\tau_i$, and this $f$ is a path through $S$ because $\tau_i \in S$ for every $i \in H$.
\end{proof}

We now study a variant of $\rwkl$ where the homogeneous sets are required to not be too sparse, namely, everywhere-packed homogeneous sets.  This notion is not to be confused with the notion of a packed homogeneous set introduced by Flood~\cite{FloodPacked}.  Flood studies the computability-theoretic content of Erd\H{o}s and Galvin's~\cite{ErdosGalvin} \emph{packed} variants of Ramsey's theorem.  These theorems weaken homogeneity to a property called \emph{semi-homogeneity}, but they require that these semi-homogeneous sets satisfy a certain density requirement.  Flood shows that the packed variants of Ramsey's theorem behave similarly to Ramsey's theorem.  We formulate an everywhere-packed variant of $\rwkl$ and prove that it is equivalent to $\wkl$.  For this formulation, we consider an alternate definition of homogeneity.

\begin{definition}\label{def:hom-gen-tree}
A partial function $h \colon \!\!\subseteq \Nb \imp \Nb$ is \emph{homogeneous} for $\sigma \in \Nb^{<\Nb}$ if $(\forall n \in \dom(h))(n < |\sigma| \imp \sigma(n) = h(n))$.  If $T$ is an infinite, finitely branching tree, a partial function $h \colon \!\!\subseteq \Nb \imp \Nb$ is \emph{homogeneous} for $T$ if the tree $\{\sigma \in T : \text{$h$ is homogeneous for $\sigma$}\}$ is infinite.
\end{definition}

In Definition~\ref{def:hom-gen-tree}, we always assume that $\dom(h)$ exists as a set.  This is no real restriction because in $\rca$ one can prove that every infinite $\Sigma^0_1$-definable set has an infinite subset that actually exists as a set.  Thus if $h$ is infinite, we may always restrict $h$ to an infinite subset of $\dom(h)$ that exists as a set.  

If $h$ is infinite and homogenous for an infinite tree $T \subseteq \str$, then both of the sets $h^{-1}(0)$ and $h^{-1}(1)$ are homogeneous for $T$, and one of them must be infinite.  Conversely, if $H$ is homogeneous for an infinite tree $T \subseteq \str$ with color $c$, then the function $h \colon H \imp 2$ with constant value $c$ is homogeneous for $T$.  Thus, over $\rca$, it is equivalent to define $\rwkl$ in terms of set-homogeneity or in terms of function-homogeneity.  However, function-homogeneity lets us impose the density constraints we need for our everywhere-packed variant of $\rwkl$.  Function-homogeneity also lets us formulate Ramsey-type variants of full K\"onig's lemma and of bounded K\"onig's lemma.

Recall that an \emph{order function} is a non-decreasing unbounded function $g \colon \Nb \imp \Nb$.

\begin{definition}
Let $g$ be an order function.  A partial function $h \colon \!\!\subseteq \Nb \imp \Nb$ is \emph{everywhere-packed} for $g$ if $\forall n(|\dom(h) \restriction n| \geq g(n))$.
\end{definition}

Our everywhere-packed variant of $\rwkl$ is equivalent to $\wkl$ by an argument that replaces a tree with a version of that tree having sufficient redundancy.

\begin{theorem}\label{thm:packed-rwkl-wkl}
$\rca$ proves that, for every order function $g$ satisfying $\forall n(g(n) \leq n)$, the following statements are equivalent:
\begin{itemize}
  \item[(i)] $\wkl$
  \item[(ii)] For every infinite tree $T \subseteq \str$, there is an infinite $h$ that is homogeneous for $T$ and everywhere-packed for $g$.
\end{itemize}
\end{theorem}

\begin{proof}
Fix an order function $g$ bounded by the identity.

The direction $(i) \imp (ii)$ is trivial.  If $f$ is a path through $T$, then $f$ is also homogeneous for $T$ and everywhere-packed for $g$.

Consider the direction $(ii) \imp (i)$, and let $T$ be an infinite subtree of $\str$.  Define a sequence $(u_n)_{n \in \Nb}$ by $u_0 = 0$ and $u_{n+1} = \mu i(g(i) \geq u_n + 1)$.  Let
\begin{align*}
S = \{\sigma \in \str : (\exists \tau \in T)[|\tau| = \mu i(|\sigma| < u_i) \andd (\forall i < |\sigma|)(\forall j < |\tau|)(i \in [u_j, u_{j+1}) \imp \sigma(i)=\tau(j))]\}.
\end{align*}

The idea behind $S$ is to ensure enough redundancy so that the domain of every infinite function that is homogeneous for $S$ and everywhere-packed for $g$ intersects each interval $[u_i, u_{i+1})$.  For example, if $g(n) = \lfloor \frac{n}{2} \rfloor$, then $u_1 = 2$, $u_2 = 6$, $u_3 = 14$, $u_4 = 30$, and the string $10101$ in $T$ corresponds in $S$ to
\begin{align*}
\overbrace{11}^{u_1 - u_0}\overbrace{0000}^{u_2 -
u_1}\overbrace{11111111}^{u_3 - u_2}\overbrace{0000000000000000}^{u_4 -
u_3}\overbrace{11111111111111111111111111111111}^{u_4-u_3}.
\end{align*}
It is easy to see that if $T$ is infinite, then so is $S$.  To see that $S$ is a tree, consider a $\sigma \in S$, and let $\tau \in T$ witness $\sigma$'s membership in $S$.  Given an $n \leq |\sigma|$, let $i = \mu i(n < u_i)$ and verify that $\tau \restriction i$ witnesses that $\sigma \restriction n$ is in $S$.  Let $h$ be an infinite function that is homogeneous for $S$ and everywhere-packed for $g$.

First we show that $(\forall j)(\dom(h) \cap [u_j, u_{j+1}) \neq \emptyset)$.  To see this, observe that $|\dom(h)\restriction u_{j+1}| \geq g(u_{j+1})$ because $h$ is everywhere-packed for $g$.  By definition, $g(u_{j+1}) \geq u_j + 1$.  Thus, by the finite pigeonhole principle, there must be an $i$ in $\dom(h)\restriction u_{j+1}$ with $i \geq u_j$.

Now, for each $j$, let $i_j$ be the least element of $\dom(h) \cap [u_j, u_{j+1})$.  Define a function $f$ by $f(j) = h(i_j)$.  This $f$ is a path through $T$.  To see this, fix $n$ and let $\sigma \in S^{u_n}$ be such that $h$ is homogeneous for $\sigma$.  Let $\tau \in T$ witness that $\sigma \in S$, and note that $|\tau| = n+1$.  For each $j < n$, we have that $\sigma(i_j) = \tau(j)$ by the choice of $i_j$ and the definition of $S$, and we also have that $\sigma(i_j) = f(j)$ by the choice of $\sigma$ and the definition of $f$.  Thus $f \restriction n = \tau \restriction n$, so $f \restriction n \in T$ as desired.
\end{proof}

\subsection{Ramsey-type K\"onig's lemma for arbitrary finitely branching trees}
Using the functional notion of homogeneity, we easily generalize $\rwkl$ to infinite, bounded trees and to infinite, finitely branching trees.  It is well known that K\"onig's lemma ($\kl$) is equivalent to $\aca$ (see \cite{Simpson2009}~Theorem~III.7.2) and that \emph{bounded K\"onig's lemma} (i.e., K\"onig's lemma for infinite bounded subtrees of $\Nb^{<\Nb}$) is equivalent to $\wkl$ (see \cite{Simpson2009}~Lemma~IV.1.4).  Interestingly, we find that the Ramsey-type variant of K\"onig's lemma is equivalent to $\aca$ and that the Ramsey-type variant of bounded K\"onig's lemma is equivalent to $\wkl$, not $\rwkl$.

\begin{definition}
$\rkl$ is the statement ``for every infinite, finitely branching subtree of $\Nb^{<\Nb}$ , there is an infinite homogeneous partial function.''\footnote{`$\rkl$' was Flood's original name for $\rwkl$.  We prefer to use `$\rkl$' for the Ramsey-type variant of K\"onig's lemma and `$\rwkl$' for the Ramsey-type variant of weak K\"onig's lemma.  See Remark~\ref{rem:rkl-renamed}.}
\end{definition}

\begin{theorem}\label{thm-RKLisACA}
$\rca \vdash \aca \biimp \rkl$.
\end{theorem}
\begin{proof}
Let $\kl_\textup{2-branching}$ denote $\kl$ restricted to trees $T \subseteq \Nb^{<\Nb}$ in which every $\sigma \in T$ has at most two immediate successors in $T$.  We take advantage of the fact that $\aca$, $\kl$, and $\kl_\textup{2-branching}$ are pairwise equivalent over $\rca$ (see~\cite{Simpson2009} Theorem~III.7.2).  Clearly $\rca \vdash \kl \imp \rkl$, so it suffices to show that $\rca \vdash \rkl \imp \kl_\textup{2-branching}$.  Thus let $T \subseteq \Nb^{<\Nb}$ be an infinite, finitely branching tree in which every $\sigma \in T$ has at most two immediate successors in $T$.  Let $(\tau_i)_{i \in \Nb}$ be a one-to-one enumeration of $\Nb^{<\Nb}$.  Define the tree $S$ by
\begin{align*}
S = \{\sigma \in \Nb^{<\Nb} : (\forall i, j < |\sigma|)[\tau_{\sigma(i)} \in T \andd |\tau_{\sigma(i)}| = i \andd (i \leq j \imp \tau_{\sigma(i)} \preceq \tau_{\sigma(j)})]\}.
\end{align*}
Clearly $S$ is a tree.  $S$ is infinite because $T$ is infinite and, given a $\tau$ in $T$, it is easy to produce a $\sigma$ in $S$ of the same length.  Now consider a $\sigma \in S$.  For $\sigma^\smf n$ to be in $S$, it must be that $\tau_n$ is an immediate successor of $\tau_{\sigma(|\sigma|-1)}$ on $T$ (or that $\tau_n = \emptyset$ in the case that $\sigma = \emptyset$).  As the enumeration $(\tau_n)_{n \in \Nb}$ is one-to-one and every string in $T$ has at most two immediate successors in $T$, it is also the case that every string in $S$ has at most two immediate successors in $S$.  In particular, $S$ is finitely branching.

By $\rkl$, let $h$ be infinite and homogeneous for $S$, and let $S_h$ be the infinite tree $\{\sigma \in S : \text{$h$ is homogeneous for $\sigma$}\}$.  Note that $S_h$ contains strings of arbitrary length because it is infinite and every string in $S_h$ contains at most two immediate successors in $S_h$.  Now, if $i$ and $j$ are in $\dom(h)$ with $i \leq j$, then $\tau_{h(i)}$ and $\tau_{h(j)}$ are in $T$ with $\tau_{h(i)} \preceq \tau_{h(j)}$, which may be seen by considering a $\sigma \in S_h$ of length $j+1$.  Hence $\bigcup_{i \in \dom(h)}\tau_{h(i)}$ is a path through $T$, as desired.
\end{proof}

In fact, the above proof shows that the restriction of $\rkl$ to trees in which each string has at most two immediate successors is also equivalent to $\aca$ over $\rca$.

Recall that a tree $T \subseteq \Nb^{<\Nb}$ is \emph{bounded} if there is a function $g \colon \Nb \imp \Nb$ such that $(\forall \sigma \in T)(\forall n < |\sigma|)(\sigma(n) < g(n))$.

\begin{definition}
$\mathsf{RbWKL}$ is the statement ``for every infinite, bounded subtree of $\Nb^{<\Nb}$, there is an infinite homogeneous partial function.''
\end{definition}

\begin{theorem}\label{thm-RbWKLisWKL}
$\rca \vdash \wkl \biimp \mathsf{RbWKL}$.\footnote{This theorem was obtained independently by Flood (personal communication).}
\end{theorem}
\begin{proof}
Over $\rca$, $\wkl$ implies $\mathsf{RbWKL}$ because $\wkl$ implies bounded K\"onig's lemma, which clearly implies $\mathsf{RbWKL}$.  Thus it suffices to show that $\mathsf{RbWKL}$ implies $\wkl$ over $\rca$.  This can be done by following the proof of Theorem~\ref{thm-RKLisACA}.  Let $T \subseteq \str$ be an infinite tree.  Let $(\tau_i)_{i \in \Nb}$ be the enumeration of $\str$ in length-lexicographic order, and let $g \colon \Nb \imp \Nb$ be a function such that $\forall n,i(|\tau_i| = n \imp i < g(n))$.  Define $S$ from $T$ as in Theorem~\ref{thm-RKLisACA}.  Then $(\forall \sigma \in S)(\forall i < |\sigma|)(\sigma(i) < g(i))$.  Thus $S$ is bounded by $g$.  The rest of the proof is similar to that of Theorem~\ref{thm-RKLisACA}.  
\end{proof}

We remark that it is not difficult to strengthen Theorem~\ref{thm-RbWKLisWKL} by fixing the function bounding the tree in the Ramsey-type bounded K\"onig's lemma instance to be an arbitrarily slow growing order function.  Indeed, $\rca$ proves the statement ``for every order function $g$, $\wkl$ if and only if Ramsey-type K\"onig's lemma holds for infinite subtrees of $\Nb^{<\Nb}$ bounded by $g$.''  However, as we will see next, it is not possible to replace an order function by a constant function.

\subsection{Locality and $k$-branching trees}\label{sect:locality-and-branching}
We analyze a notion of \emph{locality} together with Ramsey-type weak K\"onig's lemma for $k$-branching trees.  These notions aid our analysis of Ramsey-type analogs of other combinatorial principles.  Consider a function $f \colon [\Nb]^n \imp k$.  $\rt^n_k$ asserts the existence of an infinite homogeneous set $H \subseteq \Nb$.  However, for the purpose of some particular application, we may want the infinite homogeneous set $H$ to be a subset of some pre-specified infinite set $X \subseteq \Nb$.  This is the idea behind locality, and in such a situation we say that the $\rt^n_k$-instance $f$ has been \emph{localized} to $X$.  It is easy to see that $\rt^n_k$ proves that every $\rt^n_k$-instance can be localized to every infinite $X \subseteq \Nb$.  The following proposition is well-known and is often used implicitly, such as when proving $\rt^2_3$ from $\rt^2_2$.

\begin{proposition}\label{prop-lrt}
The following statements are equivalent over $\rca$:
\begin{itemize}
\item[(i)] $\rt^n_k$
\item[(ii)] For every $f \colon [\Nb]^n \imp k$ and every infinite $X \subseteq \Nb$, there is an infinite $H \subseteq X$ that is homogeneous for $f$.
\end{itemize}
\end{proposition}

\begin{proof}
Clearly $(ii) \imp (i)$, so it suffices to show that $(i) \imp (ii)$.  Let $f$ and $X$ be as in $(ii)$.  Let $(x_i)_{i \in \Nb}$ enumerate $X$ in increasing order.  Define $g \colon [\Nb]^n \imp k$ by $g(i_0, i_1, \dots, i_{n-1}) = f(x_{i_0}, x_{i_1}, \dots, x_{i_{n-1}})$ for increasing $n$-tuples $(i_0, i_1, \dots, i_{n-1})$.  Apply $\rt^n_k$ to $g$ to get an infinite $H_0 \subseteq \Nb$ that is homogeneous for $g$ with some color $c < k$.  Let $H = \{x_i : i \in H_0\}$.  Then $H \subseteq X$ is infinite, and $H$ is homogeneous for $f$ with color $c$ because if $x_{i_0} < x_{i_1} < \cdots < x_{i_{n-1}}$ are in $H$, then $i_0 < i_1 < \cdots < i_{n-1}$ are in $H_0$, hence $f(x_{i_0}, x_{i_1}, \dots, x_{i_{n-1}}) = g(i_0, i_1, \dots, i_{n-1}) = c$.
\end{proof}

By analogy with Proposition~\ref{prop-lrt}, we formulate $\lrwkl$, a localized variant of Ramsey-type weak K\"{o}nig's lemma.

\begin{definition} \index{Subsystems!$\lrwkl$}
$\lrwkl$ is the statement ``for every infinite tree $T \subseteq 2^{<\Nb}$ and every infinite $X \subseteq \Nb$, there is an infinite $H \subseteq X$ that is homogeneous for $T$.''
\end{definition}

\begin{lemma}\label{lem:rwkl-lrwkl}
$\rca \vdash \rwkl \biimp \lrwkl$.
\end{lemma}

\begin{proof}
Clearly $\rca \vdash \lrwkl \imp \rwkl$, so it suffices to prove that $\rca \vdash \rwkl \imp \lrwkl$.  Let $T \subseteq 2^{<\Nb}$ be an infinite tree and $X \subseteq \Nb$ be an infinite set.  Let $(x_i)_{i \in \Nb}$ enumerate $X$ in increasing order.  Let $S \subseteq 2^{<\Nb}$ be the set
\begin{align*}
S = \{\sigma \in 2^{<\Nb} : (\exists \tau \in T)(|\tau| = x_{|\sigma|} \andd (\forall i < |\sigma|)(\sigma(i) = \tau(x_i)))\}.
\end{align*}
$S$ exists by $\Delta^0_1$ comprehension, and $S$ is clearly closed under initial segments.  To see that $S$ is infinite, let $n \in \Nb$ and, as $T$ is infinite, let $\tau \in T$ have length $x_n$.  Then the $\sigma \in 2^n$ such that $(\forall i < n)(\sigma(i) = \tau(x_i))$ is a string in $S$ of length $n$.  Now apply $\rwkl$ to $S$ to get an infinite $H_0 \subseteq \Nb$ that is homogeneous for $S$ with some color $c < 2$.  Let $H = \{x_i : i \in H_0\}$.  $H$ is an infinite subset of $X$; we show that $H$ is homogeneous for $T$ with color $c$.  Given $n \in \Nb$, let $m \in H_0$ be such that $x_m > n$.  By the homogeneity of $H_0$ for $S$, let $\sigma \in S$ be of length $m$ and such that $(\forall i \in H_0)(i < |\sigma| \imp \sigma(i) = c)$.  By the definition of $S$, there is a $\tau \in T$ of length $x_m$ such that $(\forall i < |\sigma|)(\sigma(i) = \tau(x_i))$.  So if $x_i \in H$ is less than $|\tau| = x_m$, then $i$ is in $H_0$ and is less than $|\sigma| = m$, in which case $\tau(x_i) = \sigma(i) = c$.  Thus $H$ is homogeneous for $\tau$ with color $c$, and, as $|\tau| = x_m > n$, $\tau \restriction n$ is a string in $T$ of length $n$ for which $H$ is homogeneous with color $c$.
\end{proof}

Similarly, we can define a localized variant of Ramsey-type weak weak K\"onig's lemma.

\begin{definition} \index{Subsystems!$\lrwwkl$}
$\lrwwkl$ is the statement ``for every tree $T \subseteq 2^{<\Nb}$ of positive measure and every infinite $X \subseteq \Nb$, there is an infinite $H \subseteq X$ that is homogeneous for $T$.''
\end{definition}

\begin{theorem}\label{thm:lrwwkl-equiv}
The following statements are equivalent over $\rca$:
\begin{itemize}
  \item[(i)] $\dnr$
  \item[(ii)] $\rwwkl$
  \item[(iii)] $\lrwwkl$.
\end{itemize}
\end{theorem}
\begin{proof}
Theorem~\ref{thm:dnr-rwwkl} states  that $(i) \biimp (ii)$, and $(iii) \imp (ii)$ is clear.  To see that $(ii) \imp (iii)$, we need only check that the tree $S$ constructed in Lemma~\ref{lem:rwkl-lrwkl} has positive measure when the tree $T$ has positive measure.  To this end, notice that for every $k \in \Nb$, 
\begin{align*}
\card{\set{\sigma \in S : \card{\sigma} = k}} \geq \frac{\card{\set{\tau \in T : \card{\tau} = x_k}}}{2^{x_k - k}}.
\end{align*}
Thus
\begin{align*}
\frac{\card{\set{\sigma \in S : \card{\sigma} = k}}}{2^k} \geq \frac{\card{\set{\tau \in T : \card{\tau} = x_k}}}{2^{x_k}},
\end{align*}
which implies that $S$ has positive measure if $T$ has positive measure.
\end{proof}

Using $\lrwkl$, we prove variants of $\rwkl$ and $\lrwkl$ for $k$-branching trees.  Define a set $H \subseteq \Nb$ to be \emph{homogeneous for a string $\sigma \in k^{<\Nb}$ with color $c < k$} and a set $H \subseteq \Nb$ to be \emph{homogeneous for an infinite tree $T \subseteq k^{<\Nb}$} as in Definition~\ref{def-RWKL} but with $k$ in place of $2$.

\index{Subsystems!$\rwkl_k$}
\index{Subsystems!$\lrwkl_k$}
\begin{definition}{\ }
\begin{itemize}
\item $\rwkl_k$ is the statement ``for every infinite tree $T \subseteq k^{<\Nb}$, there is an infinite $H \subseteq \Nb$ that is homogeneous for $T$.''
\item $\lrwkl_k$ is the statement ``for every infinite tree $T \subseteq k^{<\Nb}$ and every infinite $X \subseteq \Nb$, there is an infinite $H \subseteq X$ that is homogeneous for $T$.''
\end{itemize}
\end{definition}

\begin{lemma}\label{lem:rwkl-rwklk}
For every $k \in \omega$, $\rca \vdash \lrwkl \rightarrow \rwkl_k$.
\end{lemma}

\begin{proof}
If $j < k$ then $\rca \vdash \rwkl_k \imp \rwkl_j$ by identifying $j^{<\Nb}$ with the obvious subtree of $k^{<\Nb}$.  It therefore suffices to show that, for every $k \in \omega$, $\rca \vdash \lrwkl \imp \rwkl_{2^k}$.

Let $T \subseteq (2^k)^{<\Nb}$ be an infinite tree.  The idea of the proof is to code $T$ as a subtree of $2^{<\Nb}$ by coding each number less than $2^k$ by its binary expansion.  We then obtain a homogeneous set for $T$ by using $k$ applications of $\lrwkl$.

For each $a < 2^k$ and each $i < k$, let $a(i) < 2$ denote the $(i+1)$\textsuperscript{th} digit in the binary expansion of $a$.  Then to each $\sigma \in (2^k)^{<\Nb}$ associate a string $\tau_\sigma \in 2^{<\Nb}$ of length $k|\sigma|$ by $\tau_\sigma(ki+j) = \sigma(i)(j)$ (i.e., the $j$\textsuperscript{th} digit in the binary expansion of $\sigma(i)$) for all $i < |\sigma|$ and all $j < k$.  We define infinite trees $2^{<\Nb} \supseteq S_0 \supseteq S_1 \supseteq \cdots \supseteq S_{k-1}$, and, for each $i < k$, we find an infinite set $H_i$ homogeneous for $S_i$.  Moreover, the sets $H_i$ will be such that $(\forall i < k)(\forall n \in H_i)(n \equiv i \mod k)$ and $(\forall i < k-1)(\forall n)(n \in H_{i+1} \imp n-1 \in H_i)$.  Let $S_0 = \{\tau \in 2^{<\Nb} : (\exists \sigma \in T)(|\sigma| = \lceil |\tau| / k\rceil \andd \tau \subseteq \tau_\sigma)\}$.  That is, $S_0$ consists of the substrings of the binary expansions of the strings in $T$.  $S_0$ exists by $\Delta^0_1$ comprehension, $S_0$ is clearly a tree, and $S_0$ is infinite because if $n \in \Nb$ and $\sigma \in T$ has length $n$, then $\tau_\sigma \restriction n$ is a member of $S_0$ of length $n$.  Let $X_0 = \{n \in \Nb : n \equiv 0 \mod k\}$.  Apply $\lrwkl$ to $S_0$ and $X_0$ to get an infinite set $H_0 \subseteq X_0$ and a color $c_0 < 2$ such that $H_0$ is homogeneous for $S_0$ with color $c_0$.  Now suppose that $S_\ell$, $H_\ell$, and $c_\ell$ are defined for some $\ell < k-1$.  Let $S_{\ell + 1} = \{\tau \in S_\ell : (\forall j < |\tau|)(j \in H_\ell \imp \tau(j) = c_\ell)\}$.  $S_{\ell+1}$ exists by $\Delta^0_1$ comprehension, it is easy to check that $S_{\ell+1}$ is a tree, and $S_{\ell+1}$ is infinite because $H_\ell$ is homogeneous for $S_\ell$ with color $c_\ell$.  Let $X_{\ell+1} = \{n+1 : n \in H_\ell\}$, and note that $(\forall n \in X_{\ell+1})(n \equiv \ell+1 \mod k)$ because $(\forall n \in H_\ell)(n \equiv \ell \mod k)$.  Apply $\lrwkl$ to $S_{\ell+1}$ and $X_{\ell+1}$ to get an infinite set $H_{\ell+1} \subseteq X_{\ell+1}$ and a color $c_{\ell+1} < 2$ such that $H_{\ell+1}$ is homogeneous for $S_{\ell+1}$ with color $c_{\ell+1}$.  By choice of $X_{\ell+1}$, we also have that $(\forall n \in H_{\ell+1})(n \equiv \ell+1 \mod k)$ and that $\forall n(n \in H_{\ell+1} \imp n-1 \in H_\ell)$.

Once $S_i$, $H_i$, and $c_i$ are defined for all $i < k$, let $H = \{n : kn+(k-1) \in H_{k-1}\}$ and let $a < 2^k$ be the number whose binary expansion is $c_0c_1 \cdots c_{k-1}$.  We show that $H$ is homogeneous for $T$ with color $a$.  Given $n \in \Nb$, let $\tau \in S_{k-1}$ be of length $kn$ and such that $H_{k-1}$ is homogeneous for $\tau$.  Let $\sigma \in (2^k)^{<\Nb}$ be such that $\tau = \tau_\sigma$.  As $\tau \in S_{k-1} \subseteq S_0$, it must be that $\sigma \in T$ by the definition of $S_0$.  It remains to show that $(\forall i \in H)(i < |\sigma| \imp \sigma(i) = a)$.  Consider $i \in H$ with $i < |\sigma|$.  The binary expansion of $\sigma(i)$ is $\tau(ki)\tau(ki+1) \cdots \tau(ki+(k-1))$, and $ki + (k-1) \in H_{k-1}$ by the definition of $H$.  Thus, $\tau(ki+(k-1)) = c_{k-1}$ because $H_{k-1}$ is homogeneous for $\tau$.  Now let $\ell$ be such that $0 \leq \ell < k-1$.  Then $ki+\ell \in H_\ell$ because $ki+(k-1) \in H_{k-1}$ and $(\forall i < k-1)(\forall m)(m \in H_{i+1} \imp m-1 \in H_i)$.  Thus $\tau(ki+j) = c_\ell$ because $\tau \in S_{k-1} \subseteq S_{\ell+1}$, and $S_{\ell+1}$ was chosen so that if $\eta \in S_{\ell+1}$ and $m < |\eta|$ is in $H_\ell$, then $\eta(m) = c_\ell$.  Thus the binary expansion of $\sigma(i)$ is $c_0c_1 \cdots c_{k-1}$, so $\sigma(i) = a$ as desired.
\end{proof}

Thus we have the following equivalences.
\begin{theorem}\label{thm:lrwkl-rwkl}
For every $k \in \omega$ with $k \geq 2$, the following statements are equivalent over $\rca$:
\smallskip
\begin{itemize}
\begin{minipage}{0.3\linewidth}  
  \item[(i)] $\rwkl$
  \item[(ii)] $\lrwkl$
\end{minipage}
\begin{minipage}{0.6\linewidth}
  \item[(iii)] $\rwkl_k$
  \item[(iv)] $\lrwkl_k$.
\end{minipage}
\end{itemize}
\end{theorem}
\begin{proof}
Lemma~\ref{lem:rwkl-lrwkl} states that $(i) \biimp (ii)$.  Lemma~\ref{lem:rwkl-rwklk} states that $(ii) \imp (iii)$.  A proof analogous to that of Lemma~\ref{lem:rwkl-lrwkl} shows that $(iii) \biimp (iv)$.  Clearly $(iv) \imp (i)$ when $k \geq 2$.
\end{proof}

The statement $\forall k \rwkl_k$ easily implies $\rt^1$ over $\rca$, and $\rt^1$ is equivalent to $\bst$ over $\rca$ (this equivalence is due to Hirst~\cite{Hirst1987}).  To see that $\rca \vdash \forall k \rwkl_k \imp \rt^1$, given a function $f \colon \Nb \imp k$, define the tree $T \subseteq k^{<\Nb}$ by $T = \{f \restriction n : n \in \Nb\}$.  Then $H$ is homogeneous for $T$ if and only if $H$ is homogeneous for $f$.  Thus $\wklz$ does not prove $\forall k \rwkl_k$ because $\wklz$ does not prove $\bst$.  (It is well-known that $\wklz$ is $\Pi^1_1$-conservative over $\rca$ and that $\rca$ does not prove $\bst$.  See \cite{Simpson2009}~Corollary~IX.2.6 and \cite{Hajek1998}~Section~IV.1.)  However, it is easy to see that $\wklz + \bst$ proves $\forall k \rwkl_k$.  Moreover, $\rca + \forall k \srt^2_k$ proves $\forall k \rwkl_k$ by essentially same argument used for $k=2$ in~\cite{Flood2012} Theorem~5.

\begin{question}\label{qu-SRT22vsRWKLk}
Does $\rca \vdash \srt^2_2 \imp \forall k \rwkl_k$?
\end{question}

The strength of having various kinds of homogeneous sets for various kinds of infinite trees is summarized in Table~\ref{tbl:path-and-hom}.  The columns correspond to the kinds of trees allowed, whereas the rows correspond to the kinds of homogeneous sets asserted to exist.  The first column considers infinite, finitely branching trees.  The second column restricts to trees whose nodes have at most two immediate successors.  The third column restricts to trees whose branching is bounded by some function.  The fourth column restricts to trees whose branching is bounded by a constant function.  The last column restricts to binary trees of positive measure.  The first row corresponds to K\"onig-like statements, that is, statements asserting the existence of paths through the tree.  The second row asserts the existence of everywhere-packed homogeneous sets.  The third row asserts the existence of sets that are homogeneous for a fixed color.  The fourth row asserts the existence of homogeneous sets that are contained in a prescribed infinite set.  The last row asserts the existence of homogeneous sets.

\def\arraystretch{1.5}
\begin{table}[htbp]
\begin{center}
\setlength\tabcolsep{6pt}
\begin{tabular}{|c|c|c|c|c|c|} \hline
 Tree & fin.\ branch.\ & \hspace{10pt}2-ary\hspace{10pt} & bounded & $k$-bounded & pos.\ meas.\ \\ \hline
 path & \multirow{5}{*}{$\aca$} & \multirow{5}{*}{$\aca$ \tiny{[\ref{thm-RKLisACA}]}} &
\multirow{5}{*}{$\wkl$ \tiny{[\ref{thm-RbWKLisWKL}]}} 
& \multirow{3}{*}{$\wkl$ \tiny{[\ref{thm:fixed-rwkl-wkl}, \ref{thm:packed-rwkl-wkl}]}} &
$\wwkl$ \\ \cline{1-1} \cline{6-6}
packed hom.\ & & & & & 
? \\
\cline{1-1} \cline{6-6}
 hom.\ fixed color & & & & & \multirow{3}{*}{$\dnr$ \tiny{[\ref{thm:dnr-rwwkl}, \ref{thm:lrwwkl-equiv}]}}\\ \cline{1-1}
\cline{5-5}
 local hom.\ & & & & 
\multirow{2}{*}{$\rwkl$ \tiny{[\ref{thm:lrwkl-rwkl}]}}  & \\
\cline{1-1}
 hom.\ & & & & 
 & \\ \hline
\end{tabular}
\caption{Paths and homogeneous sets existence for classes of trees}
\label{tbl:path-and-hom}
\end{center}
\end{table}

The question mark in Table~\ref{tbl:path-and-hom} indicates that we did not study principles asserting that trees of positive measure have everywhere-packed homogeneous functions.  It would be interesting to determine the proper analog of Theorem~\ref{thm:packed-rwkl-wkl} when the trees in item~(ii) are required to have positive measure.

\section{The strength of Ramsey-type satisfiability principles}\label{sec-RSAT}

One can conceivably consider a Ramsey-type variant of any $\Pi^1_2$ statement $\forall X \exists Y \varphi(X,Y)$ so long as one can provide a reasonable formulation of what it means for a set $Z$ to be consistent with a $Y$ such that $\varphi(X,Y)$.  For example, in the case of $\rwkl$, we think of a set $H$ as being consistent with a path through an infinite tree $T \subseteq \str$ if $H$ is homogeneous for $T$.  We are interested in analyzing the strengths of Ramsey-type variants of statements that are equivalent to $\wkl$ over $\rca$.  Several such statements have trivial Ramsey-type variants.  For example, $\rca$ proves that for every pair of injections $f,g \colon \Nb \imp \Nb$ with disjoint ranges, there is an infinite set $X$ consistent with being a separating set for the ranges of $f$ and $g$ because $\rca$ proves that there is an infinite subset of the range of $f$.  The obvious Ramsey-type variant of Lindenbaum's lemma (every consistent set of sentences has a consistent completion) is also easily seen to be provable in $\rca$.  For the remainder of this paper, we consider non-trivial Ramsey-type variants of the compactness theorem for propositional logic and of graph coloring theorems.  Many of these variants are equivalent to $\rwkl$, which we take as evidence that $\rwkl$ is robust.

\begin{definition}
A set $C$ of propositional formulas is \emph{finitely satisfiable} if every finite $C_0 \subseteq C$ is satisfiable (i.e., has a satisfying truth assignment).  We denote by $\sat$ the compactness theorem for propositional logic, which is the statement ``every finitely satisfiable set of propositional formulas is satisfiable.''
\end{definition}

It is well-known that $\sat$ is equivalent to $\wkl$ over $\rca$ (see \cite{Simpson2009}~Theorem IV.3.3).

If $C$ is a set of propositional formulas, then let $\atoms(C)$ denote the set of propositional atoms appearing in the formulas in $C$.  Strictly speaking, $\rca$ does not prove that $\atoms(C)$ exists for every set of propositional formulas $C$.  However, in $\rca$ we can rename the atoms appearing in a set of propositional formulas $C$ in such a way as to produce an equivalent set of propositional formulas $C'$ for which $\atoms(C')$ does exist.  Indeed, we may assume that $\atoms(C) = \Nb$ whenever $\atoms(C)$ is infinite.  Thus for ease of mind we always assume that $\atoms(C)$ exists as a set.

\begin{definition}
Let $C$ be a set of propositional formulas.  A set $H \subseteq \atoms(C)$ is \emph{homogeneous for $C$} if there is a $c \in \{\true, \false\}$ such that every finite $C_0 \subseteq C$ is satisfiable by a truth assignment $\nu$ such that $(\forall a \in H)(\nu(a) = c)$.
\end{definition}

As is typical, we identify $\true$ with $1$ and $\false$ with $0$.

\begin{definition}{\ }
\begin{itemize}
\item $\rsat$ is the statement ``for every finitely satisfiable set $C$ of propositional formulas with $\atoms(C)$ infinite, there is an infinite $H \subseteq \atoms(C)$ that is homogeneous for $C$.''

\item $\lrsat$ is the statement ``for every finitely satisfiable set $C$ of propositional formulas with $\atoms(C)$ infinite and every infinite $X \subseteq \atoms(C)$, there is an infinite $H \subseteq X$ that is homogeneous for $C$.''
\end{itemize}
\end{definition}

We also consider r.e.\ variants of $\rsat$ and $\lrsat$, denoted $\textup{r.e.-}\rsat$ and $\textup{r.e.-}\lrsat$, obtained by replacing the finitely satisfiable set of propositional formulas $C$ by a list of propositional formulas $(\varphi_i)_{i \in \Nb}$ such that $\{\varphi_i : i < n\}$ is satisfiable for every $n \in \Nb$.  This amounts to considering r.e.\ sets of propositional formulas instead of recursive sets of propositional formulas.  In this situation, we may still assume that $\atoms((\varphi_i)_{i \in \Nb})$ (the set of propositional atoms appearing in the $\varphi_i$'s) exists as a set.

We first show that $\rca \vdash \rsat \imp \rwkl$.  In fact, we show that the restriction of $\rsat$ to what we call \emph{$2$-branching} clauses implies $\rwkl$ over $\rca$.  This technical restriction is useful for the proof of Theorem~\ref{thm:rcolor3-rsat} in our analysis of Ramsey-type graph coloring principles.

Recall that a propositional formula $\ell$ is called a \emph{literal} if either $\ell = a$ or $\ell = \neg a$ for some propositional atom $a$ and that a \emph{clause} is a disjunction of literals.

\begin{definition}
Let $\{a_i : i \in \Nb\}$ be an infinite set of propositional atoms.  A set $C$ of clauses is called \emph{$2$-branching} if, for every clause $\ell_0 \orr \ell_1 \orr \cdots \orr \ell_{n-1} \in C$ and every $i < n$, the literal $\ell_i$ is either $a_i$ or $\neg a_i$.  $\rsat_{\textup{2-branching}}$ is $\rsat$ restricted to $2$-branching clauses.
\end{definition}

\begin{proposition}\label{prop:RWKL2BranchRSAT}
$\rca \vdash \rsat_{\textup{2-branching}} \imp \rwkl$.
\end{proposition}
\begin{proof}
Let $A = \{a_i : i \in \Nb\}$ be a set of propositional atoms, and to each string $\sigma \in 2^{<\Nb}$ associate the clause $\theta_\sigma = \bigvee_{i<|\sigma|}\ell_i$, where $\ell_i = a_i$ if $\sigma(i) = 0$ and $\ell_i = \neg a_i$ if $\sigma(i) = 1$.  Let $T \subseteq 2^{<\Nb}$ be an infinite tree.  Let $C = \{\theta_\sigma : \sigma \notin T\}$, and observe that $C$ is $2$-branching.  We show that $C$ is finitely satisfiable.  Given $C_0 \subseteq C$ finite, choose $n$ large enough so that the atoms appearing in the clauses in $C_0$ are among $\{a_i : i < n\}$.  As $T$ is infinite, choose a $\tau \in T$ of length $n$.  Define a truth assignment $t \colon \{a_i : i < n\} \imp \{\true, \false\}$ by $t(a_i) = \tau(i)$.  Now, if $\theta$ is a clause in $C_0$, then $\theta = \theta_\sigma = \bigvee_{i<|\sigma|}\ell_i$ for some $\sigma \notin T$ with $|\sigma| < n$.  Thus there is an $i < n$ such that $\sigma(i) \neq \tau(i)$ (because $\tau \in T$ and $\sigma \notin T$), from which we see that $t(\ell_i)=\true$ and hence that $t(\theta_\sigma)=\true$.  Thus $t$ satisfies $C_0$.

By $\rsat_{\textup{2-branching}}$, let $H_0 \subseteq A$ and $c \in \{\true,\false\}$ be such that $H_0$ is homogeneous for $C$ with truth value $c$.  Let $H = \{i \in \Nb : a_i \in H_0\}$.  We show that $H$ is homogeneous for a path through $T$ with color $c$.  Given $n \in \Nb$, we want to find a $\tau \in T$ such that $|\tau| = n$ and $(\forall i < |\tau|)(i \in H \imp \tau(i) = c)$.  Thus let $t \colon \{a_i : i < n\} \imp \{\true, \false\}$ be a truth assignment satisfying $C_0 = \{\theta_\sigma : \sigma \notin T \andd |\sigma| = n\}$ such that $(\forall a \in \{a_i : i < n\} \cap H_0)(t(a) = c)$.  Let $\tau \in 2^n$ be defined by $\tau(i) = t(a_i)$ for all $i<n$.  Notice that $(\forall i < |\tau|)(i \in H \imp \tau(i) = c)$ and that $t(\theta_\tau) = \false$.  If $\tau \notin T$, then $\theta_\tau \in C_0$, contradicting that $t$ satisfies $C_0$.  Thus $\tau \in T$ as desired.
\end{proof}

\begin{proposition}\label{prop:lrwkl-rersat}
$\rca \vdash \lrwkl \imp \textup{r.e.-}\lrsat$.
\end{proposition}

\begin{proof}
Let $(\varphi_i)_{i \in \Nb}$ be a list of propositional formulas over an infinite set of atoms $A$ such that $\{\varphi_i : i < n\}$ is satisfiable for every $n \in \Nb$, and let $X \subseteq A$ be infinite.  Let $(a_i)_{i \in \Nb}$ enumerate $A$.  For each $\sigma \in 2^{<\Nb}$, identify $\sigma$ with the truth assignment $\nu_\sigma$ on $\{a_i : i < |\sigma|\}$ given by $(\forall i < |\sigma|)(\nu_\sigma(a_i) = \true \biimp \sigma(i) = 1)$.  Let $T \subseteq 2^{<\Nb}$ be the tree
\begin{align*}
T = \{\sigma \in 2^{<\Nb} : \neg(\exists i < |\sigma|)(\nu_\sigma(\varphi_i) = \false)\},
\end{align*}
where $\nu_\sigma(\varphi_i)$ is the truth value assigned to $\varphi_i$ by $\nu_\sigma$ (we consider $\nu_\sigma(\varphi_i)$ to be undefined---hence not $\false$---if $\varphi_i$ contains an atom $a_m$ for an $m \geq |\sigma|$).  $T$ exists by $\Delta^0_1$ comprehension and is closed downward.  $T$ is infinite because for any $n \in \Nb$, any satisfying truth assignment of $\{\varphi_i : i < n\}$ restricted to $\{a_i : i<n\}$ yields a string in $T$ of length $n$.  Let $X_0 = \{i \in \Nb : a_i \in X\}$, and, by $\lrwkl$, let $H_0 \subseteq X_0$ and $c < 2$ be such that $H_0$ is infinite and homogeneous for $T$ with color $c$.  Let $H = \{a_i : i \in H_0\}$ and note that it is an infinite subset of $X$.  We show that, for every $n \in \Nb$, $\{\varphi_i : i < n\}$ can be satisfied by a truth assignment $\nu$ such that $(\forall a \in H)(\nu(a) = c)$.  Let $n \in \Nb$, and let $m$ be large enough so that $\atoms(\{\varphi_i : i < n\}) \subseteq \{a_i : i < m\}$.  Let $\sigma \in T$ be such that $|\sigma| = m$ and $H_0$ is homogeneous for $\sigma$ with color $c$.  Then $(\forall i < n)(\nu_\sigma(\varphi_i) = \true)$ because $\nu_\sigma(\varphi_i)$ is defined for all $i < n$ and $\nu_\sigma(\varphi_i) \neq \false$ for all $i < n$.  Thus $\nu_\sigma$ satisfies $\{\varphi_i : i < n\}$, and, because $H_0$ is homogeneous for $\sigma$ with color $c$, $(\forall a \in H)(\nu_\sigma(a) = c)$.
\end{proof}

\begin{theorem}
The following statements are equivalent over $\rca$:
\smallskip
\begin{itemize}
\begin{minipage}{0.3\linewidth}  
  \item[(i)] $\rwkl$
  \item[(ii)] $\rsat$
  \item[(iii)] $\lrsat$
\end{minipage}
\begin{minipage}{0.6\linewidth}
  \item[(iv)] $\textup{r.e.-}\rsat$
  \item[(v)] $\textup{r.e.-}\lrsat$.
\end{minipage}
\end{itemize}
\end{theorem}

\begin{proof}
Clearly $(v) \imp (iii) \imp (ii)$ and $(v) \imp (iv) \imp (ii)$, so it suffices to show the equivalence of $(i)$, $(ii)$, and $(v)$.  We have that $(i) \imp (v)$ by Proposition~\ref{prop:lrwkl-rersat} and Lemma~\ref{lem:rwkl-lrwkl}, that $(v) \imp (ii)$ is clear, and that $(ii) \imp (i)$ by Proposition~\ref{prop:RWKL2BranchRSAT}.
\end{proof}

\section{Ramsey-type graph coloring principles}\label{sec-RCOLOR}

Let $k \in \Nb$, and let $G = (V, E)$ be a graph.  A function $f \colon V \imp k$ is a \emph{$k$-coloring} of $G$ if $(\forall x,y \in V)((x,y) \in E \imp f(x) \neq f(y))$.  A graph is \emph{$k$-colorable} if it has a $k$-coloring, and a graph is \emph{locally $k$-colorable} if every finite subgraph is $k$-colorable.  A simple compactness argument proves that every locally $k$-colorable graph is $k$-colorable.  In the context of reverse mathematics, we have the following well-known equivalence.

\begin{theorem}[see~\cite{Hirst1990}]\label{thm-ColoringIsWKL}
For every $k \in \omega$ with $k \geq 2$, the following statements are equivalent over $\rca$:
\begin{itemize}
\item[(i)] $\wkl$
\item[(ii)] Every locally $k$-colorable graph is $k$-colorable.
\end{itemize}
\end{theorem}

In light of Theorem~\ref{thm-ColoringIsWKL}, we define Ramsey-type analogs of graph coloring principles and compare them to Ramsey-type weak K\"onig's lemma.

\index{Subsystems!$\rcolor_k$}
\index{Subsystems!$\lrcolor_k$}
\begin{definition}{\ }
\begin{itemize}
\item Let $G = (V,E)$ be a graph.  A set $H \subseteq V$ is {\itshape $k$-homogeneous for $G$} if every finite $V_0 \subseteq V$ induces a subgraph that is $k$-colorable by a coloring that colors every vertex in $V_0 \cap H$ color $0$.  We often write {\itshape homogeneous} for $k$-homogeneous when the $k$ is clear from context.

\item $\rcolor_k$ is the statement  ``for every infinite, locally $k$-colorable graph $G = (V,E)$, there is an infinite $H \subseteq V$ that is $k$-homogeneous for $G$.''

\item $\lrcolor_k$ is the statement ``for every infinite, locally $k$-colorable graph $G = (V,E)$ and every infinite $X \subseteq V$, there is an infinite $H \subseteq X$ that is $k$-homogeneous for $G$.''
\end{itemize}
\end{definition}

The goal of this section is to obtain the analog of Theorem~\ref{thm-ColoringIsWKL} with $\rwkl$ in place of $\wkl$ and with $\rcolor_k$ in place of the statement ``every locally $k$-colorable graph is $k$-colorable.''  We are able to obtain this analog for all standard $k \geq 3$ instead of all standard $k \geq 2$.  The case $k = 2$ remains open.  Showing the forward direction, that $\rca \vdash \rwkl \rightarrow \rcolor_k$ (indeed, that $\rca \vdash \rwkl \rightarrow \lrcolor_k$), is straightforward.

\begin{lemma}\label{RWKLprovesLRCOLOR}
For every $k \in \omega$, $\rca \vdash \rwkl \rightarrow \lrcolor_k$.
\end{lemma}

\begin{proof}
Let $G = (V,E)$ be an infinite graph such that every finite $V_0 \subseteq V$ induces a $k$-colorable subgraph, and let $X \subseteq V$ be infinite.  Enumerate $V$ as $(v_i)_{i \in \Nb}$, and let $T \subseteq k^{<\Nb}$ be the tree
\begin{align*}
T = \{\sigma \in k^{<\Nb} : (\forall i,j < |\sigma|)((v_i,v_j) \in E \imp \sigma(i) \neq \sigma(j))\}.
\end{align*}
$T$ exists by $\Delta^0_1$ comprehension and is closed downward.  $T$ is infinite because for any $n \in \Nb$, any $k$-coloring of the subgraph induced by $\{v_i : i < n\}$ corresponds to a string in the tree of length $n$.  Let $X_0 = \{i \in \Nb : v_i \in X\}$, and apply $\lrwkl_k$ (which follows from $\rca + \rwkl$ by Theorem~\ref{thm:lrwkl-rwkl}) to $T$ and $X_0$ to get an infinite set $H_0 \subseteq X_0$ and a color $c < k$ such that $H_0$ is homogeneous for a path through $T$ with color $c$.  Let $H = \{v_i : i \in H_0\}$.  We show that every finite $V_0 \subseteq V$ induces a subgraph that is $k$-colorable by a coloring that colors every $v \in V_0 \cap H$ color $0$.  Let $V_0 \subseteq V$ be finite, let $n = \max\{i+1 : v_i \in V_0\}$, and let $\sigma \in T$ be such that $|\sigma| = n$ and such that $H_0$ is homogeneous for $\sigma$ with color $c$.  Then the coloring of $V_0$ given by $v_i \mapsto \sigma(i)$ is a $k$-coloring of $V_0$ that colors the elements of $V_0 \cap H$ color $c$.  Swapping colors $0$ and $c$ thus gives a $k$-coloring of $V_0$ that colors the elements of $V_0 \cap H$ color $0$.
\end{proof}

We now prove that $\rca \vdash \rcolor_3 \imp \rwkl$ (Theorem~\ref{thm:rcolor3-rsat} below).  Our proof factors through the Ramsey-type satisfiability principles and is a rather elaborate exercise in circuit design.  The plan is to prove that $\rca \vdash \rcolor_3 \imp \rsat_{\textup{2-branching}}$, then appeal to Proposition~\ref{prop:RWKL2BranchRSAT}.  Given a $2$-branching set of clauses $C$, we compute a locally $3$-colorable graph $G$ such that every set homogeneous for $G$ computes a set that is homogeneous for $C$.  $G$ is built by connecting $\emph{widgets}$, which are finite graphs whose colorings have desirable properties.  A widget $W(\vec v)$ has distinguished vertices $\vec v$ through which we connect the widget to the larger graph.  These distinguished vertices can also be regarded, in a sense, as the inputs and outputs of the widget.

In an $\rcolor_3$ instance built out of widgets according to an $\rsat_{\textup{2-branching}}$ instance, some of the vertices code literals so that the colorings of these coding vertices code truth assignments of the corresponding literals in such a way that a homogeneous set for the $\rsat_{\textup{2-branching}}$ instance can be decoded from a homogeneous set for the graph that contains only coding vertices.  However, we have no control over what vertices appear in an arbitrary homogeneous set.  Therefore, we must build our graph so that the color of \emph{every} vertex gives information about the color of \emph{some} coding vertex.

When we introduce a widget, we prove a lemma concerning the three key aspects of the widget's operation:  soundness, completeness, and reversibility.  By soundness, we mean conditions on the $3$-colorings of the widget, which we think of as input-output requirements for the widget.  By completeness, we mean that the widget is indeed $3$-colorable and, moreover, that $3$-colorings of certain sub-widgets extend to $3$-colorings of the whole widget.  By reversibility, we mean that the colors of some vertices may be deduced from the colors of other vertices.

To aid the analysis of our widgets, we introduce a notation for the property that a coloring colors two vertices the same color.

\begin{notation}
Let $G = (V,E)$ be a graph, let $a, b \in V$, and let $\nu \colon V \imp k$ be a $k$-coloring of $G$.  We write $a =_\nu b$ if $\nu(a) = \nu(b)$.
\end{notation}

The graph $G$ that we build from our widgets has three distinguished vertices, $0$, $1$, and $2$, connected as a triangle.  The intention of these vertices is to code truth values.  If $v$ is a vertex coding a literal $\ell$, then $(v,2)$ is an edge in $G$, and, for a $3$-coloring $\nu$, we interpret $v =_\nu 0$ as $\ell$ is false and $v =_\nu 1$ as $\ell$ is true.  Our widgets often include vertices $0$, $1$, and $2$.

\begin{widget}
$R_{\substack{x \mapsto y \\ y \mapsto z}}(a,u)$ is the following widget.
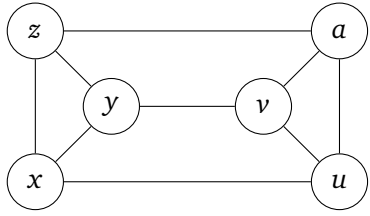
\begin{center}
\begin{tikzpicture}[x=1cm, y=1cm, node/.style={circle, draw, minimum size=2em}, widget/.style={rectangle, draw, minimum size=2em}]
	
	\node[node] (x) at (0, 0) {$x$};
	\node[node] (y) at (1, 1) {$y$};
	\node[node] (z) at (0, 2) {$z$};
	
	\node[node] (v) at (3, 1) {$v$};
	\node[node] (a) at (4, 2) {$a$};
	\node[node] (u) at (4, 0) {$u$};
		
	\draw (x) -- (y) -- (z) -- (x);	
	\draw (a) -- (u) -- (v) -- (a);

	\draw (x) -- (u);
	\draw (y) -- (v);
	\draw (z) -- (a);
		
\end{tikzpicture}
\end{center}
\end{widget}

\begin{lemma}\label{lem:RWidget}{\ }
\begin{itemize}
\item[(i)]  Let $\nu$ be a $3$-coloring of $R_{\substack{x \mapsto y \\ y \mapsto z}}(a,u)$.  If $a =_\nu x$ then $u =_\nu y$, and if $a =_\nu y$ then $u =_\nu z$.

\item[(ii)] Every $3$-coloring of the subgraph of $R_{\substack{x \mapsto y \\ y \mapsto z}}(a,u)$ induced by $\{x,y,z,a\}$ can be extended to a $3$-coloring of $R_{\substack{x \mapsto y \\ y \mapsto z}}(a,u)$.

\item[(iii)] In every $3$-coloring of $R_{\substack{x \mapsto y \\ y \mapsto z}}(a,u)$, the color of each vertex in $\{u, v\}$ determines the color of $a$.
\end{itemize}
\end{lemma}

\begin{proof}
The lemma follows from examining the two possible (up to permutations of the colors) $3$-colorings of $R_{\substack{x \mapsto y \\ y \mapsto z}}(a,u)$:
\begin{align}
a &=_\nu x & v &=_\nu z & u &=_\nu y\\
a &=_\nu y & v &=_\nu x & u &=_\nu z.
\end{align}
We see $(i)$ immediately.  For $(ii)$, if $a =_\nu x$, then color the widget according to the first coloring; and if $a =_\nu y$, then color the widget according to the second coloring.  For $(iii)$, if $u =_\nu y$ or $v =_\nu z$, then $a =_\nu x$; and if $u =_\nu z$ or $v =_\nu x$, then $a =_\nu y$.
\end{proof}

The intention is that, in $R_{\substack{x \mapsto y \\ y \mapsto z}}(a,u)$, the vertices $x$, $y$, and $z$ are some permutation of the vertices $0$, $1$, and $2$.  For example, $R_{\substack{0 \mapsto 1 \\ 1 \mapsto 2}}(a,u)$ is the instance of this widget where $x=0$, $y=1$, and $z=2$.  The notation `$R_{\substack{0 \mapsto 1 \\ 1 \mapsto 2}}(a,u)$' is evocative of Lemma~\ref{lem:RWidget}~$(i)$.  Thinking of $a$ as the widget's input and of $u$ as the widget's output, Lemma~\ref{lem:RWidget}~$(i)$ says that the widget maps $0$ to $1$ and maps $1$ to $2$.

\begin{widget}
$U_{x,y,z}(\ell,b,u)$ is the following widget.
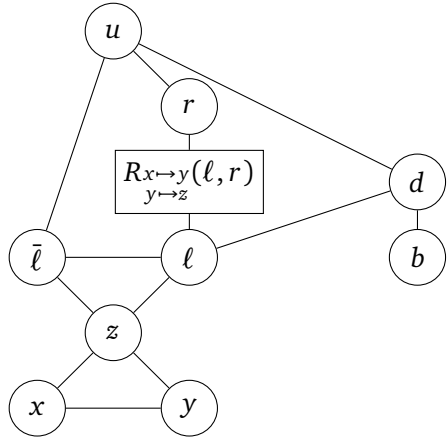
\begin{center}
\begin{tikzpicture}[x=1cm, y=1cm, node/.style={circle, draw, minimum size=2em}, widget/.style={rectangle, draw, minimum size=2em}]
	
	\node[node] (x) at (-1, -1) {$x$};
	\node[node] (y) at (1, -1) {$y$};
	\node[node] (z) at (0, -0) {$z$};
	
	\node[node] (l) at (1, 1) {$\ell$};
	\node[node] (lb) at (-1,1) {$\bar\ell$};
	\node[node] (b) at (4, 1) {$b$};
	
	\node[widget] (R) at (1, 2) {$R_{\substack{x \mapsto y \\ y \mapsto z}}(\ell, r)$};	
	\node[node] (r) at (1, 3) {$r$};
	\node[node] (d) at (4, 2) {$d$};
	\node[node] (u) at (0, 4) {$u$};
		
	\draw (x) -- (y) -- (z) -- (x);
	\draw (z) -- (l) -- (lb) -- (z);
	\draw (lb) -- (u);
	\draw (l) -- (R) -- (r) -- (u);
	\draw (l) -- (d) -- (u);
	\draw (b) -- (d);
	
\end{tikzpicture}
\end{center}
\end{widget}

In the diagram above, the box labeled `$R_{\substack{x \mapsto y \\ y \mapsto z}}(\ell, r)$' represents an $R_{\substack{x \mapsto y \\ y \mapsto z}}(\ell, r)$ sub-widget.  The vertices $\ell$ and $r$ are the same as those appearing inside $R_{\substack{x \mapsto y \\ y \mapsto z}}(\ell, r)$.  They have been displayed to show how they connect to the rest of the $U_{x,y,z}(\ell,b,u)$ widget.  The vertices $x$, $y$, and $z$ are also the same as the corresponding vertices appearing inside $R_{\substack{x \mapsto y \\ y \mapsto z}}(\ell, r)$, and some of the edges incident to them (for example, the edge $(x, r)$) have been omitted to improve legibility.

The properties of $U_{x,y,z}(\ell,b,u)$ highlighted by the next lemmas may seem ill-motivated at first.  We explain their significance after the proofs.

\begin{lemma}\label{lem:U3Widget}{\ }
\begin{itemize}
\item[(i)]  Every $3$-coloring $\nu$ of the subgraph of $U_{x,y,z}(\ell,b,u)$ induced by $\{x,y,z,\ell,b\}$ can be extended to a $3$-coloring of $U_{x,y,z}(\ell,b,u)$.

\item[(ii)]  If $\nu$ is a $3$-coloring of $U_{x,y,z}(\ell,b,u)$ in which $\ell =_\nu x$ and $b = _\nu y$, then $u =_\nu x$.

\item[(iii)]  Every $3$-coloring $\nu$ of the subgraph of $U_{x,y,z}(\ell,b,u)$ induced by $\{x,y,z,\ell,b\}$ in which $\ell =_\nu x$ and $b \neq_\nu y$ can be extended to a $3$-coloring of $U_{x,y,z}(\ell,b,u)$ in which $u =_\nu z$.

\item[(iv)]  Every $3$-coloring $\nu$ of the subgraph of $U_{x,y,z}(\ell,b,u)$ induced by $\{x,y,z,\ell,b\}$ in which $\ell =_\nu y$ can be extended to a $3$-coloring of $U_{x,y,z}(\ell,b,u)$ in which $u =_\nu y$.
\end{itemize}
\end{lemma}

\begin{proof}
For $(i)$, let $\nu$ be a $3$-coloring of the subgraph of $U_{x,y,z}(\ell,b,u)$ induced by $\{x,y,z,\ell,b\}$.  
\begin{itemize}
\item If $\ell =_\nu x$ and $b =_\nu x$, then color the widget so that $\bar\ell =_\nu y$, $r =_\nu y$, $d =_\nu y$, and $u =_\nu z$.

\item If $\ell =_\nu x$ and $b =_\nu y$, then color the widget so that $\bar\ell =_\nu y$, $r =_\nu y$, $d =_\nu z$, and $u =_\nu x$.

\item If $\ell =_\nu x$ and $b =_\nu z$, then color the widget so that $\bar\ell =_\nu y$, $r =_\nu y$, $d =_\nu y$, and $u =_\nu z$.

\item If $\ell =_\nu y$ and $b =_\nu x$, then color the widget so that $\bar\ell =_\nu x$, $r =_\nu z$, $d =_\nu z$, and $u =_\nu y$.

\item If $\ell =_\nu y$ and $b =_\nu y$, then color the widget so that $\bar\ell =_\nu x$, $r =_\nu z$, $d =_\nu x$, and $u =_\nu y$.

\item If $\ell =_\nu y$ and $b =_\nu z$, then color the widget so that $\bar\ell =_\nu x$, $r =_\nu z$, $d =_\nu x$, and $u =_\nu y$.
\end{itemize}
In each of the above cases, the sub-widget $R_{\substack{x \mapsto y \\ y \mapsto z}}(\ell, r)$ is colored according to Lemma~\ref{lem:RWidget}.

For $(ii)$, let $\nu$ be a $3$-coloring of $U_{x,y,z}(\ell,b,u)$ in which $\ell =_\nu x$ and $b =_\nu y$.  Then it must be that $\bar\ell =_\nu y$ and $d =_\nu z$, and therefore it must be that $u =_\nu x$.

Item $(iii)$ can be seen by inspecting the first and third colorings in the proof of $(i)$.  

Item $(iv)$ can be seen by inspecting the last three colorings in the proof of $(i)$.
\end{proof}

\begin{lemma}\label{lem:U3WidgetDecode}
Let $\nu$ be a $3$-coloring of $U_{x,y,z}(\ell,b,u)$.  If $w$ is $\bar\ell$, $u$, or any vertex appearing in the $R_{\substack{x \mapsto y \\ y \mapsto z}}(\ell, r)$ sub-widget that is not $x$, $y$, or $z$, then the color of $w$ determines the color of $\ell$.  Moreover,
\begin{itemize}
\item if $d =_\nu x$, then $\ell =_\nu y$;
\item if $d =_\nu y$, then $\ell =_\nu x$;
\item if $d =_\nu z$, then $b \neq_\nu z$.
\end{itemize}
\end{lemma}

\begin{proof}
Let $\nu$ be a $3$-coloring of $U_{x,y,z}(\ell,b,u)$.  It is easy to see that if $\bar\ell =_\nu x$, then $\ell =_\nu y$ and that if $\bar\ell =_\nu y$, then $\ell =_\nu x$.  If $w$ is a vertex in $R_{\substack{x \mapsto y \\ y \mapsto z}}(\ell, r)$ that is not $x$, $y$, or $z$, then the color of $w$ determines the color of $\ell$ by Lemma~\ref{lem:RWidget}~$(iii)$.  For $u$, if $u =_\nu x$ or $u =_\nu z$ it cannot be that $\ell =_\nu y$ because then $\bar\ell =_\nu x$ and, by Lemma~\ref{lem:RWidget}~$(i)$, $r =_\nu z$.  On the other hand, if $u =_\nu y$, it cannot be that $\ell =_\nu x$ because then $\bar\ell =_\nu y$.  Thus if $u =_\nu x$ or $u =_\nu z$, then $\ell =_\nu x$; and if $u =_\nu y$, then $\ell =_\nu y$.  It is easy to see that if $d =_\nu x$ then $\ell =_\nu y$, that if $d =_\nu y$ then $\ell =_\nu x$, and that if $d =_\nu z$ then $b \neq_\nu z$ because $\ell$ and $b$ are neighbors of $d$.
\end{proof}

Consider a clause $\ell_0 \orr \ell_1 \orr \cdots \orr \ell_{n-1}$.  The idea is to code truth assignments that satisfy the clause as $3$-colorings of a graph constructed by chaining together widgets of the form $U_{x,y,z}(\ell_i,b,u)$.  Let $\nu$ be a $3$-coloring of $U_{x,y,z}(\ell_i,b,u)$.  The color of the vertex $\ell_i$ represents the truth value of the literal $\ell_i$:  $\ell_i =_\nu x$ is interpreted as $\ell_i$ is false, and $\ell_i =_\nu y$ is interpreted as $\ell_i$ is true.  The color of the vertex $b$ represents the truth value of $\ell_0 \orr \ell_1 \orr \cdots \orr \ell_{i-1}$ as well as the truth value of the literal $\ell_{i-1}$:  $b =_\nu x$ is interpreted as $\ell_0 \orr \ell_1 \orr \cdots \orr \ell_{i-1}$ is true but $\ell_{i-1}$ is false; $b =_\nu y$ is interpreted as $\ell_0 \orr \ell_1 \orr \cdots \orr \ell_{i-1}$ is false (and hence also as $\ell_{i-1}$ is false); and $b =_\nu z$ is interpreted as $\ell_{i-1}$ is true (and hence also as $\ell_0 \orr \ell_1 \orr \cdots \orr \ell_{i-1}$ is true).  Similarly, the color of the vertex $u$ represents the truth value of $\ell_0 \orr \ell_1 \orr \cdots \orr \ell_i$ as well as the truth value of the literal $\ell_i$.  However, the meanings of the colors are permuted:  $u =_\nu x$ is interpreted as $\ell_0 \orr \ell_1 \orr \cdots \orr \ell_i$ is false (and hence also as $\ell_i$ is false); $u =_\nu y$ is interpreted as $\ell_i$ is true (and hence also as $\ell_0 \orr \ell_1 \orr \cdots \orr \ell_i$ is true); and $u =_\nu z$ is interpreted as $\ell_0 \orr \ell_1 \orr \cdots \orr \ell_i$ is true but $\ell_i$ is false.  Lemma~\ref{lem:U3Widget} tells us that $U_{x,y,z}(\ell_i,b,u)$ properly implements this coding scheme.  Lemma~\ref{lem:U3Widget}~$(ii)$ says that if a $3$-coloring codes that $\ell_i$ is false and that $\ell_0 \orr \ell_1 \orr \cdots \orr \ell_{i-1}$ is false, then it must also code that $\ell_0 \orr \ell_1 \orr \cdots \orr \ell_i$ is false.  Lemma~\ref{lem:U3Widget}~$(iii)$ says that if $\nu$ is a $3$-coloring of the subgraph of $U_{x,y,z}(\ell_i,b,u)$ induced by $\{x,y,z,\ell_i,b\}$ coding that $\ell_i$ is false and that $\ell_0 \orr \ell_1 \orr \cdots \orr \ell_{i-1}$ is true, then $\nu$ can be extended to a $3$-coloring of $U_{x,y,z}(\ell_i,b,u)$ coding that $\ell_0 \orr \ell_1 \orr \cdots \orr \ell_i$ is true.  The reader may worry that here it is also possible to extend $\nu$ to incorrectly code that $\ell_0 \orr \ell_1 \orr \cdots \orr \ell_i$ is false, so we assure the reader that this is irrelevant.  What is important is that it is possible to extend $\nu$ to code the correct information.  Lemma~\ref{lem:U3Widget}~$(iv)$ says that if $\nu$ is a $3$-coloring of the subgraph of $U_{x,y,z}(\ell_i,b,u)$ induced by $\{x,y,z,\ell_i,b\}$ coding that $\ell_i$ is true, then $\nu$ can be extended to a $3$-coloring of $U_{x,y,z}(\ell_i,b,u)$ coding that $\ell_0 \orr \ell_1 \orr \cdots \orr \ell_i$ is true.  Lemma~\ref{lem:U3WidgetDecode} helps us deduce the colors of literal-coding vertices from the colors of auxiliary vertices and hence helps us compute a homogeneous set for a set of clauses from a homogeneous set for a graph.

The next widget combines $U_{x,y,z}(\ell,b,u)$ widgets into widgets coding clauses.

\newpage

\begin{widget}
$D(\ell_0,\ell_1,\dots,\ell_{n-1})$ is the following widget.
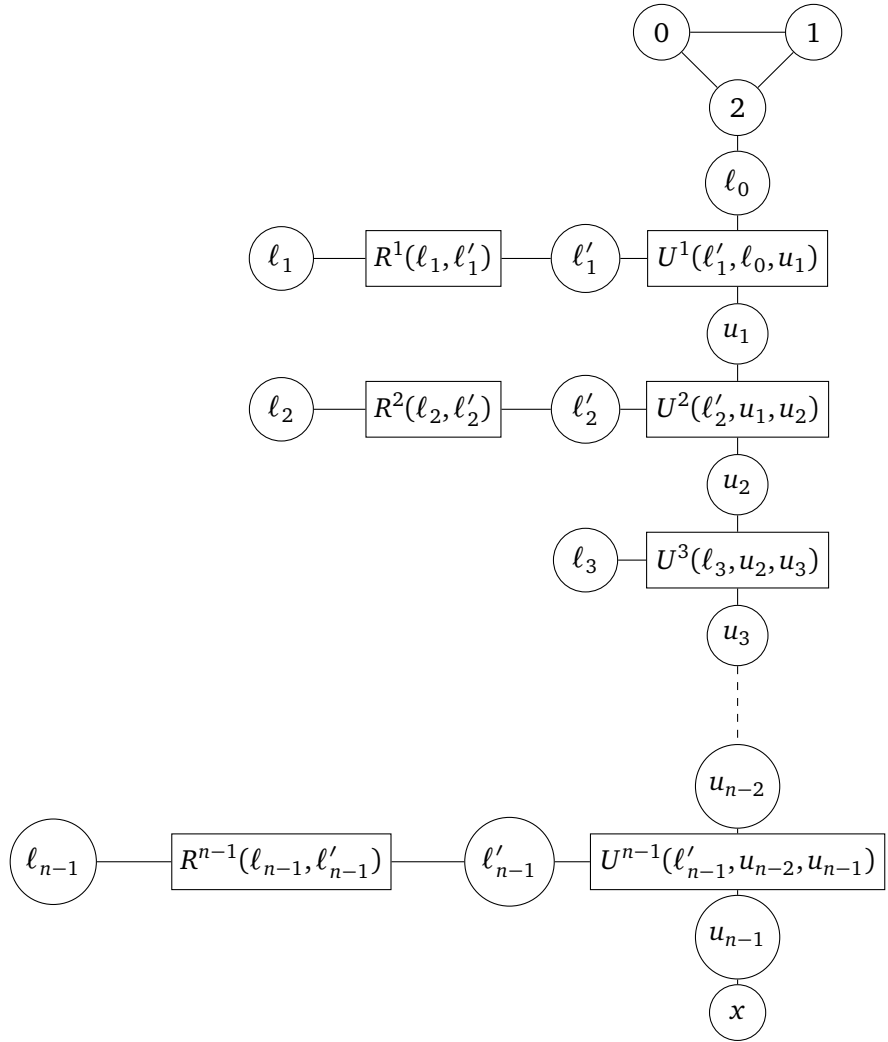
\begin{center}
\begin{tikzpicture}[x=1cm, y=1cm, node/.style={circle, draw, minimum size=2em}, widget/.style={rectangle, draw, minimum size=2em}]
	
	\node[node] (0) at (-1, 1) {$0$};
	\node[node] (1) at (1, 1) {$1$};
	\node[node] (2) at (0, 0) {$2$};
	
	\draw (0) -- (1) -- (2) -- (0);

	\node[node] (l0) at (0,-1) {$\ell_0$};
	
	\node[widget] (U1) at (0, -2) {$U^1(\ell_1', \ell_0, u_1)$};
	\node[node] (l1p) at (-2, -2) {$\ell_1'$};
	\node[widget] (R1) at (-4, -2) {$R^1(\ell_1, \ell_1')$};
	\node[node] (l1) at (-6, -2) {$\ell_1$};
	\node[node] (u1) at (0, -3) {$u_1$};
				
	\draw (2) -- (l0) -- (U1) -- (u1);
	\draw (l1) -- (R1) -- (l1p) -- (U1);
	
	\node[widget] (U2) at (0, -4) {$U^2(\ell_2', u_1, u_2)$};
	\node[node] (l2p) at (-2, -4) {$\ell_2'$};
	\node[widget] (R2) at (-4, -4) {$R^2(\ell_2, \ell_2')$};
	\node[node] (l2) at (-6, -4) {$\ell_2$};
	\node[node] (u2) at (0, -5) {$u_2$};
				
	\draw (u1) -- (U2) -- (u2);
	\draw (l2) -- (R2) -- (l2p) -- (U2);
	
	\node[widget] (U3) at (0, -6) {$U^3(\ell_3, u_2, u_3)$};
	\node[node] (l3) at (-2, -6) {$\ell_3$};
	\node[node] (u3) at (0, -7) {$u_3$};
				
	\draw (u2) -- (U3) -- (u3);
	\draw (l3) -- (U3);
	
	\node[node] (un2) at (0, -9) {$u_{n-2}$};
	\node[widget] (Un1) at (0, -10) {$U^{n-1}(\ell_{n-1}', u_{n-2}, u_{n-1})$};
	\node[node] (ln1p) at (-3, -10) {$\ell_{n-1}'$};
	\node[widget] (Rn1) at (-6, -10) {$R^{n-1}(\ell_{n-1}, \ell_{n-1}')$};
	\node[node] (ln1) at (-9, -10) {$\ell_{n-1}$};
	\node[node] (un1) at (0, -11) {$u_{n-1}$};
	\node[node] (x) at (0, -12) {$x$};

	\draw[dashed] (u3) -- (un2);				
	\draw (un2) -- (Un1) -- (un1) -- (x);
	\draw (ln1) -- (Rn1) -- (ln1p) -- (Un1);
	
\end{tikzpicture}
\end{center}
The widget also contains the edge $(2,\ell_i)$ for each $i < n$, which we omitted from the diagram to keep it legible.  For $0 < i < n$, the sub-widget $U^i(\ell_i', u_{i-1}, u_i)$ is $U_{0,1,2}(\ell_i, u_{i-1}, u_i)$ if $i \equiv 0 \mod 3$, is $U_{2,0,1}(\ell_i', u_{i-1}, u_i)$ if $i \equiv 1 \mod 3$ (with $\ell_0$ in place of $u_0$ when $i=1$), and is $U_{1,2,0}(\ell_i', u_{i-1}, u_i)$ if $i \equiv 2 \mod 3$.  For $0 < i < n$, the sub-widget $R^i(\ell_i,\ell_i')$ is $R_{\substack{1 \mapsto 0 \\ 0 \mapsto 2}}(\ell_i,\ell_i')$ if $i \equiv 1 \mod 3$ and is $R_{\substack{0 \mapsto 1 \\ 1 \mapsto 2}}(\ell_i,\ell_i')$ if $i \equiv 2 \mod 3$.  If $i \equiv 0 \mod 3$, then there is just the vertex $\ell_i$ instead of the subgraph

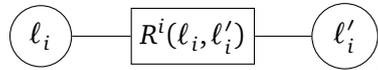
\begin{center}
\begin{tikzpicture}[x=1cm, y=1cm, node/.style={circle, draw, minimum size=2em}, widget/.style={rectangle, draw, minimum size=2em}]
	
	\node[node] (l) at (0, 0) {$\ell_i$};
	\node[widget] (R) at (2, 0) {$R^i(\ell_i, \ell_i')$};
	\node[node] (lp) at (4, 0) {$\ell_i'$};
	
	\draw (l) -- (R) -- (lp);
		
\end{tikzpicture}.
\end{center}
The vertex $x$ is $0$ if $n-1 \equiv 0 \mod 3$, is $2$ if $n-1 \equiv 1 \mod 3$, and is $1$ if $x \equiv 2 \mod 3$.  Note that the vertex $x$ is thus drawn twice because it is identical to one of $0$, $1$, $2$.  For clarity, we also point out that in the case of $D(\ell_0)$, the widget is simply
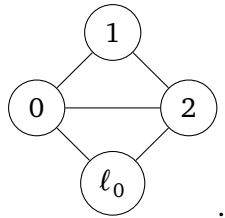
\begin{center}
\begin{tikzpicture}[x=1cm, y=1cm, node/.style={circle, draw, minimum size=2em}, widget/.style={rectangle, draw, minimum size=2em}]
	
	\node[node] (0) at (0, 0) {$0$};
	\node[node] (1) at (1, 1) {$1$};
	\node[node] (2) at (2, 0) {$2$};
	
	\draw (0) -- (1) -- (2) -- (0);

	\node[node] (l0) at (1,-1) {$\ell_0$};
	
	\draw (0) -- (l0) -- (2);
	
\end{tikzpicture}.
\end{center}
\end{widget}

\begin{lemma}\label{lem:DWidget}{\ }
\begin{itemize}
\item[(i)] Every $3$-coloring $\nu$ of the subgraph of $D(\ell_0,\ell_1,\dots,\ell_{n-1})$ induced by $\{0,1,2,\ell_0,\ell_1,\dots,\ell_{n-1}\}$ in which $\ell_i =_\nu 1$ for some $i < n$ can be extended to a $3$-coloring of $D(\ell_0,\ell_1,\dots,\ell_{n-1})$.

\item[(ii)] There is no $3$-coloring $\nu$ of $D(\ell_0,\ell_1,\dots,\ell_{n-1})$ in which $\ell_0 =_\nu \ell_1 =_\nu \cdots =_\nu \ell_{n-1} =_\nu 0$.
\end{itemize}
\end{lemma}

\begin{proof}
For~$(i)$, let $\nu$ be a $3$-coloring of the subgraph induced by $\{0,1,2,\ell_0,\ell_1,\dots,\ell_{n-1}\}$ in which $\ell_i =_\nu 1$ for some $i < n$.  For each $i < n$, let $D_i(\ell_0,\ell_1,\dots,\ell_{n-1})$ denote the subgraph of $D(\ell_0,\ell_1,\dots,\ell_{n-1})$ induced by $0$, $1$, $2$ and the vertices appearing in $R^j(\ell_j,\ell_j')$ and $U^j(\ell_j',u_{j-1},u_j)$ for all $j \leq i$.  That is, if $i < n-1$, then $D_i(\ell_0,\ell_1,\dots,\ell_{n-1})$ is $D(\ell_0,\ell_i,\dots,\ell_i)$ without the edge between $u_i$ and $x$; and if $i=n-1$, then $D_i(\ell_0,\ell_1,\dots,\ell_{n-1})$ is $D(\ell_0,\ell_i,\dots,\ell_{n-1})$.  Item~$(i)$ is then the instance $i = n-1$ of the following claim.

\begin{claim*}
For all $i < n$, $\nu$ can be extended to a $3$-coloring of $D_i(\ell_0,\ell_1,\dots,\ell_{n-1})$.  Moreover, if $\ell_j =_\nu 1$ for some $j \leq i$, then $\nu$ can be extended to a $3$-coloring of $D_i(\ell_0,\ell_1,\dots,\ell_{n-1})$ in which $\nu(u_i)$ codes this fact.  That is, if $i \equiv 0 \mod 3$, then $u_i \neq_\nu 0$; if $i \equiv 1 \mod 3$, then $u_i \neq_\nu 2$; and if $i \equiv 2 \mod 3$, then $u_i \neq_\nu 1$ (for $i=0$, interpret $u_0$ as $\ell_0$).
\end{claim*}

\begin{proof}
By induction on $i < n$.  For $i=0$, $D_0(\ell_0,\ell_1,\dots,\ell_{n-1})$ is the subgraph of induced by $\{0,1,2,\ell_0\}$, which is $3$-colored by $\nu$ by assumption.  Clearly if $\ell_0 =_\nu 1$, then $\ell_0 \neq_\nu 0$.  Now suppose that $\nu$ has been extended to a $3$-coloring of $D_{i-1}(\ell_0,\ell_1,\dots,\ell_{n-1})$.  For the sake of argument, suppose that $i \equiv 1 \mod 3$ (the $i \equiv 0 \mod 3$ and $i \equiv 2 \mod 3$ cases are symmetric), and suppose that if $\ell_j =_\nu 1$ for some $j \leq i-1$, then $u_{i-1} \neq_\nu 0$.  First suppose that $\ell_i =_\nu 0$.  As $R^i(\ell_i,\ell_i') = R_{\substack{1 \mapsto 0 \\ 0 \mapsto 2}}(\ell_i,\ell_i')$, apply Lemma~\ref{lem:RWidget}~$(i)$ to extend $\nu$ to $R^i(\ell_i,\ell_i')$ so that $\ell_i' =_\nu 2$.  By Lemma~\ref{lem:U3Widget}~$(i)$, it is possible to extend $\nu$ to $U^i(\ell_i',u_{i-1},u_i)$.  Furthermore, if $\ell_j =_\nu 1$ for some $j \leq i-1$, then $u_{i-1} \neq_\nu 0$.  In this situation, by Lemma~\ref{lem:U3Widget}~$(iii)$, it is possible to extend $\nu$ to $U^i(\ell_i',u_{i-1},u_i) = U_{2,0,1}(\ell_i',u_{i-1},u_i)$ so that $u_i =_\nu 1$ (and hence $u_i \neq_\nu 2$).  Now suppose that $\ell_i =_\nu 1$.  As $R^i(\ell_i,\ell_i') = R_{\substack{1 \mapsto 0 \\ 0 \mapsto 2}}(\ell_i,\ell_i')$, apply Lemma~\ref{lem:RWidget}~$(i)$ to extend $\nu$ to $R^i(\ell_i,\ell_i')$ so that $\ell_i' =_\nu 0$.  By Lemma~\ref{lem:U3Widget}~$(iv)$, it is possible to extend $\nu$ to $U^i(\ell_i',u_{i-1},u_i) = U_{2,0,1}(\ell_i',u_{i-1},u_i)$ so that $u_i =_\nu 0$ (and hence $u_i \neq_\nu 2$).
\end{proof}

For~$(ii)$, suppose for a contradiction that $\nu$ is a $3$-coloring of $D(\ell_0,\ell_1,\dots,\ell_{n-1})$ in which $\ell_0 =_\nu \ell_1 =_\nu \cdots =_\nu \ell_{n-1} =_\nu 0$.  We prove by induction on $i < n$ that $u_i =_\nu 0$ if $i \equiv 0 \mod 3$, $u_i = 2$ if $i \equiv 1 \mod 3$, and $u_i =_\nu 1$ if $i \equiv 2 \mod 3$ (again $u_0$ is interpreted as $\ell_0$).  Item~$(ii)$ follows from the case $i=n-1$ because this gives the contradiction $u_{n-1} =_\nu x$.  For $i = 0$, $\ell_0 =_\nu 0$ by assumption.  Now consider $0 < i < n$, assume for the sake of argument that $i \equiv 1 \mod 3$ (the $i \equiv 0 \mod 3$ and $i \equiv 2 \mod 3$ cases are symmetric), and assume that $u_{i-1} =_\nu 0$.  By Lemma~\ref{lem:RWidget}~$(i)$ for the widget $R^i(\ell_i,\ell_i') = R_{\substack{1 \mapsto 0 \\ 0 \mapsto 2}}(\ell_i,\ell_i')$, we have that $\ell_i' =_\nu 2$.  Thus $U^i(\ell_i',u_{i-1},u_i) = U_{2,0,1}(\ell_i',u_{i-1},u_i)$, $\ell_i' =_\nu 2$, and $u_{i-1} =_\nu 0$, so it must be that $u_i =_\nu 2$ by Lemma~\ref{lem:U3Widget}~$(ii)$.
\end{proof}

\begin{lemma}\label{lem:DWidgetDecode}
Let $\nu$ be a $3$-coloring of $D(\ell_0,\ell_1,\dots,\ell_{n-1})$.  If $0<i<n$ and $w$ is a vertex appearing in an $R^i(\ell_i,\ell_i')$ sub-widget or a $U^i(\ell_i',u_{i-1},u_i)$ sub-widget that is not $0$, $1$, or $2$, then the color of $w$ determines either the color of $\ell_i$ or the color of $\ell_{i-1}$.
\end{lemma}

\begin{proof}
Consider a $3$-coloring $\nu$ of $D(\ell_0,\ell_1,\dots,\ell_{n-1})$, an $i$ with $0<i<n$, and a vertex $w$ in an $R^i(\ell_i,\ell_i')$ sub-widget or a $U^i(\ell_i',u_{i-1},u_i)$ sub-widget that is not $0$, $1$, or $2$.  If $w$ appears in $R^i(\ell_i,\ell_i')$, then the color of $w$ determines the color of $\ell_i$ by Lemma~\ref{lem:RWidget}~$(iii)$.  If $w$ appears in $U^i(\ell_i',u_{i-1},u_i)$, then there are a few cases.  If $w$ is not $u_{i-1}$ or $d$, then the color of $w$ determines the color $\ell_i'$ by Lemma~\ref{lem:U3WidgetDecode}, which we have just seen determines the color of $\ell_i$ (or $\ell_i'$ is $\ell_i$ in the case $i \equiv 0 \mod 3$).  Consider $w = u_{i-1}$.  If $i=1$, then $u_{i-1}$ is really $\ell_0$, and of course the color of $\ell_0$ determines the color of $\ell_0$.  Otherwise, $i>1$, $u_{i-1}$ appears in the sub-widget $U^{i-1}(\ell_{i-1}',u_{i-2},u_{i-1})$, and hence the color of $u_{i-1}$ determines the color of $\ell_{i-1}$.

Lastly, consider $w=d$.  $U^i(\ell_i',u_{i-1},u_i)$ is $U_{x,y,z}(\ell_i',u_{i-1},u_i)$, where $x$, $y$, and $z$ are some permutation of $0$, $1$, and $2$.  If $d =_\nu x$ or $d =_\nu y$, then this determines the color of $\ell_i'$ by Lemma~\ref{lem:U3WidgetDecode}, which in turn determines the color of $\ell_i$.  Otherwise $d =_\nu z$, meaning that $u_{i-1} \neq_\nu z$ by Lemma~\ref{lem:U3WidgetDecode}.  If $i=1$, then $z=1$, $u_0$ is really $\ell_0$, and we conclude that $\ell_0 =_\nu 0$.  If $i > 1$, then $U^{i-1}(\ell_{i-1}', u_{i-2}, u_{i-1})$ is $U_{y,z,x}(\ell_{i-1}', u_{i-2}, u_{i-1})$ and, by examining the proof of Lemma~\ref{lem:U3WidgetDecode}, $u_{i-1} \neq_\nu z$ implies that $\ell_{i-1}' =_\nu y$, which in turn determines the color of $\ell_{i-1}$.
\end{proof}

To code the conjunction of two clauses $\ell_0 \orr \ell_1 \orr \cdots \orr \ell_{n-1}$ and $s_0 \orr s_1 \orr \cdots \orr s_{m-1}$, we overlap the widgets $D(\ell_0,\ell_1,\dots,\ell_{n-1})$ and $D(s_0,s_1,\dots,s_{m-1})$ by sharing the vertices pertaining to the longest common prefix of $\ell_0,\ell_1,\dots,\ell_{n-1}$ and $s_0,s_1,\dots,s_{m-1}$.  For example, consider the clauses $\ell_0 \orr \ell_1 \orr \ell_2 \orr \ell_3 \orr \ell_4$ and $\ell_0 \orr \ell_1 \orr s_2 \orr s_3$, where $\ell_2 \neq s_2$.  We overlap $D(\ell_0,\ell_1,\ell_2,\ell_3,\ell_4)$ and $D(\ell_0,\ell_1,s_2,s_3)$ as follows:

\begin{center}
\begin{tikzpicture}[x=1cm, y=1cm, node/.style={circle, draw, minimum size=2em}, widget/.style={rectangle, draw, minimum size=2em}]
	
	\node[node] (0) at (-1, 1) {$0$};
	\node[node] (1) at (1, 1) {$1$};
	\node[node] (2) at (0, 0) {$2$};
	
	\node[node] (l0) at (0,-1) {$\ell_0$};
			
	\draw (0) -- (1) -- (2) -- (0);
	
	\draw (2) -- (l0);
	
	\node[widget] (U1) at (-0, -2) {$U^1(\ell_1', \ell_0, u_1)$};
	\node[node] (l1p) at (-2, -2) {$\ell_1'$};
	\node[widget] (R1) at (-4, -2) {$R^1(\ell_1, \ell_1')$};
	\node[node] (l1) at (-6, -2) {$\ell_1$};
	\node[node] (u1) at (0, -3) {$u_1$};
	
	\draw (l0) -- (U1) -- (u1);
	\draw (l1) -- (R1) -- (l1p) -- (U1);
	
	\node[widget] (U2l) at (-2, -4) {$U^2(\ell_2', u_1, u_2)$};
	\node[node] (l2p) at (-4, -4) {$\ell_2'$};
	\node[widget] (R2l) at (-6, -4) {$R^2(\ell_2, \ell_2')$};
	\node[node] (l2) at (-6, -5) {$\ell_2$};
	\node[node] (u2) at (-2, -5) {$u_2$};
	
	\node[widget] (U2s) at (2, -4) {$U^2(s_2', u_1, v_2)$};
	\node[node] (s2p) at (4, -4) {$s_2'$};
	\node[widget] (R2s) at (6, -4) {$R^2(s_2, s_2')$};
	\node[node] (s2) at (6, -5) {$s_2$};
	\node[node] (v2) at (2, -5) {$v_2$};

	\draw (u1) -- (U2l) -- (u2);
	\draw (l2) -- (R2l) -- (l2p) -- (U2l);
	
	\draw (u1) -- (U2s) -- (v2);
	\draw (s2) -- (R2s) -- (s2p) -- (U2s);
	
	\node[widget] (U3l) at (-2, -6) {$U^3(\ell_3', u_2, u_3)$};
	\node[node] (l3) at (-4, -6) {$\ell_3$};
	\node[node] (u3) at (-2, -7) {$u_3$};
	
	\node[widget] (U3s) at (2, -6) {$U^3(s_3', v_2, v_3)$};
	\node[node] (s3) at (4, -6) {$s_3$};
	\node[node] (v3) at (2, -7) {$v_3$};

	\draw (u2) -- (U3l) -- (u3);
	\draw (l3) -- (U3l);
	
	\draw (v2) -- (U3s) -- (v3);
	\draw (s3) -- (U3s);
	
	\node[widget] (U4l) at (-2, -8) {$U^4(\ell_4', u_3, u_4)$};
	\node[node] (l4p) at (-4, -8) {$\ell_4'$};
	\node[widget] (R4l) at (-6, -8) {$R^4(\ell_4, \ell_4')$};
	\node[node] (l4) at (-6, -9) {$\ell_4$};
	\node[node] (u4) at (-2, -9) {$u_4$};

	\draw (u3) -- (U4l) -- (u4);
	\draw (l4) -- (R4l) -- (l4p) -- (U4l);
	
	\node[node] (2p) at (-2, -10) {$2$};
	\node[node] (0p) at (2, -8) {$0$};
	
	\draw (u4) -- (2p);
	\draw (v3) -- (0p);
	
\end{tikzpicture}
\end{center}

\begin{theorem}\label{thm:rcolor3-rsat}
$\rca \vdash \rcolor_3 \imp \rwkl$.
\end{theorem}

\begin{proof}
We prove $\rca \vdash \rcolor_3 \imp \rsat_{\textup{2-branching}}$.  The theorem follows by Proposition~\ref{prop:RWKL2BranchRSAT}.

Let $C$ be a $2$-branching and finitely satisfiable set of clauses over an infinite set of atoms $A = \{a_i : i \in \Nb\}$.  We assume that no clause in $C$ is a proper prefix of any other clause in $C$ by removing from $C$ every clause that has a proper prefix also in $C$.  We build a locally $3$-colorable graph $G$ such that every infinite homogeneous set for $G$ computes an infinite homogeneous set for $C$.  To start, $G$ contains the vertices $0$, $1$, and $2$, as well as the literal-coding vertices $a_i$ and $\neg a_i$ for each atom $a_i \in A$.  These vertices are connected according to the diagram below.

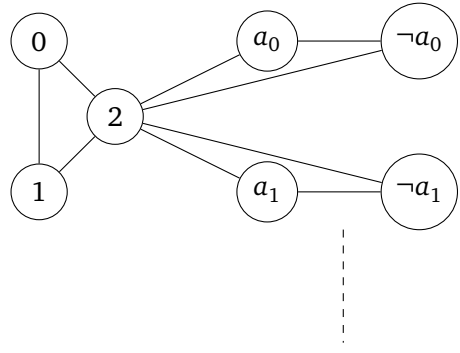
\begin{center}
\begin{tikzpicture}[x=1cm, y=1cm, node/.style={circle, draw, minimum size=2em}]

	\node[node] (0) at (-1,1) {$0$};
	\node[node] (1) at (-1, -1)  {$1$};
	\node[node] (2) at (0, 0) {$2$};
	
	\node[node] (a0) at (2,1) {$a_0$};
	\node[node] (na0) at (4,1) {$\neg a_0$};
	\node[node] (a1) at (2,-1) {$a_1$};
	\node[node] (na1) at (4,-1) {$\neg a_1$};
	
	\draw (0) -- (1) -- (2) -- (0);
	\draw (2) -- (a0) -- (na0) -- (2);
	\draw (2) -- (a1) -- (na1) -- (2);
	\draw[dashed] (3,-1.5) -- (3,-3);
	
\end{tikzpicture}
\end{center}

Now build $G$ in stages by considering the clauses in $C$ one-at-a-time.  For clause $\ell_0 \orr \ell_1 \orr \cdots \orr \ell_{n-1}$, find the previously appearing clause $s_0 \orr s_1 \orr \cdots \orr s_{m-1}$ having the longest common prefix with $\ell_0 \orr \ell_1 \orr \cdots \orr \ell_{n-1}$.  Then add the widget $D(\ell_0,\ell_1,\dots,\ell_{n-1})$ by overlapping it with $D(s_0,s_1,\dots,s_{m-1})$ as described above.  In $D(\ell_0,\ell_1,\dots,\ell_{n-1})$, for each $i < n$, the vertex $\ell_i$ is the vertex $a_i$ if the literal $\ell_i$ is the literal $a_i$, and the vertex $\ell_i$ is the vertex $\neg a_i$ if the literal $\ell_i$ is the literal $\neg a_i$.  The vertices appearing in the sub-widgets $R^i(\ell_i,\ell_i')$ and $U^i(\ell_i',u_{i-1},u_i)$ for $i$ beyond the index at which $\ell_0 \orr \ell_1 \orr \cdots \orr \ell_{n-1}$ differs from $s_0 \orr s_1 \orr \cdots \orr s_{m-1}$ are chosen fresh, except for $0$, $1$, $2$, and the literal-coding vertices $\ell_i$.  This completes the construction of $G$.

\begin{claim*}
$G$ is locally $3$-colorable.
\end{claim*}

\begin{proof}
Let $G_0$ be a finite subgraph of $G$.  Let $s$ be the latest stage at which a vertex in $G_0$ appears, and let $C_0 \subseteq C$ be the set of clauses considered up to stage $s$.  By extending $G_0$, we may assume that it is the graph constructed up to stage $s$.

By the finite satisfiability of $C$, let $t \colon \atoms(C_0) \imp \{\true, \false\}$ be a truth assignment satisfying $C_0$.  The truth assignment $t$ induces a $3$-coloring $\nu$ on the literal-coding vertices in $G_0$.  First define $\nu$ on the truth value-coding vertices by $\nu(0) = 0$, $\nu(1)=1$, and $\nu(2)=2$.  If $t(\ell)$ is defined for the literal $\ell$, then set $\nu(\ell) = t(\ell)$ (identifying $0$ with $\false$ and $1$ with $\true$).  If $\ell$ is a literal-coding vertex in $G_0$ on which $t$ is not defined, then set $\nu(\ell)=1$ if $\ell$ is a positive literal and set $\nu(\ell)=0$ if $\ell$ is a negative literal.  For each clause $\ell_0 \orr \ell_1 \orr  \cdots \orr \ell_{n-1}$ in $C_0$, extend $\nu$ to a $3$-coloring of $G_0$ by coloring each widget $D(\ell_0,\ell_1,\dots,\ell_{n-1})$ according to the algorithm implicit in the proof of Lemma~\ref{lem:DWidget}~$(i)$.  The hypothesis of Lemma~\ref{lem:DWidget}~$(i)$ is satisfied because $t$ satisfies $C_0$, so for each clause $\ell_0 \orr \ell_1 \orr \cdots \orr \ell_{n-1}$ in $C_0$, there is an $i < n$ such that $\ell_i =_\nu 1$.  Overlapping widgets $D(\ell_0,\ell_1,\dots,\ell_{n-1})$ and $D(s_0,s_1,\dots,s_{m-1})$ are colored consistently because the colors of the shared vertices depend only on the colors of the literal-coding vertices corresponding to the longest common prefix of the two clauses.
\end{proof}

Apply $\rcolor_3$ to $G$ to get an infinite homogeneous set $H$.  We may assume that $H$ contains exactly one of the truth value-coding vertices $0$, $1$, or $2$.  Call this vertex $c$.

Consider a vertex $w \in H$ that is not $c$.  The vertex $w$ appears in some widget $D(\ell_0, \ell_1, \dots, \ell_{n-1})$, and, by Lemma~\ref{lem:DWidgetDecode}, from $w$ we can compute an $i < n$ and a $c_i \in \{0,1\}$ such that $\ell_i =_\nu c_i$ whenever $\nu$ is a $3$-coloring of $D(\ell_0, \ell_1, \dots, \ell_{n-1})$ in which $w =_\nu c$.  Moreover, for each literal $\ell$, we can compute a bound on the number of vertices $w$ in the graph whose color determines the color of $\ell$.  Still by Lemma~\ref{lem:DWidgetDecode}, if $w$ appears in an $R^i(\ell_i,\ell_i')$ sub-widget or a $U^i(\ell_i',u_{i-1},u_i)$ sub-widget, then the color of $w$ determines either the color of $\ell_i$ or the color of $\ell_{i-1}$.  Thus the vertices whose colors determine the color of $\ell_i$ only appear in $R^i(\ell_i,\ell_i')$, $U^i(\ell_i',u_{i-1},u_i)$, $R^{i+1}(\ell_{i+1},\ell_{i+1}')$, and $U^{i+1}(\ell_{i+1}',u_i,u_{i+1})$ sub-widgets.  The fact that $C$ is a $2$-branching set of clauses and our protocol for overlapping the $D(\ell_0, \ell_1, \dots, \ell_{n-1})$ widgets together imply that, for every $j > 0$, there are at most $2^j$ sub-widgets of the form $R^j(\ell_j,\ell_j')$ and at most $2^j$ sub-widgets of the form $U^j(\ell_j',u_{j-1},u_j)$.  This induces the desired bound on the number of vertices whose colors determine the color of $\ell_i$.

Thus from $H$ we can compute an infinite set $H'$ of pairs $\la \ell, c_\ell \ra$, where each $\ell$ is a literal-coding vertex and each $c_\ell$ is either $0$ or $1$, such that every finite subgraph of $G$ is $3$-colorable by a coloring $\nu$ such that $(\forall \la \ell, c_\ell \ra \in H')(\ell =_\nu c_\ell)$.  Modify $H'$ to contain only pairs $\la a, c_a \ra$ for positive literal-coding vertices $a$ by replacing each pair of the form $\la \neg a, c_{\neg a} \ra$ with $\la a, 1-c_{\neg a} \ra$.  Now apply the infinite pigeonhole principle to $H'$ to get an infinite set $H''$ of positive literal-coding vertices $a$ and a new $c \in \{0,1\}$ such that the corresponding $c_a$ is always $c$.  We identify a positive literal-coding vertex $a$ with the corresponding atom and show that $H''$ is homogeneous for $C$.

Let $C_0 \subseteq C$ be finite.  Let $G_0$ be the finite subgraph of $G$ containing $\{0,1,2\}$, the literal-coding vertices whose atoms appear in the clauses in $C_0$, and the $D(\ell_0,\ell_1,\dots,\ell_{n-1})$ widgets for the clauses $\ell_0 \orr \ell_1 \orr \cdots \orr \ell_{n-1}$ in $C_0$.  By the homogeneity of $H''$ for $G$, there is a $3$-coloring $\nu$ of $G_0$ such that $a =_\nu c$ for every $a \in H''$.  From $\nu$, define a truth assignment $t$ on $\atoms(C_0)$ by $t(a) = \true$ if $a =_\nu 1$ and $t(a) = \false$ if $a =_\nu 0$.  This truth assignment satisfies every clause $\ell_0 \orr \ell_1 \orr \cdots \orr \ell_{n-1}$ in $C_0$.  The $3$-coloring $\nu$ must color the widget $D(\ell_0,\ell_1,\dots,\ell_{n-1})$, so by Lemma~\ref{lem:DWidget}~$(ii)$, it must be that $\ell_i =_\nu 1$ for some $i < n$.  Then $t(\ell_i) = \true$ for this same $i$, so $t$ satisfies $\ell_0 \orr \ell_1 \orr \cdots \orr \ell_{n-1}$.  Moreover, $t(a)$ is the truth value coded by $c$ for every $a \in H''$, so $H''$ is indeed an infinite homogeneous set for $C$.
\end{proof}

It follows that $\rwkl$, $\rcolor_k$, and $\lrcolor_k$ are equivalent for every fixed $k \geq 3$.

\begin{corollary}\label{cor:RCOLORkisRWKL}
For every $k \in \omega$ with $k \geq 3$, $\rca \vdash \rwkl \biimp \rcolor_k \biimp \lrcolor_k$.
\end{corollary}

\begin{proof}
Fix $k \in \omega$ with $k \geq 3$.  $\rca \vdash \rwkl \imp \lrcolor_k$ by Lemma~\ref{RWKLprovesLRCOLOR}, and clearly $\rca \vdash \lrcolor_k \imp \rcolor_k$.  It is easy to see that $\rca \vdash \rcolor_k \imp \rcolor_3$.  Given a locally $3$-colorable graph $G$, augment $G$ by a clique $C$ containing $k-3$ fresh vertices, and put and edge between every vertex in $C$ and every vertex in $G$.  The resulting graph $G'$ is locally $k$-colorable, and every infinite set that is $k$-homogeneous for $G'$ is also $3$-homogeneous for $G$.  Finally, $\rca \vdash \rcolor_3 \imp \rwkl$ by Theorem~\ref{thm:rcolor3-rsat}.
\end{proof}

The question of the exact strength of $\rcolor_2$ remains open.  We are unable to determine if $\rcolor_2$ implies $\rwkl$ or even if $\rcolor_2$ implies $\dnr$.

\begin{question}\label{qu-RCOLOR2vsRWKL}
Does $\rca \vdash \rcolor_2 \imp \rwkl$?
\end{question}

\begin{question}\label{qu-RCOLOR2vsDNR}
Does $\rca \vdash \rcolor_2 \imp \dnr$?
\end{question}

However, we are able to show that $\rcolor_2$ and $\lrcolor_2$ are equivalent.

\begin{theorem}\label{thm-LRCOLOR2isRCOLOR2}
$\rca \vdash \rcolor_2 \biimp \lrcolor_2$.
\end{theorem}

\begin{proof}
$\rca \vdash \lrcolor_2 \imp \rcolor_2$ is clear.  We show that $\rca \vdash \rcolor_2 \imp \lrcolor_2$.

$\rca$ suffices to prove that a finite graph is $2$-colorable if and only if it does not contain an odd-length cycle.  Thus the condition that every finite subset of vertices of a graph induces a $2$-colorable subgraph is equivalent to the condition that the graph does not contain an odd-length cycle.  Moreover, if $G = (V,E)$ is a graph such that every finite subset of $V$ induces a $2$-colorable subgraph, then, for any $H \subseteq V$, every finite $V_0 \subseteq V$ induces a subgraph that is $2$-colorable by a coloring that colors every $v \in V_0 \cap H$ color $0$ if and only if no two elements of $H$ are connected by an odd-length path.  Thus, over $\rca$, we immediately have the following two equivalences:
\begin{itemize}
\item $\rcolor_2$ is equivalent to the statement ``for every infinite graph $G = (V,E)$, if $G$ does not contain an odd-length cycle, then there is an infinite $H \subseteq V$ such that no two vertices of $H$ are connected by an odd-length path.''

\item $\lrcolor_2$ is equivalent to the statement ``for every infinite graph $G = (V,E)$ and every infinite $X \subseteq V$, if $G$ does not contain an odd-length cycle, then there is an infinite $H \subseteq X$ such that no two vertices of $H$ are connected by an odd-length path.''
\end{itemize}

Let $G = (V,E)$ be an infinite graph that does not contain an odd-length cycle, and let $X \subseteq V$ be infinite.  If there is a bound $m$ such that 
$$
(\forall x,y \in X)(\text{$x$ and $y$ are connected by an odd-length path} \imp x,y < m),
$$ 
then we may take $H = \{x \in X : x > m\}$.  So suppose instead that there are infinitely many distinct pairs $(x,y)$ of vertices in $X$ that are connected by odd-length paths, let $((x_n,y_n))_{n \in \Nb}$ enumerate this collection of pairs, and let $(p_n)_{n \in \Nb}$ enumerate a collection of odd-length paths such that the endpoints of $p_n$ are $x_n$ and $y_n$.

Define a graph $G' = (V',E')$ by
\begin{align*}
V' &= X \cup \{a_n : n \in \Nb\} \cup \{b_n : n \in \Nb\}\\
E' &= \{(x, a_n), (a_n, b_n), (b_n, y) : x, y \in X \andd x < y \andd \text{ $x$ and $y$ are the endpoints of $p_n$}\}.
\end{align*}

$G'$ does not contain an odd-length cycle.  To see this, suppose for a contradiction that $G'$ does contain an odd-length cycle.  This cycle must be of the form 
\begin{align*}
x_0, c_{m_0}, d_{m_0}, x_1, c_{m_0}, d_{m_0}, x_2, \dots, x_{n-1}, c_{m_{n-1}}, d_{m_{n-1}}, x_0,
\end{align*}
where $n$ is odd and, for each $i < n$, $x_i \in X$ and $\{c_{m_i}, d_{m_i}\} = \{a_{m_i}, b_{m_i}\}$.  Thus, for each $i < n-1$, $p_{m_i}$ is an odd-length path with endpoints $x_i$ and $x_{i+1}$, and also $p_{m_{n-1}}$ is an odd-length path with endpoints $x_{n-1}$ and $x_0$.  Therefore the path in $G$ obtained by starting at $x_0$, following $p_{m_0}$ to $x_1$, following $p_{m_1}$ to $x_2$, and so on, finally following $p_{m_{n-1}}$ from $x_{n-1}$ back to $x_0$, is an odd-length cycle in $G$, a contradiction.

Hence by $\rcolor_2$, there is an infinite $H_0 \subseteq V'$ such that no two vertices of $H_0$ are connected by an odd-length path.  In $G'$, infinitely many vertices of $X$ are connected to $H_0$.  Clearly this holds if $X \cap H_0$ is infinite.  Otherwise, $H_0$ contains infinitely many vertices of the form $a_n$ or $b_n$, and these must be connected to infinitely many vertices in $X$ because 
\begin{align*}
(\forall m)(\exists n_0)(\forall n > n_0)(\text{some endpoint of $p_n$ is $> m$}),
\end{align*}
and therefore 
\begin{align*}
(\forall m)(\exists n_0)(\forall n > n_0)(\text{$a_n$ and $b_n$ are connected to an $x \in X$ with $x > m$}).
\end{align*}
Thus there is an infinite set $H \subseteq X$ such that, in $G'$, either every $x \in H$ is connected to a vertex in $H_0$ by an even-length path, or every $x \in H$ is connected to a vertex in $H_0$ by an odd-length path.  To finish the proof, we show that, in $G$, no two vertices in $H$ are connected by an odd-length path.  Suppose for a contradiction that $x,y \in H$ are connected by an odd-length path.  Then there is an $n$ such that $x$ and $y$ are the endpoints of $p_n$, and therefore $x$ and $y$ are connected by an odd-length path in $G'$ via the vertices $a_n$ and $b_n$.  Now, in $G'$, $x$ is connected to some $u \in H_0$, $y$ is connected to some $v \in H_0$, and the witnessing paths from $x$ to $u$ and from $y$ to $v$ either both have even length or both have odd length.  In either case, the path in $G'$ from $u$ to $x$ to $y$ to $v$ has odd length.  Thus $u$ and $v$ are two vertices in $H_0$ connected by an odd-length path in $G'$, which is a contradiction.
\end{proof}

\section{The strength of Ramsey-type graph 2-coloring}\label{sec-nonImp}

In this section, we prove various non-implications concerning $\rwkl$ and $\rcolor_2$.  The main result is that $\rca + \wwkl \nvdash \rcolor_2$ (Theorem~\ref{thm-WWKLDoesNotProveRCOLOR2}).  From this it follows that $\rca + \dnr \nvdash \rwkl$, which answers Flood's question of whether or not $\rca \vdash \dnr \imp \rwkl$ from~\cite{Flood2012}.  We also show that $\rca + \cac \nvdash \rcolor_2$ (Theorem~\ref{cac-not-imply-rcolor2}).  Note that it is immediate that $\rca + \cac \nvdash \rwkl$ because $\rca + \rwkl \vdash \dnr$ (by~\cite{Flood2012}) but $\rca + \cac \nvdash \dnr$ (by~\cite{hirschfeldt2007combinatorial}).  We do not know if $\rca \vdash \rcolor_2 \imp \dnr$, so we must give a direct proof that $\rca + \cac \nvdash \rcolor_2$.

In summary, the situation is thus.  $\wkl$ and $\rt^2_2$ each imply $\rwkl$ and therefore each imply $\rcolor_2$.  However, if $\wkl$ is weakened to $\wwkl$, then it no longer implies $\rcolor_2$.  Similarly, if $\rt^2_2$ is weakened to $\cac$, then it no longer implies $\rcolor_2$.

We begin our analysis of $\rcolor_2$ by constructing an infinite, recursive, bipartite graph with no infinite, recursive, homogeneous set.  It follows that $\rca \nvdash \rcolor_2$.  The graph we construct avoids potential infinite, r.e., homogeneous sets in a strong way that aids our proof that $\rca + \cac \nvdash \rcolor_2$.

\begin{definition}
Let $G = (V,E)$ be an infinite graph.  A set $W \subseteq V^2$ is \emph{column-wise homogeneous for $G$} if $W^{[x]}$ is infinite for infinitely many $x$ (where $W^{[x]} = \set{y : \tuple{x,y} \in W}$ is the $x$\textsuperscript{th} column of $W$), and $\forall x \forall y(y \in W^{[x]} \imp \text{$\set{x,y}$ is homogeneous for $G$})$.
\end{definition}

\begin{lemma}\label{lem-RCOLOR2notCE}
There is an infinite, recursive, bipartite graph $G = (\omega, E)$ such that no r.e.\ set is column-wise homogeneous for $G$.
\end{lemma}

\begin{proof}
The construction proceeds in stages, starting at stage $0$ with $E = \emptyset$.  We say that \emph{$W_e$ requires attention at stage $s$} if $e < s$ and there is a least pair $\tuple{x,y}$ such that
\begin{itemize}
\item $e < x < y < s$,
\item $y \in W_{e,s}^{[x]}$,
\item $x$ and $y$ are not connected to each other, and
\item neither $x$ nor $y$ is connected to a vertex $\leq e$.
\end{itemize}
At stage $s$, let $e$ be least such that $W_e$ requires attention at stage $s$ and has not previously received attention.  $W_e$ then receives attention by letting $\tuple{x,y}$ witness that $W_e$ requires attention at stage $s$, letting $u$ and $v$ be the least isolated vertices $> s$, and adding the edges $(x,u)$, $(u,v)$, and $(v,y)$ to $E$.  This completes the construction.

We verify the construction.  We first show that $G$ is acyclic by showing that it is acyclic at every stage.  It follows that $G$ is bipartite because a graph is bipartite if and only if it has no odd cycles.  All vertices are isolated at the beginning of stage $0$, hence $G$ is acyclic at the beginning of stage~$0$.  By induction, suppose that $G$ is acyclic at the beginning of stage~$s$.  If no $W_e$ requires attention at stage~$s$, then no edge is added at stage $s$, hence $G$ is acyclic at the beginning of stage~$s+1$.  If some least $W_e$ requires attention at stage $s$, then during stage~$s$ we add a length-$3$ path connecting the connected components of the $x$ and $y$ such that $\tuple{x,y}$ witnesses that $W_e$ requires attention at stage~$s$.  This action does not add a cycle because by the definition of requiring attention,~$x$ and~$y$ are not connected at the beginning of stage~$s$.  Hence $G$ is acyclic at the beginning of stage~$s+1$.

We now show that, for every $e$, if there are infinitely many $x$ such that $W_e^{[x]}$ is infinite, then there are an~$x$ and a~$y$ with $y \in W_e^{[x]}$ and $\set{x,y}$ not homogeneous for $G$.  If $W_e$ receives attention, then there is a length-$3$ path between an~$x$ and a~$y$ with $y \in W_e^{[x]}$, in which case $\set{x,y}$ is not homogeneous for $G$.  Thus it suffices to show that if $W_e^{[x]}$ is infinite for infinitely many~$x$, then $W_e$ requires attention at some stage.

Suppose that $W_e^{[x]}$ is infinite for infinitely many $x$, and suppose for a contradiction that $W_e^{[x]}$ never requires attention.  Let $s_0$ be a stage by which every $W_i$ for $i < e$ that ever requires attention has received attention.  The graph contains only finitely many edges at each stage, so let $x_0$ be an upper bound for the vertices that are connected to the vertices $\leq e$ at stage $s_0$.  Notice that when some $W_i$ receives attention, the vertices connected at that stage are not connected to vertices $\leq i$.  Therefore once all the $W_i$ for $i < e$ that ever require attention have received attention, no vertex that is not connected to a vertex $\leq e$ is ever connected to a vertex $\leq e$.  In particular, no vertex $\geq x_0$ is ever connected to a vertex $\leq e$.  Now let $x > x_0$ be such that $W_e^{[x]}$ is infinite, and let $s_1 > s_0$ be a stage by which every $W_i$ for $i < x$ that ever requires attention has received attention.  Let $y_0$ be an upper bound for the vertices that are connected to $x$ and the vertices $\leq e$ at stage $s_1$, and again note that no vertex $\geq y_0$ is ever connected to $x$ or a vertex $\leq e$.  As $W_e^{[x]}$ is infinite, let $s > s_1$ be a stage at which there is a $y > y_0$ with $x < y < s$ and $y \in W_{e,s}^{[x]}$.  This $y$ is not connected to $x$, and neither $x$ nor $y$ is connected to a vertex $\leq e$, so $W_e$ requires attention at stage $s$, a contradiction.
\end{proof}

\begin{proposition}\label{prop-RCADoesNotProveRCOLOR2}
$\rca \nvdash \rcolor_2$.
\end{proposition}
\begin{proof}
Consider the $\omega$-model of $\rca$ whose second-order part consists of exactly the recursive sets.  The graph $G$ from Lemma~\ref{lem-RCOLOR2notCE} is in the model because $G$ is recursive.  However, the model contains no homogeneous set for $G$ because if $H$ were an infinite, recursive, homogeneous set, then $\{\la x, y \ra : x, y \in H\}$ would be a recursive, column-wise homogeneous set, thus contradicting Lemma~\ref{lem-RCOLOR2notCE}.
\end{proof}

The notion of \emph{restricted $\Pi^1_2$ conservativity} helps separate Ramsey-type weak K\"{o}nig's lemma and the Ramsey-type coloring principles from the following weak principles.

\begin{itemize}
\item $\coh$ (\emph{cohesiveness}; see Definition~\ref{def-COH}).
\item $\crt^2_2$ (\emph{cohesive Ramsey's theorem for pairs and two colors}; see~\cite{hirschfeldt2007combinatorial} for the definition).
\item $\cads$ (\emph{cohesive ascending or descending sequence}; see~\cite{hirschfeldt2007combinatorial} for the definition).
\item $\pizog$ (\emph{$\Pi^0_1$-generic}; see~\cite{hirschfeldt2009atomic} for the definition).
\item $\amt$ (\emph{atomic model theorem}; see~\cite{hirschfeldt2009atomic} for the definition).
\item $\opt$  (\emph{omitting partial types}; see~\cite{hirschfeldt2009atomic} for the definition).
\item $\fip$ (\emph{finite intersection principle}; see~\cite{Dzhafarov2011} for the definition).
\item $\ndtip$ (\emph{$\bar{D}_2$ intersection principle}; see~\cite{Dzhafarov2011} for the definition).
\end{itemize}

\begin{definition}[see \cite{hirschfeldt2007combinatorial,hirschfeldt2009atomic}]{\ }
\begin{itemize}
\item A sentence is \emph{restricted $\Pi^1_2$} if it is of the form $\forall A(\Theta(A) \imp \exists B(\Phi(A,B)))$, where $\Theta$ is arithmetic
and $\Phi$ is $\Sigma^0_3$.

\item A theory $T$ is \emph{restricted $\Pi^1_2$ conservative} over a theory $S$ if $S \vdash \varphi$ whenever $T \vdash \varphi$ and $\varphi$ is restricted $\Pi^1_2$.
\end{itemize}
\end{definition}

\begin{theorem}\label{thm-rPi12conservative}{\ }
\begin{itemize}
\item \textup{(\cite{hirschfeldt2007combinatorial})} $\rca + \coh$ is restricted $\Pi^1_2$ conservative over $\rca$.
\item \textup{(\cite{hirschfeldt2009atomic})} $\rca + \pizog$ is restricted $\Pi^1_2$ conservative over $\rca$.
\end{itemize}
\end{theorem}

$\rcolor_2$ is a restricted $\Pi^1_2$ sentence, so we immediately have that neither $\coh$ nor $\pizog$ implies $\rcolor_2$ over $\rca$.  Consequently, over $\rca$, the following principles are all incomparable with $\rwkl$ and with $\rcolor_2$:  $\coh$, $\crt^2_2$, $\cads$, $\pizog$, $\amt$, $\opt$, $\fip$, and $\ndtip$.

\begin{theorem}
$\rwkl$ is incomparable with each of $\coh$, $\crt^2_2$, $\cads$, $\pizog$, $\amt$, $\opt$, $\fip$, and $\ndtip$ over $\rca$.  $\rcolor_2$ is incomparable with these principles over $\rca$ as well.
\end{theorem}
\begin{proof}
Over $\rca$, we have the implications $\coh \imp \crt^2_2 \imp \cads$~\cite{cholak2001strength,hirschfeldt2007combinatorial}, $\pizog \imp \amt \imp \opt$~\cite{hirschfeldt2009atomic}, and $\pizog \imp \fip \imp \ndtip \imp \opt$~\cite{Dzhafarov2011}.  Thus we need only show that neither $\rca + \coh$ nor $\rca + \pizog$ prove $\rcolor_2$ and that $\rca + \rwkl$ proves neither $\cads$ nor $\opt$.  Observe that $\rcolor_2$ is a restricted $\Pi^1_2$ sentence, so we have that neither $\rca + \coh$ nor $\rca + \pizog$ proves $\rcolor_2$ by Proposition~\ref{prop-RCADoesNotProveRCOLOR2} and Theorem~\ref{thm-rPi12conservative}.  $\rca + \rwkl$ proves neither $\cads$ nor $\opt$ because $\rca + \wkl$ proves $\rca + \rwkl$ and $\rca + \wkl$ proves neither $\cads$~\cite{hirschfeldt2007combinatorial} nor $\opt$~\cite{hirschfeldt2009atomic}.
\end{proof}

We now adapt the proof that $\rca + \cac \nvdash \dnr$ in~\cite{hirschfeldt2007combinatorial} to prove that $\rca + \cac \nvdash \rcolor_2$.  We build an $\omega$-model of $\rca + \scac + \coh$ that is not a model of $\rcolor_2$ by alternating between adding chains or antichains to stable partial orders and adding cohesive sets without ever adding an infinite set homogeneous for the graph from Lemma~\ref{lem-RCOLOR2notCE}.

\begin{lemma}\label{lem-SCACForcing}
Let $X$ be a set, let $G = (V, E)$ be a graph recursive in $X$ such that no column-wise homogeneous set for $G$ is r.e.\ in $X$, and let $P = (P, \leq_P)$ be an infinite, stable partial order recursive in $X$.  Then there is an infinite $C \subseteq P$ that is either a chain or an antichain such that no column-wise homogeneous set for $G$ is r.e.\ in $X \oplus C$.
\end{lemma}

\begin{proof}
For simplicity, assume that $X$ is recursive.  The proof relativizes to non-recursive $X$.  As $P$ is stable, assume for the sake of argument that $P$ satisfies $(\forall i \in P)(\exists s)[(\forall j > s)(j \in P \imp i \leq_P j) \orr (\forall j > s)(j \in P \imp i \mid_P j)]$.  The case with $\geq_P$ in place of $\leq_P$ is symmetric.  Also assume that there is no recursive, infinite antichain $C \subseteq P$, for otherwise we are done.

Let $U = \set{i \in P : (\exists s)(\forall j > s)(j \in P \imp i \leq_P j)}$.  The fact that there is no recursive, infinite antichain in $P$ implies that $U$ is infinite.  Let $F = (F, \sqsubseteq)$ be the partial order consisting of all $\sigma \in U^{<\omega}$ that are increasing in both $<$ and $\leq_P$, where $\tau \sqsubseteq \sigma$ if $\tau \succeq \sigma$.  Let $H$ be sufficiently generic for $F$, and notice that $H$ (or rather, the range of $H$, which is computable from $H$ as $H$ is increasing in $<$) is an infinite chain in $P$.  Suppose for a contradiction that $W_e^H$ is column-wise homogeneous for $G$.  Fix a $\sigma \preceq H$ such that
\begin{align*}
\sigma \Vdash \forall x \forall y (y \in (W_e^H)^{[x]} \imp \text{$\set{x,y}$ is homogeneous for $G$}).
\end{align*}
Define a partial computable function $\tau \colon \omega^2 \imp P^{<\omega}$ by letting $\tau(x,i) \in P^{<\omega}$ be the string with the least code such that $\tau(x,i) \supseteq \sigma$, that $\tau(x,i)$ is increasing in both $<$ and $\leq_P$, and that $|(W_e^{\tau(x,i)})^{[x]}| > i$.  From here there are two cases.

Case~1: There are infinitely many pairs $\tuple{x,i}$ such that $\tau(x,i)$ is defined and there is a $y \in (W_e^{\tau(x,i)})^{[x]}$ with $\set{x,y}$ not homogeneous for $G$.  The last element of such a $\tau(x,i)$ is in $P \setminus U$ because otherwise $\tau(x,i) \in F$ and $\tau(x,i) \preceq \sigma$, contradicting that $\sigma \Vdash \forall x \forall y (y \in (W_e^H)^{[x]} \imp \text{$\set{x,y}$ is homogeneous for $G$})$.  Thus the set $C$ consisting of the last elements of such strings $\tau(x,i)$ is an infinite r.e.\ subset of $P \setminus U$.  As elements $i$ of $P \setminus U$ have the property $(\exists s)(\forall j > s)(j \in P \imp i \mid_P j)$, we can thin $C$ to an infinite r.e.\ antichain in $P$ and hence to an infinite recursive antichain in $P$, a contradiction.

Case~2: There are finitely many pairs $\tuple{x,i}$ such that $\tau(x,i)$ is defined and there is a $y \in (W_e^{\tau(x,i)})^{[x]}$ with $\set{x,y}$ not homogeneous for $G$.  In this case, let $x_0$ be such that if $x > x_0$ and $\tau(x,i)$ is defined, then $(\forall y \in (W_e^{\tau(x,i)})^{[x]})(\text{$\set{x,y}$ is homogeneous for $G$})$.  Notice that if $|(W_e^H)^{[x]}| > i$, then there is a $\tau$ with $\sigma \preceq \tau \preceq H$ such that $|(W_e^\tau)^{[x]}| > i$.  Hence if $(W_e^H)^{[x]}$ is infinite, then $\tau(x,i)$ is defined for all $i$.  Thus let 
\begin{align*}
W = \set{\tuple{x, \max(W_e^{\tau(x,i)})^{[x]}} : x > x_0 \andd i \in \omega \andd \text{$\tau(x,i)$ is defined}}.
\end{align*}
Then $W$ is an r.e.\ set that is column-wise homogeneous for $G$, a contradiction.

Thus there is no column-wise homogeneous set for $G$ that is r.e.\ in $H$.  Therefore (the range of) $H$ is our desired chain $C$.
\end{proof}

\begin{lemma}\label{lem-COHForcing}
Let $X$ be a set, let $G = (V, E)$ be a graph recursive in $X$ such that no column-wise homogeneous set for $G$ is r.e.\ in $X$, and let $\vec R = (R_i)_{i \in \omega}$ be a sequence of sets uniformly recursive in $X$.  Then there is an infinite set $C$ that is cohesive for $\vec R$ such that no column-wise homogeneous set for $G$ is r.e.\ in $X \oplus C$.
\end{lemma}

\begin{proof}
For simplicity, assume that $X$ is recursive.  The proof relativizes to non-recursive $X$.

We force with recursive Mathias conditions $(D,L)$, where $D \subseteq \omega$ is finite, $L \subseteq \omega$ is infinite and recursive, and every element of $D$ is less than every element of $L$.  The order is $(D_1, L_1) \sqsubseteq (D_0, L_0)$ if $D_0 \subseteq D_1$, $L_1 \subseteq L_0$, and $D_1 \setminus D_0 \subseteq L_0$.  Let $H$ be sufficiently generic.  Then $H$ is an infinite cohesive set for $\vec R$ (as in, for example, Section~4 of~\cite{cholak2001strength}).

Suppose for a contradiction that $W_e^H$ is column-wise homogeneous for $G$.  Let $(D,L)$ be a condition such that $D \subseteq H \subseteq L$ and
\begin{align*}
(D,L) \Vdash \forall x \forall y (y \in (W_e^H)^{[x]} \imp \text{$\set{x,y}$ is homogeneous for $G$}).
\end{align*}
Let
\begin{align*}
W = \set{\tuple{x,y} : \exists E(\text{$E$ is finite} \andd D \subseteq E \subseteq L \andd \tuple{x,y} \in W_e^E)}.
\end{align*}
$W$ is an r.e.\ set, and $\forall x \forall y (y \in W^{[x]} \imp \text{$\set{x,y}$ is homogeneous for $G$}$).  To see the second statement, suppose there is a $\tuple{x,y} \in W$ such that $\set{x,y}$ is not homogeneous for $G$, and let $E$ witness $\tuple{x,y} \in W$.  Then $(E, L \setminus E) \preceq (D,L)$, but $(E, L \setminus E) \Vdash (y \in (W_e^H)^{[x]} \andd \text{$\set{x,y}$ is not homogeneous for $G$})$, a contradiction.  Finally, $W \supseteq W_e^H$ because if $\tuple{x,y} \in W_e^H$, then there is a finite $E$ with $D \subseteq E \subseteq L$ such that $\tuple{x,y} \in W_e^E$, in which case $\tuple{x,y} \in W$.  Thus $W$ is an r.e.\ set that is column-wise homogeneous for $G$.  This contradicts the lemma's hypothesis.  Therefore no column-wise homogeneous set for $G$ is r.e.\ in $H$, so $H$ is the desired cohesive set.
\end{proof}

\begin{theorem}\label{cac-not-imply-rcolor2}
$\rca + \cac \nvdash \rcolor_2$
\end{theorem}
\begin{proof}
Iterate and dovetail applications of Lemma~\ref{lem-SCACForcing} and Lemma~\ref{lem-COHForcing} to build a collection of sets $\Scal$ such that $(\omega, \Scal) \vDash \rca + \scac + \coh$, the graph $G$ from Lemma~\ref{lem-RCOLOR2notCE} is in $\Scal$, and no set that is r.e.\ in any set in $\Scal$ is column-wise homogeneous for $G$.  Then $(\omega, \Scal) \vDash \cac$ by \cite{hirschfeldt2007combinatorial}, and $(\omega, \Scal) \nvDash \rcolor_2$ by the same argument as in Proposition~\ref{prop-RCADoesNotProveRCOLOR2}.
\end{proof}

We conclude by proving that $\rca + \dnr \nvdash \rwkl$, thereby answering Question~9 of~\cite{Flood2012}.  In fact, we prove the stronger result $\rca + \wwkl \nvdash \rcolor_2$.  This is accomplished by building a recursive bipartite graph $G$ such that the measure of the set of oracles that compute homogeneous sets for $G$ is $0$.  It follows that there is a Martin-L\"of random $X$ that does not compute a homogenous set for $G$, and a model of $\rca + \wwkl + \neg \rcolor_2$ is then easily built from the columns of~$X$.

Recall that, in the context of a bipartite graph $G = (V,E)$, a set $H \subseteq V$ is $2$-homogeneous for $G$ if no two vertices in $H$ are connected by an odd-length path in $G$.  Here we simply say that such an $H$ is \emph{$G$-homogeneous} (or just \emph{homogeneous}).  Likewise, if $H \subseteq V$ contains two vertices that are connected by an odd-length path in $G$, then $H$ is \emph{$G$-inhomogeneous} (or just \emph{inhomogeneous}).

\begin{theorem}\label{RCOLOR2hasNRA}
There is a recursive bipartite graph $G = (\omega, E)$ such that the measure of the set of oracles that enumerate homogeneous sets for $G$ is $0$. 
\end{theorem}

\begin{proof}
By Lebesgue density considerations (see, for example, \cite{nies2009}~Theorem~1.9.4), if a positive measure of oracles enumerate infinite homogeneous sets for a graph $G$, then
\begin{align*}
(\forall \epsilon > 0)(\exists e)[\mu\{X : \text{$W_e^X$ is infinite and $G$-homogeneous}\} > 1-\epsilon].
\end{align*}
Thus it suffices to build $G$ to satisfy the following requirement $R_e$ for each $e \in \omega$:
\begin{align*}
R_e:\; \mu\{X : \text{$W_e^X$ is infinite and $G$-homogeneous}\} \leq 0.9.
\end{align*}

Let us first give a rough outline of the construction.  Observe our construction must necessarily produce a graph $G$ that does not contain an infinite connected component.  If $G$ has an infinite connected component, then that component contains a vertex $v$ such that infinitely many vertices are connected to $v$ by an even-length path.  These vertices that are at an even distance from $v$ can be effectively enumerated, and they form a homogeneous set.  Thus our graph $G$ must be a union of countably many finite connected components.  Each stage of the construction adds at most finitely many edges, and thus at each stage of the construction all but finitely many vertices are isolated. For each~$e$, our plan is the following.  We monitor the action of $W_e^X$ for all oracles~$X$ until we see a sufficient measure of $X$'s produce enough vertices (in a sense to made precise). Then, the idea is to satisfy $R_e$ by adding edges to these vertices in a way that defeats about half (in the measure-theoretic sense) of the oracles~$X$. This is done by a two-step process. Requirement $R_e$ acts by either type~I or type~II actions, the second type following the first type.  In a type~I action, $R_e$ locks some finite number of vertices, thereby preventing lower priority requirements from adding edges to these locked vertices.  In a type~II action, $R_e$ merges finitely many of $G$'s connected components into one connected component by adding some new edges while maintaining that $G$ is a bipartite graph. This merging is made in a way which ensures that for a sufficient measure of oracles~$X$, $W_e^X$ is inhomogeneous for the resulting graph.

We now present the construction in full detail.  At stage $s$, we say that
\begin{itemize}
\item \emph{$R_e$ requires type I attention} if $R_e$ has no vertices locked and there are strings of length $s$ witnessing that
\begin{align*}
\mu\{X : (\exists x \in W_{e,s}^X)(\text{$x$ is not connected to any $v$ locked by $R_k$ for any $k < e$})\} > 0.9;
\end{align*}

\item \emph{$R_e$ requires type II attention} if it currently has locked vertices due to a type~I action, has never acted according to type II, and there are strings of length $s$ witnessing that
\begin{align*}
\mu\{X : (\exists y \in W_{e,s}^X)(\text{$y$ is not connected to any $v$ locked by $R_k$ for any $k \leq e$})\} > 0.9;
\end{align*}

\item \emph{$R_e$ requires attention} if $R_e$ requires type I attention or requires type II attention.
\end{itemize}

At stage $0$, $E = \emptyset$, and no requirement has locked any vertices.

At stage $s+1$, let $e < s$ be least such that $R_e$ requires attention (if there is no such $e$, then go on to the next stage).  If $R_e$ requires type I attention, let $x_0, x_1, \dots, x_{n-1}$ be vertices that are not connected to any $v$ locked by $R_k$ for any $k < e$ and such that the strings of length $s$ witness that $\mu\{X : (\exists i < n)(x_i \in W_{e,s}^X)\} > 0.9$.  $R_e$ locks the vertices $x_0, x_1, \dots, x_{n-1}$.  All requirements $R_k$ for $k > e$ unlock all of their vertices.

If $R_e$ requires type II attention, let $y_0, y_1, \dots, y_{m-1}$ be vertices that are not connected to any $v$ locked by $R_k$ for any $k \leq e$ and such that the strings of length $s$ witness that $\mu\{X : (\exists j < m)(y_j \in W_{e,s}^X)\} > 0.9$.  Let $x_0, x_1, \dots, x_{n-1}$ be the vertices that are locked by $R_e$.  First we merge the connected components of the $x_i$'s into a single connected component and the connected components of the $y_j$'s into a single connected component.  To do this, let $a, b, c, d > s$ be fresh vertices, and add the edges $(a,b)$ and $(c,d)$.  The graph is currently bipartite, so for each $i < n$ add either the edge $(x_i, a)$ or $(x_i, b)$ so as to maintain a bipartite graph.  Similarly, merge the connected components of the $y_j$'s by adding either the edge $(y_j, c)$ or $(y_j, d)$ for each $j < m$.  The component of the $x_i$'s is disjoint from the component of the $y_j$'s because the $y_j$'s were chosen not to be connected to the $x_i$'s.  Thus both the graph $G_1$ obtained by adding the edge $(a,c)$ and the graph $G_2$ obtained by adding the edge $(a,d)$ are bipartite.  Each pair $\{x_i, y_j\}$ is homogeneous for exactly one of $G_1$ and $G_2$, and the strings of length $s$ witness that
\begin{align*}
\mu\{X : (\exists i < n)(\exists j < m)(x_i \in W_{e,s}^X \andd y_j \in W_{e,s}^X)\} > 0.8
\end{align*}
and therefore that
\begin{align*}
\mu\{X : \text{$W_{e,s}^X$ is either $G_1$-inhomogeneous or $G_2$ inhomogeneous}\} > 0.8.
\end{align*}
Thus the strings of length $s$ either witness that
\begin{align*}
\mu\{X : \text{$W_{e,s}^X$ is $G_1$-inhomogeneous}\} > 0.4,
\end{align*}
in which case we extend to $G_1$ by adding the edge $(a,c)$,
or that
\begin{align*}
\mu\{X : \text{$W_{e,s}^X$ is $G_2$-inhomogeneous}\} > 0.4,
\end{align*}
in which case we extend to $G_2$ by adding the edge $(a,d)$.  This completes the construction.

To verify the construction, we first notice that $G$ is bipartite because it is bipartite at every stage.  Furthermore, $G$ is recursive because if an edge $(u,v)$ is added at stage $s$, either $u > s$ or $v > s$.  Thus to check whether an edge $(u,v)$ is in $G$, it suffices to check whether the edge has been added by stage $\max(u,v)$. 

We now verify that every requirement is satisfied.  Suppose that $R_e$ acts according to type II at some stage $s+1$.  Then $R_e$ is satisfied because we have ensured that 
\begin{align*}
 \mu \{X : W_e^X~  \text{is $G$-inhomogeneous}\} > 0.4 
\end{align*}
and thus that
\begin{align*}
  \mu \{X : W_e^X~  \text{is $G$-homogeneous}\} \leq 0.6.
\end{align*}

We prove by induction that, for every $e \in \omega$, $R_e$ is satisfied and there is a stage past which $R_e$ never requires attention.  Consider $R_e$.  If $\mu\{X : \text{$W_e^X$ is infinite}\} \leq 0.9$, then $R_e$ is satisfied and $R_e$ never requires attention.  So assume that $\mu\{X : \text{$W_e^X$ is infinite}\} > 0.9$.  By induction, let $s_0$ be a stage such that no $R_k$ for $k < e$ ever requires attention at a stage past $s_0$.  If $R_e$ has locked vertices at stage $s_0$, then these vertices remain locked at all later stages because no higher priority $R_k$ ever unlocks them.  If $R_e$ does not have locked vertices at stage $s_0$, then let $s_1 \geq s_0$ be least such that the strings of length $s_1$ witness that $R_e$ requires type~I attention.  Such an $s_1$ exists because $\mu\{X : \text{$W_e^X$ is infinite}\} > 0.9$ and because the finite set of vertices that are connected to vertices locked by the $R_k$ for $k < e$ have stabilized by stage $s_0$.  $R_e$ then requires and receives type I attention at stage $s_1$, and the vertices that $R_e$ locks at stage $s_1$ are never later unlocked.  So there is a stage $s_1 \geq s_0$ by which $R_e$ has locked a set of vertices that are never unlocked.  If $R_e$ has acted according to type II by stage $s_1$, then $R_e$ is satisfied and never requires attention past stage $s_1$.  If $R_e$ has not acted according to type II by stage $s_1$, let $s_2 \geq s_1$ be least such that the strings of length $s_2$ witness that $R_e$ requires type~II attention.  Such an $s_2$ exists because $\mu\{X : \text{$W_e^X$ is infinite}\} > 0.9$ and because, past stage $s_1$, no requirement except $R_e$ can act to connect a vertex to a vertex locked by an $R_k$ for a $k \leq e$.  $R_e$ then requires and receives type~II attention at stage $s_2$.  Hence $R_e$ is satisfied, and $R_e$ never requires attention at a later stage.  This completes the proof.
\end{proof}

\begin{theorem}\label{thm-WWKLDoesNotProveRCOLOR2}
$\rca + \wwkl \nvdash \rcolor_2$.
\end{theorem}

\begin{proof}
Let $G$ be the recursive graph from Theorem~\ref{RCOLOR2hasNRA}.  There are measure $1$ many Martin-L\"of random sets, but only measure $0$ many sets compute homogeneous sets for $G$.  Thus let $X$ be a Martin-L\"of random set that does not compute a homogeneous set for $G$, and let $\mathfrak{M}$ be the structure whose first-order part is $\omega$ and whose second-order part is $\{Y : \exists k (Y \leq_T \bigoplus_{i<k}X^{[i]})\}$.  It is well-known that $\mathfrak{M} \vDash \rca + \wwkl$, which one may see by appealing to van Lambalgen's theorem (see \cite{DH10}~Section 6.9) and the equivalence between $\wwkl$ and $\ran{1}$.  Moreover, $\mathfrak{M} \nvDash \rcolor_2$ because $\mathfrak{M}$ contains the bipartite graph $G$, but it does not contain any homogeneous set for $G$.
\end{proof}

It now follows that $\rca + \dnr \nvdash \rwkl$.  This has been proved independently by Flood and Towsner~\cite{Flood2014Separating} using
the techniques introduced by Lerman, Solomon, and Towsner~\cite{lerman2013separating}.  Recently, Patey~\cite{PateyRcolor} enhanced the separation of $\dnr$ and $\rwkl$ by proving that for every recursive order $h$, there is an $\omega$-model of the statement ``for every $X$ there is a function that is $\dnrf$ relative to $X$ and bounded by $h$'' that is not a model of $\rcolor_2$.  This answers a question in~\cite{Flood2014Separating}.

\begin{corollary}\label{cor:DNRNoImpRWKL}
$\rca + \dnr \nvdash \rwkl$.
\end{corollary}

\begin{proof}
This follows from Theorem~\ref{thm-WWKLDoesNotProveRCOLOR2} because $\rca \vdash \wwkl \imp \dnr$ and $\rca \vdash \rwkl \imp \rcolor_2$.
\end{proof}

\section{Summary and open questions}

In this section, we briefly recall the remaining open questions surrounding the Ramsey-type combinatorial principles.

By Avigad, Dean, and Rute~\cite{AvigadDeanRute}, $\rca + \wwkls{2} \vdash \bsig^0_2$, but by Slaman~\cite{slaman2011first}, $\rca + \ran{2} \nvdash \bsig^0_2$.  Thus we ask whether or not $\rca + \rwwkls{2}$ proves $\bsig^0_2$.

\begin{question*}[\ref{qu-2RWWKLvsBSig2}]
Does $\rca + \rwwkls{2} \vdash \bsig^0_2$?
\end{question*}

We readily see that $\rca \vdash \forall k(\srt^2_k \imp \rwkl_k)$ and therefore that $\rca \vdash \forall k \srt^2_k \imp \forall k \rwkl_k$.  However, the use of $\forall k\srt^2_k$ may not be strictly necessary.

\begin{question*}[\ref{qu-SRT22vsRWKLk}]
Does $\rca \vdash \srt^2_2 \imp \forall k \rwkl_k$?
\end{question*}

We proved that the Ramsey-type graph $k$-coloring problems are equivalent to $\rwkl$ over $\rca$ for all $k \in \omega$ with $k \geq 3$ (Corollary~\ref{cor:RCOLORkisRWKL}).  However, we do not know if the $k=2$ case has the same strength as the $k \geq 3$ cases.

\begin{question*}[\ref{qu-RCOLOR2vsRWKL}]
Does $\rca \vdash \rcolor_2 \imp \rwkl$?  
\end{question*}

By Theorem~\ref{thm-WWKLDoesNotProveRCOLOR2}, there is an $\omega$-model of~$\dnr$ (and even of~$\wwkl$) which is not a model of $\rcolor_2$.  Therefore $\dnr$ does not imply $\rcolor_2$ over $\rca$.  However, we are unable to determine whether or not the converse holds.  The combinatorics of $\rcolor_2$ differ enough from the combinatorics of $\rwkl$ so that it is not possible to directly adapt Flood's proof that $\rca \vdash \rwkl \imp \dnr$ to a proof that $\rca \vdash \rcolor_2 \imp \dnr$.

\begin{question*}[\ref{qu-RCOLOR2vsDNR}]
Does $\rca \vdash \rcolor_2 \imp \dnr$?
\end{question*}

Of course, a negative answer to Question~\ref{qu-RCOLOR2vsDNR} would also provide a negative answer to Question~\ref{qu-RCOLOR2vsRWKL}.\\

The following `before/after' diagrams summarize the progress made in this paper towards the development of the reverse mathematics zoo below $\rt^2_2$. Double arrows indicate strict implications, single arrows indicate implications not known to be strict, and dotted arrows indicate non-implications.  All implications and non-implications are over $\rca$.

\usetikzlibrary{arrows}
\usetikzlibrary{decorations.markings}

\begin{center}
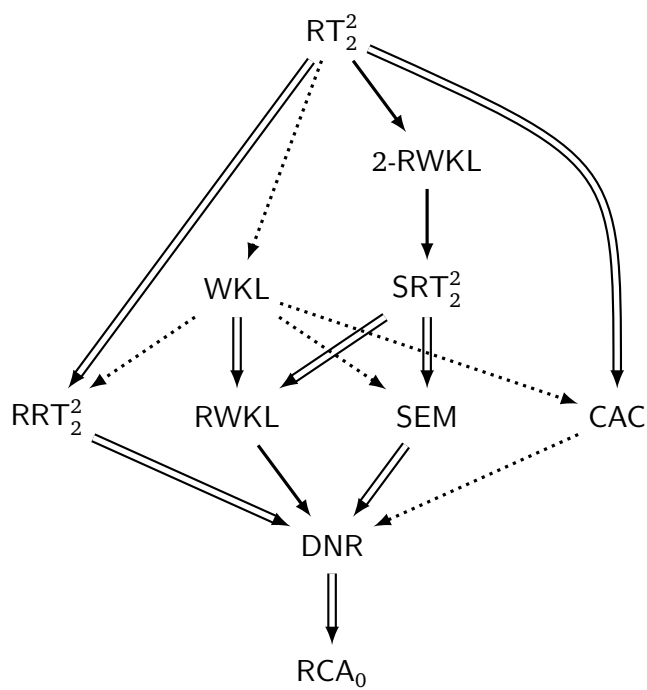

\begin{tikzpicture}[x=2.5cm, y=1.7cm, 
	node/.style={minimum size=2em},
	impl/.style={draw,very thick,-latex},
	strict/.style={draw, thick, -latex, double distance=2pt},
	nonimpl/.style={draw, very thick, dotted, -latex}]

	\node[node] (rt22) at (1.5,5) {$\rt^2_2$};
	\node[node] (rwklzp) at (2, 4)  {$\rwkls{2}$};
	\node[node] (wkl) at (1, 3) {$\wkl$};
	\node[node] (srt22) at (2, 3) {$\srt^2_2$};
	\node[node] (rrt22) at (0, 2) {$\rrt^2_2$};
	\node[node] (rwkl) at (1, 2) {$\rwkl$};
	\node[node] (sem) at (2, 2) {$\semo$};
	\node[node] (cac) at (3, 2) {$\cac$};
	\node[node] (dnr) at (1.5, 1) {$\dnr$};
	\node[node] (rca) at (1.5, 0) {$\rca$};
	
	\draw[impl] (rt22) -- (rwklzp);
	\draw[impl] (rwklzp) -- (srt22);
	\draw[impl] (rwkl) -- (dnr);

	\draw[strict] (rt22) -- (rrt22);
	\draw[strict] (rrt22) -- (dnr);
	\draw[strict] (dnr) -- (rca);
	\draw[strict] (sem) -- (dnr);
	\draw[strict] (wkl) -- (rwkl);
	\draw[strict] (srt22) -- (rwkl);
	\draw[strict] (srt22) -- (sem);
	\draw[strict] (rt22) .. controls (3,4) ..  (cac);
	
	\draw[nonimpl] (rt22) -- (wkl);
	\draw[nonimpl] (wkl) -- (rrt22);
	\draw[nonimpl] (wkl) -- (cac);
	\draw[nonimpl] (wkl) -- (sem);
	\draw[nonimpl] (cac) -- (dnr);
\end{tikzpicture}
\captionof{figure}{Local zoo before.}
\end{center}


\begin{center}
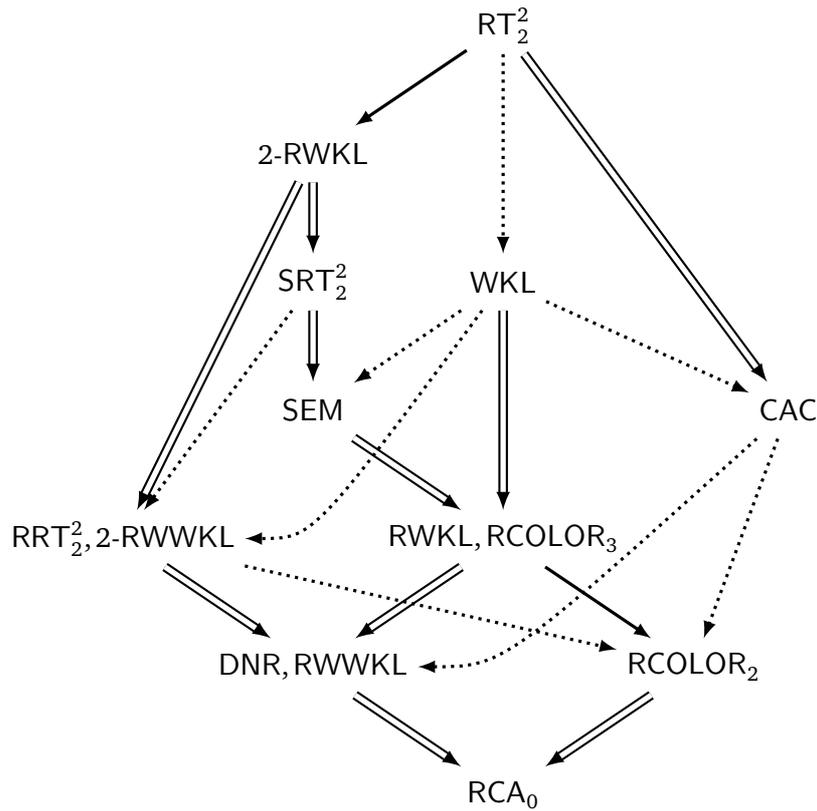

\begin{tikzpicture}[x=2.5cm, y=1.7cm, 
	node/.style={minimum size=2em},
	impl/.style={draw,very thick,-latex},
	strict/.style={draw, thick, -latex, double distance=2pt},
	nonimpl/.style={draw, very thick, dotted, -latex}]

	\node[node] (rt22) at (2, 6) {$\rt^2_2$};
	\node[node] (rwklzp) at (1, 5)  {$\rwkls{2}$};
	\node[node] (wkl) at (2, 4) {$\wkl$};
	\node[node] (srt22) at (1, 4) {$\srt^2_2$};
	\node[node] (rrt22) at (0, 2) {$\rrt^2_2, \rwwkls{2}$};
	\node[node] (rwkl) at (2, 2) {$\rwkl, \rcolor_3$};
	\node[node] (sem) at (1, 3) {$\semo$};
	\node[node] (cac) at (3.5, 3) {$\cac$};
	\node[node] (dnr) at (1, 1) {$\dnr, \rwwkl$};
	\node[node] (rca) at (2, 0) {$\rca$};
	\node[node] (rcolor2) at (3,1) {$\rcolor_2$};
	
	\draw[impl] (rt22) -- (rwklzp);
	\draw[impl] (rwkl) -- (rcolor2);

	\draw[strict] (rwklzp) -- (srt22);
	\draw[strict] (rwkl) -- (dnr);
	\draw[strict] (rwklzp) -- (rrt22);
	\draw[strict] (rrt22) -- (dnr);
	\draw[strict] (dnr) -- (rca);
	\draw[strict] (sem) -- (rwkl);
	\draw[strict] (wkl) -- (rwkl);
	\draw[strict] (srt22) -- (sem);
	\draw[strict] (rt22) --  (cac);
	\draw[strict] (rcolor2) -- (rca);
	
	\draw[nonimpl] (rt22) -- (wkl);
	\draw[nonimpl] (wkl) .. controls (1,2) ..  (rrt22);
	\draw[nonimpl] (wkl) -- (cac);
	\draw[nonimpl] (wkl) -- (sem);
	\draw[nonimpl] (cac) .. controls (2,1) ..  (dnr);
	\draw[nonimpl] (cac) -- (rcolor2);
	\draw[nonimpl] (srt22) -- (rrt22);
	\draw[nonimpl] (rrt22) -- (rcolor2);
\end{tikzpicture}
\captionof{figure}{Local zoo after.}
\end{center}

\section*{Acknowledgments}
We thank Fran\c{c}ois Dorais, Emanuele Frittaion, and our anonymous reviewer for helpful comments on the drafts of this work.

\bibliography{bibliography}
\bibliographystyle{plain}

\end{document}